\documentclass[letterpaper]{amsart}

\usepackage[letterpaper]{geometry}
\usepackage[mathscr]{euscript}
\usepackage{amsmath, amsthm, amsfonts, amsbsy, thmtools, amssymb, hyperref, stmaryrd, cleveref, tikz, pgfplots}
\usetikzlibrary{arrows}

\usepgfplotslibrary{polar}
\usepgfplotslibrary{fillbetween}
\usepgflibrary{shapes.geometric}
\usetikzlibrary{calc}

\numberwithin{equation}{section}

\newtheorem{thm}{Theorem}[section]
\newtheorem{prop}[thm]{Proposition}
\newtheorem{lem}[thm]{Lemma}
\newtheorem{cor}[thm]{Corollary}

\theoremstyle{remark}
\newtheorem{rem}[thm]{Remark}
\theoremstyle{definition}
\newtheorem{definition}[thm]{Definition}

\title{Nonexistence and Uniqueness for Pure States of Ferroelectric Six-Vertex Models}

\author{Amol Aggarwal} 

\begin{document}

	\begin{abstract}
		
		In this paper we consider the existence and uniqueness of pure states with some fixed slope $(s, t) \in [0, 1]^2$ for a general ferroelectric six-vertex model. First, we show there is an open subset $\mathfrak{H} \subset [0, 1]^2$, which is parameterized by the region between two explicit hyperbolas, such that there is no pure state for the ferroelectric six-vertex model of any slope $(s, t) \in \mathfrak{H}$. Second, we show that there is a unique pure state for this model of any slope $(s, t)$ on the boundary $\partial \mathfrak{H}$ of $\mathfrak{H}$. These results confirm predictions of Bukman--Shore from 1995. 
	\end{abstract}

	\maketitle
	
	\tableofcontents 
	
	\section{Introduction}   
	
	\label{Model}

	\subsection{Preface}
	
	\label{States}  

 	A fundamental question in mathematical physics concerns the classification and analysis of pure states (translation-invariant, ergodic Gibbs measures) for statistical mechanical systems. Over the past two decades, a deep understanding in this direction has been achieved for two-dimensional dimer models. Indeed, the classification aspect of this question was addressed in 2005 by Sheffield \cite{RS}, who showed that dimer pure states are parameterized by pairs of real numbers, also called slopes, that govern the average gradient for the height function associated with the model. These pure states were then analyzed in 2006 by Kenyon--Okounkov--Sheffield \cite{DA}, who showed that they come in three types. The first is \emph{frozen}, where the associated height function is essentially deterministic; the second is \emph{gaseous}, where the height variance is nonzero but bounded; and the third is \emph{liquid}, where the height fluctuations diverge logarithmically in the lattice size.
 	
 	In this paper we consider the existence and uniqueness of pure states for the six-vertex model. This model, whose definition we will recall more precisely in \Cref{Measures} below, involves six positive numbers $(a_1, a_2, b_1, b_2, c_1, c_2)$ that provide weights for local configurations. Associated with these weights is an \emph{anisotropy parameter} $\Delta = \Delta (a_1, a_2, b_1, b_2, c_1, c_2)$, defined by setting 
 	\begin{flalign}
 	\label{parameter1} 
 	\Delta = \displaystyle\frac{a_1 a_2 + b_1 b_2 - c_1 c_2}{2 \sqrt{a_1 a_2 b_1 b_2}},
 	\end{flalign}
 	
 	\noindent that is believed to dictate both qualitative and quantitative features of the model. 
 	
 	Here, we will be interested in the regime when $\Delta > 1$, which is known as the \emph{ferroelectric phase} of the six-vertex model. The description of pure states for this six-vertex model differs considerably from its counterpart in the dimer setting. Indeed, even the question of existence for pure states has different answers in these two contexts; in particular, we will see that there are certain slopes for the ferroelectric six-vertex model that are in principle allowable but do not admit pure states. 
 	
 	However, before explaining our results in more detail, let us first outline the predictions in this direction from the physics literature, due to Bukman--Shore \cite{TCPFSVM} in 1995. Associated with any pure state $\mu$ of a six-vertex model is a \emph{slope} $(s, t)$, where $s$ and $t$ denote the probabilities that a given vertical and horizontal edge is occupied under $\mu$, respectively; as such, we must have $(s, t) \in [0, 1]^2$. It is widely believed that each slope $(s, t) \in (0, 1)^2$ admits at most one pure state for any ferroelectric six-vertex model. In the context of gradient models governed by a simply attractive Gibbs potential, the stronger statement was shown in \cite{RS} that every slope $(s, t) \in (0, 1)^2$ admits a unique pure state. Yet, this result does not apply for the ferroelectric six-vertex model, which can be viewed as a discrete gradient model, but not under a potential that is simply attractive. 
 	
 	In fact, it was predicted in Section 3.5 of \cite{TCPFSVM} that there is a ``lens-shaped'' region $\mathfrak{H} \subset [0, 1]^2$ admitting no pure states for the ferroelectric six-vertex model; we refer to the shaded part of \Cref{st} for a depiction. A precise parameterization for this open set $\mathfrak{H}$ was also proposed there, by writing its boundary $\partial \mathfrak{H}$ as the union $\mathfrak{h}_1 \cup \mathfrak{h}_2$ of two explicit hyperbolas (given by \eqref{h} and \eqref{h1h2} below). In addition, \cite{TCPFSVM} proposed characterizations for how the pure state $\mu_{s, t}$ of slope $(s, t)$ should qualitatively behave in different regions of $[0, 1]^2$. See also Section 1.2.1 of the work \cite{SLSVM} by Corwin--Ghosal--Shen--Tsai for a restatement of these predictions, and also Section 4 of the work \cite{LSAFM} by de Gier--Kenyon--Watson where the first (nonexistence) part of the below characterization is predicted for the five-vertex model. In what follows, $\overline{\mathfrak{H}} = \mathfrak{H} \cup \partial \mathfrak{H}$ denotes the closure of $\mathfrak{H}$.

 	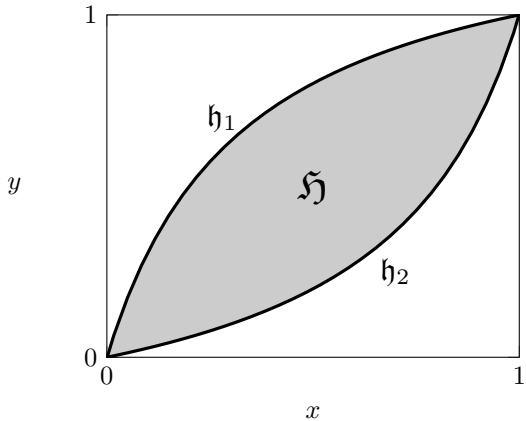
\begin{figure}
 		
 		\begin{center}
 			
 			\begin{tikzpicture}
 			\begin{axis}[
 			scale = .8,
 			samples= 300, 
 			xlabel = {$x$},
 			ylabel = {$y$},
 			ylabel style={rotate=270},
 			xmin=0, 
 			xmax=1, 
 			ymin=0, 
 			ymax=1, 
 			xtick distance=1,
 			ytick distance=1
 			]
 			\addplot[no marks, very thick, color=black,  name path=h1] {4*x/(3*x+1)};
 			\addplot[no marks, very thick, color=black,  name path=h2] {x/(4 - 3*x)};
 			\addplot[no marks, thick, color=black] coordinates {(.5, .5)} node[scale = 1.6]{$\mathfrak{H}$};
 			\addplot[no marks, thick, color=black] coordinates {(.28, .7)} node[scale = 1.2]{$\mathfrak{h}_1$};
 			\addplot[no marks, thick, color=black] coordinates {(.7, .25)} node[scale = 1.2]{$\mathfrak{h}_2$};
 			
 			\addplot[gray!40!white] fill between[of=h1 and h2];
 			\end{axis}

 			\end{tikzpicture}

 			\caption{\label{st} Depicted above is the region $\mathfrak{H}$ and its boundary $\partial \mathfrak{H} = \mathfrak{h}_1 \cup \mathfrak{h}_2$ in the phase diagram for the behavior of pure states in the ferroelectric six-vertex model.}
 			
 		\end{center}
 		
 	\end{figure}

 	\begin{enumerate}
 		\item \emph{Nonexistence}: If $(s, t) \in \mathfrak{H}$, then there should be no pure state $\mu_{s, t}$ of slope $(s, t)$. 
 		\item \emph{KPZ States}: If $(s, t) \in \partial \mathfrak{H}$, then $\mu_{s, t}$ should exhibit Kardar--Parisi--Zhang (KPZ) behavior.
 		\item \emph{Liquid States}: If $(s, t) \in (0, 1)^2 \setminus \overline{\mathfrak{H}}$, then $\mu_{s, t}$ should be liquid. 
 		\item \emph{Frozen States}: If $(s, t)$ is on the boundary of $[0, 1]^2$, then $\mu_{s, t}$ should be frozen.
 	\end{enumerate}
 
 	This description of pure states contrasts sharply with its counterpart in the dimer model. Indeed, while frozen and liquid states both appear in dimers, regions of nonexistence (that disconnect the space of liquid states) and KPZ states do not. Let us mention that, although \cite{TCPFSVM} did not precisely predict the sense in which the measures $\mu_{s, t}$ for $(s, t) \in \partial \mathfrak{H}$ should exhibit KPZ behavior, systems in the Kardar--Parisi--Zhang universality class are typically characterized by exhibiting fluctuations of order $N^{1 / 3}$ on a domain of size $N$ \cite{DSGI}. We refer to the surveys \cite{EUC} of Corwin and \cite{IUE} of Quastel for a more detailed introduction to this class. 
 	
 	Although it will not be a focus of this paper, let us briefly mention that there are also predicted phase diagrams for the behavior of six-vertex pure states outside of the ferroelectric regime (that is, when $\Delta < 1$). In this context, every slope $(s, t) \in (0, 1)^2$ is believed to admit a pure state with qualitative characteristics resembling those that appear in dimer models. More specifically, in the \emph{critical phase} $\Delta \in (-1, 1)$, all pure states should be either liquid or frozen and, in the \emph{antiferroelectric phase} $\Delta < -1$, all pure states should be liquid, gaseous, or frozen. For more information on these predictions, we refer to Section 9 of the survey \cite{ISVM} by Reshetikhin and Section 4 of that \cite{SVMFBC} by Palamarchuk--Reshetikhin. In certain cases, these predictions for $\Delta < 1$ have been mathematically established \cite{DGH,DPTM,LVHFI,CPTM,RD,HFD,DUFTHP,TDPM,CI,F,DSVM}. 
 
	However, except for its fourth part concerning frozen phases (whose analysis follows directly from definitions), no aspect of the above ferroelectric six-vertex model phase diagram has been mathematically proven until recently. To our knowledge, the only result in this direction concerns its second (KPZ) regime and appeared in Appendix A.2 of \cite{CFSAEPSSVMCL}, where a pure state $\mu (s) = \mu_{s, t}$ of any slope $(s, t) \in \partial \mathfrak{H}$ was introduced for the ferroelectric six-vertex model. Additionally, \cite{CFSAEPSSVMCL} established both qualitative and quantitative properties for $\mu (s)$, which are considerably different from those for pure states of tiling models. For instance, it was shown by Kenyon the latter are conformally invariant \cite{CI} with Gaussian free field fluctuations \cite{F}. In contrast, the six-vertex pure states $\mu (s)$ are quite anisotropic, exhibiting Baik--Rains fluctuations of exponent $\frac{1}{3}$ along a single direction and Gaussian fluctuations of exponent $\frac{1}{2}$ elsewhere \cite{CFSAEPSSVMCL}; this is an indication of the KPZ behavior manifested by these states. Essentially nothing (including existence and uniqueness) remains known about the conjecturally liquid states in the third regime of the phase diagram, where $(s, t) \notin \overline{\mathfrak{H}}$.
	
	Our results in this paper are twofold. First, we show that there are no pure states of any slope $(s, t) \in \mathfrak{H}$, thereby establishing the first regime of the above phase diagram. Second, we show that there is at most one pure state of any slope $(s, t) \in \partial \mathfrak{H}$, thereby showing that the KPZ states introduced in \cite{CFSAEPSSVMCL} are the unique ones of their slopes. 
	
	To establish these results, we will make use of the \emph{stochastic six-vertex model}, which is one whose six weights are of the form $(1, 1, B_1, B_2, 1 - B_1, 1 - B_2)$. This specialization was first considered by Gwa--Spohn \cite{SVMRSASH} in 1992 as an instance of the six-vertex model whose weights are stochastic and therefore give rise to a Markov process with local update rules. As observed in \cite{SVMRSASH}, and also by Borodin--Corwin--Gorin in \cite{SSVM}, this feature enables an interacting particle system interpretation of the stochastic six-vertex model; both algebraic and probabilistic ideas developed in the former context can then be adapted to analyze it. One reason as to why the stochastic six-vertex model is prevalent in our setting is that the pure states $\mu (s)$ described above (with slopes on $\partial \mathfrak{H}$) in fact serve as its translation-invariant stationary measures. For this reason, these pure states $\mu (s)$ are sometimes referred to as \emph{stochastic Gibbs states} \cite{SLSVM}. 
	
	So, we first apply a gauge transformation to view a pure state for any ferroelectric six-vertex model as one for some stochastic six-vertex model; the existence of such a transformation is guaranteed by the ferroelectricity of the original model. Next, we introduce the notion of a \emph{partition function stochastic lower bound} for a pure state $\mu$ of the stochastic six-vertex model, which essentially states the following (see \Cref{estimateprobabilitylower} below for a more precise formulation). With high probability, the partition function on a large $N \times N$ domain for this model with boundary data sampled under $\mu$ is at least $e^{-o(N^2)}$. The benefit to a measure $\mu$ satisfying this property is that it can be ``compared'' with a stochastic six-vertex model with free exit data. Indeed, since the partition function induced by the latter is equal to $1$, it can be quickly shown (see \Cref{modelmodelstochastic} below) that any event very unlikely under the free stochastic model is also unlikely under $\mu$. 
	
	Our task then reduces to establishing two results. The first (see \Cref{musmu} below) states that any pure state $\mu$ satisfying a partition function stochastic lower bound must coincide with a stochastic Gibbs state $\mu (s)$ introduced in \cite{CFSAEPSSVMCL}. The second (see \Cref{mulowerprobability} below) states that any pure state of slope $(s, t) \in \overline{\mathfrak{H}}$ must satisfy a partition function stochastic lower bound. Since the measures $\mu (s)$ all have slopes in $\partial \mathfrak{H}$, this shows that no pure state of slope $(s, t) \in \mathfrak{H}$ can exist and that it is uniquely given by $\mu (s)$ if $(s, t) \in \partial \mathfrak{H}$. 
	
	To prove the first of the two results mentioned above, we use one from \cite{LSLSSSVM}, which essentially states that the local statistics for any stochastic six-vertex model with free exit data are given by a stochastic Gibbs state $\mu (s)$ in the thermodynamical limit. Combining this with the above mentioned comparison between $\mu$ and such a stochastic six-vertex model, this will show $\mu = \mu (s)$. 
	
	To prove the second, we use the property that the partition function of any stochastic Gibbs state on an $N \times N$ domain is likely at least $e^{-O(N)}$ (differing from the more typical $e^{cN^2}$ asymptotics expected for liquid states). This can quickly be deduced from the facts that stochastic Gibbs states are stationary measures for the stochastic six-vertex model, that the total partition function for any such model with free exit data is equal to $1$, and that there are at most $e^{O(N)}$ choices of exit data. We will then compare the partition function for the pure state $\mu$ of slope $(s, t) \in \overline{\mathfrak{H}}$ with that for a stochastic Gibbs state, as follows. First, we introduce a ``sparsification procedure'' that reduces the slope of any pure state, while only reducing its partition function by at most a factor of $e^{o(N^2)}$. Next, since $(s, t) \in \overline{\mathfrak{H}}$, there exists a slope $(s_0, t_0) \in \partial \mathfrak{H}$ with $s_0 \ge s$ and $t_0 \ge t$ (see \Cref{s0t0} below). We can then interpret the pure state $\mu$, of slope $(s, t)$, as a sparsification of the stochastic Gibbs state $\mu (s_0)$, of slope $(s_0, t_0)$. Since the partition function of the latter is at least $e^{-O(N)}$, and since sparsification does not decrease partition functions by more than $e^{o (N^2)}$, this will yield the required lower bound of $e^{-o (N^2)}$ for the partition function induced by $\mu$. 
	
	Throughout this paper, we let $\mathbb{P}_{\mu}$ and $\mathbb{E}_{\mu}$ denote the probability and expectation with respect to any measure $\mu$, respectively. Furthermore, we let $\mathcal{A}^c$ denote the complement of any event $\mathcal{A}$, and $\textbf{1}_{\mathcal{A}} = \textbf{1} (\mathcal{A})$ denote the indicator function for $\mathcal{A}$.

	\subsection{Gibbs Measures for the Six-Vertex Model} 
	
	\label{Measures} 
	
	A \emph{domain} $\Lambda \subseteq \mathbb{Z}^2$ is a connected induced subgraph of $\mathbb{Z}^2$. The \emph{boundary} of $\Lambda$, denoted by $\partial \Lambda \subseteq \Lambda$, is the set of vertices in $\mathbb{Z}^2 \setminus \Lambda$ that are adjacent to some vertex in $\Lambda$, and the \emph{closure} of $\Lambda$ is defined to be the union $\overline{\Lambda} = \Lambda \cup \partial \Lambda$ of it with its boundary. 
	
	We now define six-vertex ensembles on domains $\Lambda \subseteq \mathbb{Z}^2$, to which end we begin by introducing arrow configurations. An \emph{arrow configuration} is a quadruple $(i_1, j_1; i_2, j_2)$ such that $i_1, j_1, i_2, j_2 \in \{ 0, 1 \}$ and $i_1 + j_1 = i_2 + j_2$. We view such a quadruple as an assignment of arrows to a vertex $v \in \Lambda$. Specifically, $i_1$ and $j_1$ denote the numbers of vertical and horizontal arrows entering $v$, respectively; similarly, $i_2$ and $j_2$ denote the numbers of vertical and horizontal arrows exiting $v$, respectively. The fact that $i_1 + j_1 = i_2 + j_2$ means that the numbers of incoming and outgoing arrows at $v$ are equal; this is sometimes referred to as \emph{arrow conservation}. There are six possible arrow configurations, which are depicted on the left side of \Cref{vertexdomain}. 
	
	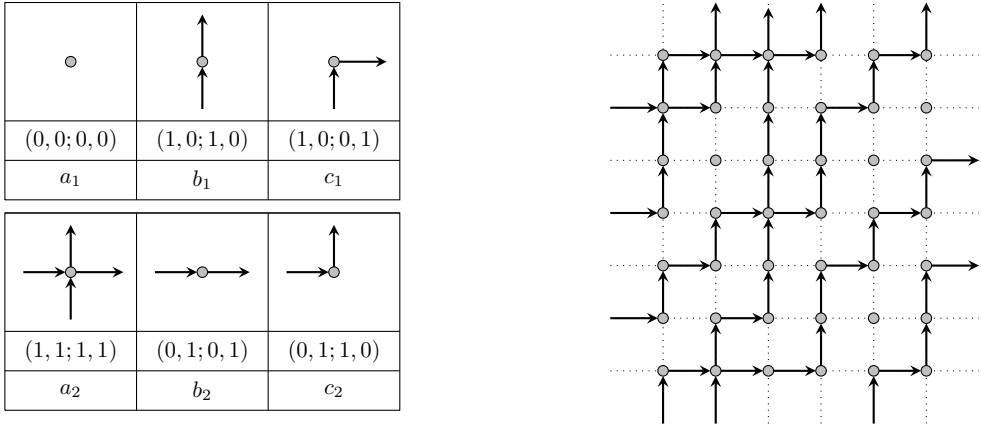
\begin{figure}[t]
		
		\begin{center}
			
			\begin{tikzpicture}[
			>=stealth,
			scale = .7	
			]
			
			\draw[] (-11.5, 3.25) -- (-4, 3.25) -- (-4, 7) -- (-11.5, 7) -- (-11.5, 3.25);
			
			\draw[] (-11.5, -.75) -- (-4, -.75) -- (-4, 3) -- (-11.5, 3) -- (-11.5, -.75);

			\draw[] (-11.5, 4) -- (-4, 4);
			\draw[] (-9, 3.25) -- (-9, 7);
			\draw[] (-6.5, 3.25) -- (-6.5, 7);
			\draw[] (-11.5, 4.75) -- (-4, 4.75);
			
			\draw[] (-11.5, 3) -- (-4, 3);
			\draw[] (-9, -.75) -- (-9, 3);
			\draw[] (-6.5, -.75) -- (-6.5, 3);
			\draw[] (-11.5, .75) -- (-4, .75);
			\draw[] (-11.5, 0) -- (-4, 0);
			
			\draw[] (-10.25, 4.375) circle [radius = 0] node[scale = .85]{$(0, 0; 0, 0)$};
			\draw[] (-7.75, 4.375) circle [radius = 0] node[scale = .85]{$(1, 0; 1, 0)$};
			\draw[] (-5.25, 4.375) circle [radius = 0] node[scale = .85]{$(1, 0; 0, 1)$};
			
			\draw[] (-10.25, .375) circle [radius = 0] node[scale = .85]{$(1, 1; 1, 1)$};
			\draw[] (-7.75, .375) circle [radius = 0] node[scale = .85]{$(0, 1; 0, 1)$};
			\draw[] (-5.25, .375) circle [radius = 0] node[scale = .85]{$(0, 1; 1, 0)$};

			\draw[] (-10.25, 3.625) circle [radius = 0] node[scale = .85]{$a_1$};
			\draw[] (-7.75, 3.625) circle [radius = 0] node[scale = .85]{$b_1$};
			\draw[] (-5.25, 3.625) circle [radius = 0] node[scale = .85]{$c_1$};
			
			\draw[] (-10.25, -.375) circle [radius = 0] node[scale = .85]{$a_2$};
			\draw[] (-7.75, -.375) circle [radius = 0] node[scale = .85]{$b_2$};
			\draw[] (-5.25, -.375) circle [radius = 0] node[scale = .85]{$c_2$};

			\draw[->, black,  thick] (-8.65, 1.875) -- (-7.85, 1.875);
			\draw[->, black,  thick] (-7.65, 1.875) -- (-6.85, 1.875);
			
			\draw[->, black,  thick] (-7.75, 4.975) -- (-7.75, 5.775);
			\draw[->, black,  thick] (-7.75, 5.975) -- (-7.75, 6.775);

			\draw[->, black,  thick] (-5.15, 5.875) -- (-4.25, 5.875);
			\draw[->, black,  thick] (-5.25, 4.975) -- (-5.25, 5.775);
			
			\draw[->, black,  thick] (-6.15, 1.875) -- (-5.35, 1.875);
			\draw[->, black,  thick] (-5.25, 1.975) -- (-5.25, 2.775);
			
			\draw[->, black,  thick] (-11.15, 1.875) -- (-10.35, 1.875);
			\draw[->, black,  thick] (-10.15, 1.875) -- (-9.25, 1.875);
			\draw[->, black,  thick] (-10.25, .975) -- (-10.25, 1.775);
			\draw[->, black,  thick] (-10.25, 1.975) -- (-10.25, 2.775);
			
			\filldraw[fill=gray!50!white, draw=black] (-10.25, 1.875) circle [radius=.1];
			\filldraw[fill=gray!50!white, draw=black] (-10.25, 5.875) circle [radius=.1];
			\filldraw[fill=gray!50!white, draw=black] (-7.75, 1.875) circle [radius=.1];
			\filldraw[fill=gray!50!white, draw=black] (-7.75, 5.875) circle [radius=.1];
			\filldraw[fill=gray!50!white, draw=black] (-5.25, 1.875) circle [radius=.1];
			\filldraw[fill=gray!50!white, draw=black] (-5.25, 5.875) circle [radius=.1];

			\draw[-, dotted] (0, 0) -- (7, 0);
			\draw[-, dotted] (0, 1) -- (7, 1);
			\draw[-, dotted] (0, 2) -- (7, 2);
			\draw[-, dotted] (0, 3) -- (7, 3);
			\draw[-, dotted] (0, 4) -- (7, 4);
			\draw[-, dotted] (0, 5) -- (7, 5);
			\draw[-, dotted] (0, 6) -- (7, 6);
			
			\draw[-, dotted] (1, -1) -- (1, 7); 
			\draw[-, dotted] (2, -1) -- (2, 7);  
			\draw[-, dotted] (3, -1) -- (3, 7); 
			\draw[-, dotted] (4, -1) -- (4, 7); 
			\draw[-, dotted] (5, -1) -- (5, 7);  
			\draw[-, dotted] (6, -1) -- (6, 7); 
			
			\draw[->, thick, black] (1, -1) -- (1, -.1); 
			\draw[->, thick, black] (2, -1) -- (2, -.1);
			\draw[->, thick, black] (5, -1) -- (5, -.1);
			
			\draw[->, thick, black] (1.1, 0) -- (1.9, 0);
			\draw[->, thick, black] (2.1, 0) -- (2.9, 0);
			\draw[->, thick, black] (3.1, 0) -- (3.9, 0);

			\draw[->, thick, black] (5.1, 0) -- (5.9, 0);

			\draw[->, thick, black] (2, .1) -- (2, .9);
			\draw[->, thick, black] (4, .1) -- (4, .9);
			\draw[->, thick, black] (4, 1.1) -- (4, 1.9);
			
			\draw[->, thick, black] (0, 1) -- (.9, 1);
			\draw[->, thick, black] (2.1, 1) -- (2.9, 1);
			\draw[->, thick, black] (4.1, 2) -- (4.9, 2);
			
			\draw[->, thick, black] (1, 1.1) -- (1, 1.9); 
			\draw[->, thick, black] (3, 1.1) -- (3, 1.9);

			\draw[->, thick, black] (1.1, 2) -- (1.9, 2);
			
			\draw[->, thick, black] (2, 2.1) -- (2, 2.9); 
			\draw[->, thick, black] (3, 2.1) -- (3, 2.9);
			\draw[->, thick, black] (5, 2.1) -- (5, 2.9);

			\draw[->, thick, black] (0, 3) -- (.9, 3);
			\draw[->, thick, black] (2.1, 3) -- (2.9, 3);
			\draw[->, thick, black] (3.1, 3) -- (3.9, 3);
			\draw[->, thick, black] (5.1, 3) -- (5.9, 3);
			
			\draw[->, thick, black] (1, 3.1) -- (1, 3.9); 
			\draw[->, thick, black] (3, 3.1) -- (3, 3.9);
			\draw[->, thick, black] (4, 3.1) -- (4, 3.9);
			
			\draw[->, thick, black] (3, 4.1) -- (3, 4.9); 
			\draw[->, thick, black] (3, 5.1) -- (3, 5.9); 
			\draw[->, thick, black] (3, 6.1) -- (3, 6.9);

			\draw[->, thick, black] (1, 4.1) -- (1, 4.9); 
			\draw[->, thick, black] (4, 4.1) -- (4, 4.9);

			\draw[->, thick, black] (0, 5) -- (.9, 5);
			\draw[->, thick, black] (1.1, 5) -- (1.9, 5);
			\draw[->, thick, black] (4.1, 5) -- (4.9, 5);

			\draw[->, thick, black] (1, 5.1) -- (1, 5.9); 
			\draw[->, thick, black] (2, 5.1) -- (2, 5.9);
			\draw[->, thick, black] (5, 5.1) -- (5, 5.9);

			\draw[->, thick, black] (1.1, 6) -- (1.9, 6);
			\draw[->, thick, black] (2.1, 6) -- (2.9, 6);
			\draw[->, thick, black] (3.1, 6) -- (3.9, 6);
			
			\draw[->, thick, black] (2, 6.1) -- (2, 7); 
			\draw[->, thick, black] (4, 6.1) -- (4, 7);
			
			\draw[->, thick, black] (5.1, 6) -- (5.9, 6);
			\draw[->, thick, black] (6, 6.1) -- (6, 7);

			\draw[->, thick, black] (6, 3.1) -- (6, 3.9);	
			\draw[->, thick, black] (6.1, 4) -- (7, 4);

			\draw[->, thick, black] (6, .1) -- (6, .9);
			\draw[->, thick, black] (6, 1.1) -- (6, 1.9);
			\draw[->, thick, black] (6.1, 2) -- (7, 2);

			\filldraw[fill=gray!50!white, draw=black] (1, 0) circle [radius=.1];
			\filldraw[fill=gray!50!white, draw=black] (1, 1) circle [radius=.1];
			\filldraw[fill=gray!50!white, draw=black] (1, 2) circle [radius=.1];
			\filldraw[fill=gray!50!white, draw=black] (1, 3) circle [radius=.1];
			\filldraw[fill=gray!50!white, draw=black] (1, 4) circle [radius=.1];
			\filldraw[fill=gray!50!white, draw=black] (1, 5) circle [radius=.1];
			\filldraw[fill=gray!50!white, draw=black] (1, 6) circle [radius=.1];
			
			\filldraw[fill=gray!50!white, draw=black] (2, 0) circle [radius=.1];
			\filldraw[fill=gray!50!white, draw=black] (2, 1) circle [radius=.1];
			\filldraw[fill=gray!50!white, draw=black] (2, 2) circle [radius=.1];
			\filldraw[fill=gray!50!white, draw=black] (2, 3) circle [radius=.1];
			\filldraw[fill=gray!50!white, draw=black] (2, 4) circle [radius=.1];
			\filldraw[fill=gray!50!white, draw=black] (2, 5) circle [radius=.1];
			\filldraw[fill=gray!50!white, draw=black] (2, 6) circle [radius=.1];
			
			\filldraw[fill=gray!50!white, draw=black] (3, 0) circle [radius=.1];
			\filldraw[fill=gray!50!white, draw=black] (3, 1) circle [radius=.1];
			\filldraw[fill=gray!50!white, draw=black] (3, 2) circle [radius=.1];
			\filldraw[fill=gray!50!white, draw=black] (3, 3) circle [radius=.1];
			\filldraw[fill=gray!50!white, draw=black] (3, 4) circle [radius=.1];
			\filldraw[fill=gray!50!white, draw=black] (3, 5) circle [radius=.1];
			\filldraw[fill=gray!50!white, draw=black] (3, 6) circle [radius=.1];
			
			\filldraw[fill=gray!50!white, draw=black] (4, 0) circle [radius=.1];
			\filldraw[fill=gray!50!white, draw=black] (4, 1) circle [radius=.1];
			\filldraw[fill=gray!50!white, draw=black] (4, 2) circle [radius=.1];
			\filldraw[fill=gray!50!white, draw=black] (4, 3) circle [radius=.1];
			\filldraw[fill=gray!50!white, draw=black] (4, 4) circle [radius=.1];
			\filldraw[fill=gray!50!white, draw=black] (4, 5) circle [radius=.1];
			\filldraw[fill=gray!50!white, draw=black] (4, 6) circle [radius=.1];
			
			\filldraw[fill=gray!50!white, draw=black] (5, 0) circle [radius=.1];
			\filldraw[fill=gray!50!white, draw=black] (5, 1) circle [radius=.1];
			\filldraw[fill=gray!50!white, draw=black] (5, 2) circle [radius=.1];
			\filldraw[fill=gray!50!white, draw=black] (5, 3) circle [radius=.1];
			\filldraw[fill=gray!50!white, draw=black] (5, 4) circle [radius=.1];
			\filldraw[fill=gray!50!white, draw=black] (5, 5) circle [radius=.1];
			\filldraw[fill=gray!50!white, draw=black] (5, 6) circle [radius=.1];
			
			\filldraw[fill=gray!50!white, draw=black] (6, 0) circle [radius=.1];
			\filldraw[fill=gray!50!white, draw=black] (6, 1) circle [radius=.1];
			\filldraw[fill=gray!50!white, draw=black] (6, 2) circle [radius=.1];
			\filldraw[fill=gray!50!white, draw=black] (6, 3) circle [radius=.1];
			\filldraw[fill=gray!50!white, draw=black] (6, 4) circle [radius=.1];
			\filldraw[fill=gray!50!white, draw=black] (6, 5) circle [radius=.1];
			\filldraw[fill=gray!50!white, draw=black] (6, 6) circle [radius=.1];

			\end{tikzpicture}
			
		\end{center}	
		
		\caption{\label{vertexdomain} The chart to the left shows all six possible arrow configurations, along with the associated vertex weights. An example of a six-vertex ensemble is shown to the right.}
	\end{figure}

	A \emph{six-vertex ensemble} on $\Lambda$ is an assignment of an arrow configuration to each vertex of $\Lambda$ in such a way that neighboring arrow configurations are \emph{consistent}; this means that, if $v_1, v_2 \in \Lambda$ are two adjacent vertices, then there is an arrow to $v_2$ in the configuration at $v_1$ if and only if there is an arrow from $v_1$ in the configuration at $v_2$. Observe in particular that the arrows in a six-vertex ensemble form non-crossing up-right directed paths connecting vertices of $\Lambda$, which enter and exit through its boundaries; see the right side of \Cref{vertexdomain} for a depiction.

	Let $\mathfrak{E} (\Lambda)$ denote the set of six-vertex ensembles on $\Lambda$. For any $\mathcal{E} \in \mathfrak{E} (\Lambda)$ and subdomain $\Lambda' \subseteq \Lambda$, we let $\mathcal{E}_{\Lambda'} \in \mathfrak{E} (\Lambda')$ denote the restriction of $\mathcal{E}$ to $\Lambda'$. We refer to \Cref{property} for a depiction. A \emph{cylinder set} in $\mathfrak{E} (\Lambda)$ is one of the form $\big\{ \mathcal{E} \in \mathfrak{E}(\Lambda): \mathcal{E}_{\Lambda'} = \mathcal{F} \big\}$, for some finite subdomain $\Lambda' \subseteq \Lambda$ and six-vertex ensemble $\mathcal{F} \in \mathfrak{E}(\Lambda')$. Assigning to $\mathfrak{E} (\Lambda)$ the $\sigma$-algebra generated by all cylinder sets, let $\mathscr{P} \big( \mathfrak{E} (\Lambda) \big)$ denote the space of probability measures on $\mathfrak{E} (\Lambda)$. 
	
	For any $\mathcal{E} \in \mathfrak{E} (\Lambda)$ and $(x, y) \in \Lambda$, let $\chi^{(v)} (x, y) = \chi_{\mathcal{E}}^{(v)} (x, y) \in \{ 0, 1 \}$ denote the indicator for the event that an arrow in $\mathcal{E}$ vertically exits from $(x, y)$. Stated alternatively, letting $\big( i_1 (x, y), j_1 (x, y); i_2 (x, y), j_2 (x, y) \big)$ denote the arrow configuration at $(x, y)$, we set $\chi^{(v)} (x, y) = i_1 (x, y + 1) = i_2 (x, y)$. Similarly, let $\chi^{(h)} (x, y)  = \chi_{\mathcal{E}}^{(h)} (x, y) = j_1 (x + 1, y) = j_2 (x, y) \in \{ 0, 1 \}$ denote the indicator for the event that an arrow in $\mathcal{E}$ horizontally exits through $(x, y)$. 
	
	Next fix six real numbers $a_1, a_2, b_1, b_2, c_1, c_2 > 0$. We assign a \emph{weight} $w(i_1, j_1; i_2, j_2)$ to each of the six possible arrow configurations by setting
	\begin{flalign}
	\label{aibici} 
	\begin{aligned} 
	& w (0, 0; 0, 0) = a_1; \qquad w (1, 0; 1, 0) = b_1; \qquad w (1, 0; 0, 1) = c_1; \\
	& w (1, 1; 1, 1) = a_2; \qquad w (0, 1; 0, 1) = b_2; \qquad w (0, 1; 1, 0) = c_2,
	\end{aligned}
	\end{flalign}
	
	\noindent and $w (i_1, j_1; i_2, j_2) = 0$ for any $(i_1, j_1; i_2, j_2)$ not of the above form. We refer to the left side of \Cref{vertexdomain} for a depiction.

	Given some domain $\Lambda \subseteq \mathbb{Z}^2$, a six-vertex ensemble $\mathcal{E} \in \mathfrak{E} (\Lambda)$, and a vertex $v \in \Lambda$ with some arrow configuration $(i_1, j_1; i_2, j_2) = \big( i_1 (v), j_1 (v); i_2 (v), j_2 (v) \big)$ under $\mathcal{E}$, we define the weight of $v$ with respect to $\mathcal{E}$ to be $w_{\mathcal{E}} (v) = w(i_1, j_1; i_2, j_2)$. If $\Lambda$ is finite, we may define the weight $w (\mathcal{E})$ of $\mathcal{E}$ to be the product of the weights of its vertices, namely,
	\begin{flalign}
	\label{eweight} 
	w (\mathcal{E}) = \prod_{v \in \Lambda} w_{\mathcal{E}} (v).
	\end{flalign} 
	
	\noindent The \emph{six-vertex model} (with free boundary conditions) on $\Lambda$ is then the probability measure $\mathbb{P} = \mathbb{P}_{\Lambda}^{(a_1, a_2, b_1, b_2, c_1, c_2)} \in \mathscr{P} \big( \mathfrak{E} (\Lambda) \big)$ such that $\mathbb{P} [\mathcal{E}] = Z^{-1} w (\mathcal{E})$ for any six-vertex ensemble $\mathcal{E} \in \mathfrak{E} (\Lambda)$, where $Z = \sum_{\mathcal{E} \in \mathfrak{E} (\Lambda)} w(\mathcal{E})$ is the \emph{partition function} chosen so these probabilities sum to $1$. 
	
	This definition no longer applies directly when $\Lambda$ is infinite, since then the product \eqref{eweight} might not converge. In this case, we consider the notion of Gibbs measures, given as follows. 
	
	\begin{definition}
		
		\label{ensemblemeasure}
		
		Let $\Lambda \subseteq \mathbb{Z}^2$ denote a domain, and let $\mu \in \mathscr{P} \big( \mathfrak{E} (\Lambda) \big)$ denote a measure. We say that $\mu$ satisfies the \emph{Gibbs property}, or is a \emph{Gibbs measure}, for the six-vertex model with weights $(a_1, a_2, b_1, b_2, c_1, c_2)$ if the following holds for any finite subdomain $\Lambda' \subseteq \Lambda$. Sample $\mathcal{E}\in \mathfrak{E} (\Lambda)$ under $\mu$, and condition on $\mathcal{E}_{\Lambda \setminus \Lambda'} = \mathcal{H}$, for some $\mathcal{H} \in \mathfrak{E} (\Lambda \setminus \Lambda')$. Then, for any $\mathcal{E}' \in \mathfrak{E} (\Lambda')$ consistent with $\mathcal{H}$, the conditional probability that $\mathcal{E}_{\Lambda'} = \mathcal{E}'$ is proportional to $w (\mathcal{E}')$ from \eqref{eweight}. Stated alternatively, $\mathbb{P}_{\mu} [\mathcal{E}_{\Lambda'} = \mathcal{E}' | \mathcal{E}_{\Lambda \setminus \Lambda'} = \mathcal{H}] = Z^{-1} w (\mathcal{E}')$ for any $\mathcal{E}' \in \mathfrak{E} (\Lambda')$ consistent with $\mathcal{H}$, where $Z = Z_{\Lambda'; \mu; \mathcal{H}}$ denotes the constant so that these probabilities sum to $1$. We refer to \Cref{property} for a depiction.

	\end{definition}

	\begin{figure}[t]
		
		\begin{center}
			
			\begin{tikzpicture}[
			>=stealth,
			scale = .5
			]

			\draw[->] (0, 10) -- (.9, 10);
			\draw[->] (1, 10) -- (1.9, 10);
			\draw[->] (2, 10) -- (2.9, 10);
			\draw[->] (3, 10) -- (3.9, 10);
			\draw[->] (4, 10) -- (4.9, 10);
			\draw[->] (5, 10) -- (5, 11);
			
			\draw[->] (0, 7) -- (.9, 7);
			\draw[->] (1, 7) -- (1.9, 7);
			\draw[->] (2, 7) -- (2.9, 7);
			\draw[->] (3, 7) -- (3, 7.9);
			\draw[->] (3, 8) -- (3.9, 8);
			\draw[->] (4, 8) -- (4.9, 8);
			\draw[->] (5, 8) -- (5, 8.9); 
			\draw[->] (5, 9) -- (5.9, 9);
			\draw[->] (6, 9) -- (6, 9.9); 
			\draw[->] (6, 10) -- (6.9, 10);
			\draw[->] (7, 10) -- (7, 11);
			
			\draw[->] (0, 4) -- (.9, 4);
			\draw[->] (1, 4) -- (1.9, 4);
			\draw[->] (2, 4) -- (2, 4.9);
			\draw[->] (2, 5) -- (2.9, 5);
			\draw[->] (3, 5) -- (3, 5.9);
			\draw[->] (3, 6) -- (3.9, 6);
			\draw[->, thick, dotted] (4, 6) -- (4.9, 6);
			\draw[->, thick, dotted] (5, 6) -- (5.9, 6);
			\draw[->, thick, dotted] (6, 6) -- (6, 6.9);
			\draw[->] (6, 7) -- (6, 7.9);
			\draw[->] (6, 8) -- (6.9, 8); 
			\draw[->] (7, 8) -- (7, 8.9);
			\draw[->] (7, 9) -- (7.9, 9); 
			\draw[->] (8, 9) -- (8, 9.9);
			\draw[->] (8, 10) -- (8.9, 10); 
			\draw[->] (9, 10) -- (9, 11);

			\draw[->] (0, 1) -- (.9, 1);
			\draw[->] (1, 1) -- (1.9, 1);
			\draw[->] (2, 1) -- (2.9, 1);
			\draw[->] (3, 1) -- (3, 1.9);
			\draw[->] (3, 2) -- (3.9, 2);
			\draw[->] (4, 2) -- (4, 2.9);
			\draw[->, thick, dotted] (4, 3) -- (4, 3.9);
			\draw[->, thick, dotted] (4, 4) -- (4, 4.9);
			\draw[->, thick, dotted] (4, 5) -- (4.9, 5);
			\draw[->, thick, dotted] (5, 5) -- (5.9, 5); 
			\draw[->, thick, dotted] (6, 5) -- (6, 5.9); 
			\draw[->, thick, dotted] (6, 6) -- (6.9, 6);
			\draw[->, thick, dotted] (7, 6) -- (7.9, 6); 
			\draw[->, thick, dotted] (8, 6) -- (8, 6.9);
			\draw[->] (8, 7) -- (8, 7.9);
			\draw[->] (8, 8) -- (8, 8.9);
			\draw[->] (8, 9) -- (8.9, 9);
			\draw[->] (9, 9) -- (9.9, 9);
			\draw[->] (10, 9) -- (10, 9.9);
			\draw[->] (10, 10) -- (11, 10);

			\draw[->] (4, 0) -- (4, .9); 
			\draw[->] (4, 1) -- (4, 1.9);
			\draw[->] (4, 2) -- (4.9, 2); 
			\draw[->] (5, 2) -- (5.9, 2);
			\draw[->] (6, 2) -- (6, 2.9); 
			\draw[->, thick, dotted] (6, 3) -- (6.9, 3); 
			\draw[->, thick, dotted] (7, 3) -- (7, 3.9);
			\draw[->, thick, dotted] (7, 4) -- (7, 4.9);
			\draw[->, thick, dotted] (7, 5) -- (7.9, 5); 
			\draw[->, thick, dotted] (8, 5) -- (8, 5.9);
			\draw[->, thick, dotted] (8, 6) -- (8.9, 6); 
			\draw[->, thick, dotted] (9, 6) -- (9, 6.9);
			\draw[->] (9, 7) -- (9, 7.9);
			\draw[->] (9, 8) -- (9.9, 8);
			\draw[->] (10, 8) -- (10, 8.9);
			\draw[->] (10, 9) -- (11, 9);

			\draw[->] (5, 0) -- (5, .9);
			\draw[->] (5, 1) -- (5.9, 1);
			\draw[->] (6, 1) -- (6.9, 1);
			\draw[->] (7, 1) -- (7, 1.9);
			\draw[->] (7, 2) -- (7, 2.9);
			\draw[->, thick, dotted] (7, 3) -- (7.9, 3);
			\draw[->, thick, dotted] (8, 3) -- (8, 3.9);
			\draw[->, thick, dotted] (8, 4) -- (8.9, 4);
			\draw[->, thick, dotted] (9, 4) -- (9, 4.9);
			\draw[->, thick, dotted] (9, 5) -- (9, 5.9);
			\draw[->] (9, 6) -- (9.9, 6);
			\draw[->] (10, 6) -- (10, 6.9);
			\draw[->] (10, 7) -- (10, 7.9);
			\draw[->] (10, 8) -- (11, 8);
			
			\draw[->] (7, 0) -- (7, .9); 
			\draw[->] (7, 1) -- (7.9, 1); 
			\draw[->] (8, 1) -- (8, 1.9);
			\draw[->] (8, 2) -- (8.9, 2);
			\draw[->] (9, 2) -- (9, 2.9); 
			\draw[->, thick, dotted] (9, 3) -- (9, 3.9); 
			\draw[->] (9, 4) -- (9.9, 4);
			\draw[->] (10, 4) -- (10, 4.9);
			\draw[->] (10, 5) -- (10, 5.9);
			\draw[->] (10, 6) -- (11, 6);
			
			\draw[->] (9, 0) -- (9, .9);
			\draw[->] (9, 1) -- (9.9, 1);
			\draw[->] (10, 1) -- (10, 1.9); 
			\draw[->] (10, 2) -- (10, 2.9); 
			\draw[->] (10, 3) -- (11, 3);
			
			\draw[dashed, very thick] (3.5, 2.5) -- (9.5, 2.5) -- (9.5, 7.5) -- (3.5, 7.5) -- (3.5, 2.5);
			
			\filldraw[fill=gray!50!white, draw=black] (1, 1) circle [radius=.1];
			\filldraw[fill=gray!50!white, draw=black] (1, 2) circle [radius=.1];
			\filldraw[fill=gray!50!white, draw=black] (1, 3) circle [radius=.1];
			\filldraw[fill=gray!50!white, draw=black] (1, 4) circle [radius=.1];
			\filldraw[fill=gray!50!white, draw=black] (1, 5) circle [radius=.1];
			\filldraw[fill=gray!50!white, draw=black] (1, 6) circle [radius=.1];
			\filldraw[fill=gray!50!white, draw=black] (1, 7) circle [radius=.1];
			\filldraw[fill=gray!50!white, draw=black] (1, 8) circle [radius=.1];
			\filldraw[fill=gray!50!white, draw=black] (1, 9) circle [radius=.1];
			\filldraw[fill=gray!50!white, draw=black] (1, 10) circle [radius=.1];
			
			\filldraw[fill=gray!50!white, draw=black] (2, 1) circle [radius=.1];
			\filldraw[fill=gray!50!white, draw=black] (2, 2) circle [radius=.1];
			\filldraw[fill=gray!50!white, draw=black] (2, 3) circle [radius=.1];
			\filldraw[fill=gray!50!white, draw=black] (2, 4) circle [radius=.1];
			\filldraw[fill=gray!50!white, draw=black] (2, 5) circle [radius=.1];
			\filldraw[fill=gray!50!white, draw=black] (2, 6) circle [radius=.1];
			\filldraw[fill=gray!50!white, draw=black] (2, 7) circle [radius=.1];
			\filldraw[fill=gray!50!white, draw=black] (2, 8) circle [radius=.1];
			\filldraw[fill=gray!50!white, draw=black] (2, 9) circle [radius=.1];
			\filldraw[fill=gray!50!white, draw=black] (2, 10) circle [radius=.1];
			
			\filldraw[fill=gray!50!white, draw=black] (3, 1) circle [radius=.1];
			\filldraw[fill=gray!50!white, draw=black] (3, 2) circle [radius=.1];
			\filldraw[fill=gray!50!white, draw=black] (3, 3) circle [radius=.1];
			\filldraw[fill=gray!50!white, draw=black] (3, 4) circle [radius=.1];
			\filldraw[fill=gray!50!white, draw=black] (3, 5) circle [radius=.1];
			\filldraw[fill=gray!50!white, draw=black] (3, 6) circle [radius=.1];
			\filldraw[fill=gray!50!white, draw=black] (3, 7) circle [radius=.1];
			\filldraw[fill=gray!50!white, draw=black] (3, 8) circle [radius=.1];
			\filldraw[fill=gray!50!white, draw=black] (3, 9) circle [radius=.1];
			\filldraw[fill=gray!50!white, draw=black] (3, 10) circle [radius=.1];
			
			\filldraw[fill=gray!50!white, draw=black] (4, 1) circle [radius=.1];
			\filldraw[fill=gray!50!white, draw=black] (4, 2) circle [radius=.1];
			\filldraw[fill=gray!50!white, draw=black] (4, 3) circle [radius=.1];
			\filldraw[fill=gray!50!white, draw=black] (4, 4) circle [radius=.1];
			\filldraw[fill=gray!50!white, draw=black] (4, 5) circle [radius=.1];
			\filldraw[fill=gray!50!white, draw=black] (4, 6) circle [radius=.1];
			\filldraw[fill=gray!50!white, draw=black] (4, 7) circle [radius=.1];
			\filldraw[fill=gray!50!white, draw=black] (4, 8) circle [radius=.1];
			\filldraw[fill=gray!50!white, draw=black] (4, 9) circle [radius=.1];
			\filldraw[fill=gray!50!white, draw=black] (4, 10) circle [radius=.1];
			
			\filldraw[fill=gray!50!white, draw=black] (5, 1) circle [radius=.1];
			\filldraw[fill=gray!50!white, draw=black] (5, 2) circle [radius=.1];
			\filldraw[fill=gray!50!white, draw=black] (5, 3) circle [radius=.1];
			\filldraw[fill=gray!50!white, draw=black] (5, 4) circle [radius=.1];
			\filldraw[fill=gray!50!white, draw=black] (5, 5) circle [radius=.1];
			\filldraw[fill=gray!50!white, draw=black] (5, 6) circle [radius=.1];
			\filldraw[fill=gray!50!white, draw=black] (5, 7) circle [radius=.1];
			\filldraw[fill=gray!50!white, draw=black] (5, 8) circle [radius=.1];
			\filldraw[fill=gray!50!white, draw=black] (5, 9) circle [radius=.1];
			\filldraw[fill=gray!50!white, draw=black] (5, 10) circle [radius=.1];
			
			\filldraw[fill=gray!50!white, draw=black] (6, 1) circle [radius=.1];
			\filldraw[fill=gray!50!white, draw=black] (6, 2) circle [radius=.1];
			\filldraw[fill=gray!50!white, draw=black] (6, 3) circle [radius=.1];
			\filldraw[fill=gray!50!white, draw=black] (6, 4) circle [radius=.1];
			\filldraw[fill=gray!50!white, draw=black] (6, 5) circle [radius=.1];
			\filldraw[fill=gray!50!white, draw=black] (6, 6) circle [radius=.1];
			\filldraw[fill=gray!50!white, draw=black] (6, 7) circle [radius=.1];
			\filldraw[fill=gray!50!white, draw=black] (6, 8) circle [radius=.1];
			\filldraw[fill=gray!50!white, draw=black] (6, 9) circle [radius=.1];
			\filldraw[fill=gray!50!white, draw=black] (6, 10) circle [radius=.1];
			
			\filldraw[fill=gray!50!white, draw=black] (7, 1) circle [radius=.1];
			\filldraw[fill=gray!50!white, draw=black] (7, 2) circle [radius=.1];
			\filldraw[fill=gray!50!white, draw=black] (7, 3) circle [radius=.1];
			\filldraw[fill=gray!50!white, draw=black] (7, 4) circle [radius=.1];
			\filldraw[fill=gray!50!white, draw=black] (7, 5) circle [radius=.1];
			\filldraw[fill=gray!50!white, draw=black] (7, 6) circle [radius=.1];
			\filldraw[fill=gray!50!white, draw=black] (7, 7) circle [radius=.1];
			\filldraw[fill=gray!50!white, draw=black] (7, 8) circle [radius=.1];
			\filldraw[fill=gray!50!white, draw=black] (7, 9) circle [radius=.1];
			\filldraw[fill=gray!50!white, draw=black] (7, 10) circle [radius=.1];
			
			\filldraw[fill=gray!50!white, draw=black] (8, 1) circle [radius=.1];
			\filldraw[fill=gray!50!white, draw=black] (8, 2) circle [radius=.1];
			\filldraw[fill=gray!50!white, draw=black] (8, 3) circle [radius=.1];
			\filldraw[fill=gray!50!white, draw=black] (8, 4) circle [radius=.1];
			\filldraw[fill=gray!50!white, draw=black] (8, 5) circle [radius=.1];
			\filldraw[fill=gray!50!white, draw=black] (8, 6) circle [radius=.1];
			\filldraw[fill=gray!50!white, draw=black] (8, 7) circle [radius=.1];
			\filldraw[fill=gray!50!white, draw=black] (8, 8) circle [radius=.1];
			\filldraw[fill=gray!50!white, draw=black] (8, 9) circle [radius=.1];
			\filldraw[fill=gray!50!white, draw=black] (8, 10) circle [radius=.1];
			
			\filldraw[fill=gray!50!white, draw=black] (9, 1) circle [radius=.1];
			\filldraw[fill=gray!50!white, draw=black] (9, 2) circle [radius=.1];
			\filldraw[fill=gray!50!white, draw=black] (9, 3) circle [radius=.1];
			\filldraw[fill=gray!50!white, draw=black] (9, 4) circle [radius=.1];
			\filldraw[fill=gray!50!white, draw=black] (9, 5) circle [radius=.1];
			\filldraw[fill=gray!50!white, draw=black] (9, 6) circle [radius=.1];
			\filldraw[fill=gray!50!white, draw=black] (9, 7) circle [radius=.1];
			\filldraw[fill=gray!50!white, draw=black] (9, 8) circle [radius=.1];
			\filldraw[fill=gray!50!white, draw=black] (9, 9) circle [radius=.1];
			\filldraw[fill=gray!50!white, draw=black] (9, 10) circle [radius=.1];
			
			\filldraw[fill=gray!50!white, draw=black] (10, 1) circle [radius=.1];
			\filldraw[fill=gray!50!white, draw=black] (10, 2) circle [radius=.1];
			\filldraw[fill=gray!50!white, draw=black] (10, 3) circle [radius=.1];
			\filldraw[fill=gray!50!white, draw=black] (10, 4) circle [radius=.1];
			\filldraw[fill=gray!50!white, draw=black] (10, 5) circle [radius=.1];
			\filldraw[fill=gray!50!white, draw=black] (10, 6) circle [radius=.1];
			\filldraw[fill=gray!50!white, draw=black] (10, 7) circle [radius=.1];
			\filldraw[fill=gray!50!white, draw=black] (10, 8) circle [radius=.1];
			\filldraw[fill=gray!50!white, draw=black] (10, 9) circle [radius=.1];
			\filldraw[fill=gray!50!white, draw=black] (10, 10) circle [radius=.1];
			
			\filldraw[fill=gray!50!black, draw=black] (5.5, 11.6) circle [radius=0] node[scale = 1.3]{$\mathcal{E}$};
			\filldraw[fill=gray!50!black, draw=black] (0, 2.5) circle [radius=0] node[scale = 1.3]{$\Lambda$};
			
			\filldraw[fill=gray!50!black, draw=black] (3.1, 3.5) circle [radius=0] node[scale = .9]{$\Lambda'$};
			\filldraw[fill=gray!50!black, draw=black] (4.55, 6.8) circle [radius=0] node[scale = .9]{$\mathcal{E}_{\Lambda'}$};

			\end{tikzpicture}
			
		\end{center}	
		
		\caption{\label{property} Depicted above is the Gibbs property from \Cref{ensemblemeasure}.}
	\end{figure}
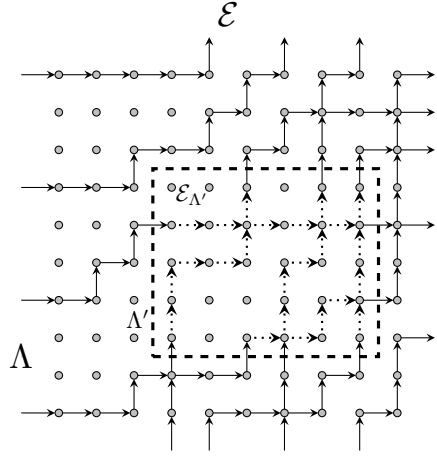

		Let us set some additional notation when $\Lambda = \mathbb{Z}^2$. For any $u \in \mathbb{Z}^2$, define the translation map $\mathfrak{T}_u: \mathbb{Z}^2 \rightarrow \mathbb{Z}^2$ by setting $\mathfrak{T}_u (v) = v - u$ for each $v \in \mathbb{Z}^2$. Then $\mathfrak{T}_u$ induces a map on both $\mathfrak{E} (\mathbb{Z}^2)$ and $\mathscr{P} \big( \mathfrak{E} (\mathbb{Z}^2) \big)$ that translates a six-vertex ensemble by $-u$; we also refer to them by $\mathfrak{T}_u$. A measure $\mu \in \mathscr{P} \big( \mathfrak{E} (\mathbb{Z}^2) \big)$ is called \emph{translation-invariant} if $\mathfrak{T}_u \mu = \mu$, for any $u \in \mathbb{Z}^2$. Recalling the horizontal and vertical indicator functions $\chi^{(v)} (x, y) = \chi_{\mathcal{E}}^{(v)} (x, y)$ and $\chi^{(h)} (x, y) = \chi_{\mathcal{E}}^{(h)} (x, y)$ associated with any six-vertex ensemble $\mathcal{E}\in \mathfrak{E} (\Lambda)$, we define the \emph{slope} of any translation-invariant Gibbs measure $\mu \in \mathscr{P} \big( \mathfrak{E} (\mathbb{Z}^2) \big)$ to be the pair $\big( \mathbb{E}_{\mu} \big[ \chi^{(v)} (0, 0) \big], \mathbb{E}_{\mu} \big[ \chi^{(h)} (0, 0) \big] \big) \in [0, 1] \times [0, 1]$. 
	
		Similarly to as above, an event $\mathcal{A}$ in the $\sigma$-algebra associated with $\mathfrak{E} (\mathbb{Z}^2)$ is called translation-invariant if $\mathfrak{T}_u \mathcal{A} = \mathcal{A}$, for any $u \in \mathbb{Z}^2$. We further call a translation-invariant measure $\mu \in \mathscr{P} \big( \mathfrak{E} (\mathbb{Z}^2) \big)$ \emph{ergodic} if, for any translation-invariant event $\mathcal{A}$, we have that $\mathbb{P}_{\mu} [A] \in \{ 0, 1 \}$. A translation-invariant, ergodic Gibbs measure for the six-vertex model on $\mathbb{Z}^2$ is called a \emph{pure state}.

	\subsection{Results} 
	
	\label{ModelDelta1}

	To state our results, we must first explicitly parameterize the ``lens-shaped'' region $\mathfrak{H}$ (shaded in \Cref{st}) where pure states of the ferroelectric six-vertex model should not exist. 
	
	So, fix parameters $a_1, a_2, b_1, b_2, c_1, c_2 > 0$ such that the $\Delta$ from \eqref{parameter1} satisfies $\Delta > 1$. We further assume throughout this paper that $a_1 a_2 > b_1 b_2$. The results we will state can then be obtained in the opposite regime by ``complementing'' vertical edges, that is, by replacing the arrow configuration $\big( i_1 (x, y), j_1 (x, y); i_2 (x, y), j_2 (x, y) \big)$ at $(x, y)$ with $\big(1 - i_2 (x, -y), j_1 (x, -y); 1 - i_1 (x, -y), j_2 (x, -y) \big)$ (equivalently, by first switching vertical arrows with absences of them and then reflecting the ensemble across the $x$-axis). This amounts to changing the six-vertex weights from $(a_1, a_2, b_1, b_2, c_1, c_2)$ to $(b_1, b_2, a_1, a_2, c_1, c_2)$ and slope of an associated pure state from $(s, t)$ to $(1 - s, t)$. Thus, by applying these symmetries, one can transform from the $a_1 a_2 > b_1 b_2$ setting to the $b_1 b_2 > a_1 a_2$ one. 	
	
	Now, define the function $\mathfrak{h} = \mathfrak{h}^{(a_1, a_2, b_1, b_2, c_1, c_2)}: \mathbb{R}^2 \rightarrow \mathbb{R}$ by setting 
	\begin{flalign} 
	\label{h} 
	\mathfrak{h} (x, y) = 2 xy \sqrt{\Delta^2 - 1} - \left( \sqrt{\displaystyle\frac{a_1 a_2}{b_1 b_2}} - \Delta + \sqrt{\Delta^2 - 1} \right) x +  \left( \sqrt{ \displaystyle\frac{a_1 a_2}{b_1 b_2}} - \Delta - \sqrt{\Delta^2 - 1} \right) y,
	\end{flalign}
	
	\noindent and the region $\mathfrak{H} = \mathfrak{H}^{(a_1, a_2, b_1, b_2, c_1, c_2)} \subset [0, 1]^2$ by setting
	\begin{flalign*}
	\mathfrak{H} = \Big\{ (s, t) \in (0, 1)^2: \displaystyle\max\big\{ \mathfrak{h} (s, t), \mathfrak{h} (t, s) \big\} < 0 \Big\}. 
	\end{flalign*}
	
	\noindent Further let $\partial \mathfrak{H}$ and $\overline{\mathfrak{H}} = \mathfrak{H} \cup \partial \mathfrak{H}$ denote the boundary and closure of $\mathfrak{H}$, respectively. In particular, defining the hyperbolas
	\begin{flalign}
	\label{h1h2} 
	\mathfrak{h}_1 = \big\{ (s, t) \in [0, 1]^2: \mathfrak{h} (s, t) = 0 \big\}; \qquad \mathfrak{h}_2 = \big\{ (s, t) \in [0, 1]^2: \mathfrak{h} (t, s) = 0 \big\},
	\end{flalign}
	
	\noindent we have that $\partial \mathfrak{H} = \mathfrak{h}_1 \cup \mathfrak{h}_2$. In the context of the ferroelectric six-vertex model, these curves $\mathfrak{h}_1$ and $\mathfrak{h}_2$ originally appeared as equation (3.38) of \cite{TCPFSVM}. We refer to \Cref{st} for a depiction.

	Now we can state the following theorem, which will be established in \Cref{Probability1} below. 
	
	\begin{thm} 
		
		\label{hstate}
		
		Fix $a_1, a_2, b_1, b_2, c_1, c_2 > 0$ with $a_1 a_2 > b_1 b_2$ and $\Delta > 1$, and let $(s, t) \in \overline{\mathfrak{H}}$.
		
		\begin{enumerate} 
			\item If $(s, t) \in \mathfrak{H}$, then there does not exist a pure state for the six-vertex model with weights $(a_1, a_2, b_1, b_2, c_1, c_2)$ and slope $(s, t)$.
			
			\item If $(s, t) \in \partial \mathfrak{H}$, then there exists a unique pure state for the six-vertex model with weights $(a_1, a_2, b_1, b_2, c_1, c_2)$ and slope $(s, t)$. 
			
		\end{enumerate} 
	\end{thm} 
	
	The measures described in the second part of \Cref{hstate} are reasonably explicit; they were introduced in \cite{CFSAEPSSVMCL} and will be recalled in \Cref{Translation} (see also the proof of \Cref{murhoaibici}) below.

	The remainder of this paper is organized as follows. In \Cref{ModelStochasticModel} we set notation and recall several results that will be used later in this article. Next, in \Cref{Probability1}, we establish \Cref{hstate} assuming two results given by \Cref{musmu} and \Cref{mulowerprobability} below. The former is then established in \Cref{StateEstimate} and the latter in \Cref{ProbabilityLower}.

	\subsection*{Acknowledgments}

	The author heartily thanks Alexei Borodin for helpful conversations and valuable suggestions on an earlier draft of this paper. The author is also grateful to Ivan Corwin, Jan de Gier, Vadim Gorin, and Richard Kenyon for enlightening conversations. This work was partially supported by NSF grant NSF DMS-1664619, the NSF Graduate Research Fellowship under grant DGE-1144152, and a Harvard Merit/Graduate Society Term-time Research Fellowship.

	\section{Miscellaneous Preliminaries}
	
	\label{ModelStochasticModel}

	In this section we introduce some notation and collect several (essentially known) results that will be used later in this work. We begin in \Cref{Paths} by describing notation for six-vertex ensembles that will be used throughout this paper. We then in \Cref{MeasuresExtremal} establish a result for the regularity of boundary data induced by a pure state of the six-vertex model. In \Cref{ModelStochastic} we recall the definition of the stochastic six-vertex model from \cite{SVMRSASH}, and in \Cref{Translation} we recall a family of pure states associated with it from \cite{CFSAEPSSVMCL}. We then in \Cref{Local} describe a result from \cite{LSLSSSVM} for the convergence of local statistics in the stochastic six-vertex model with free exit data.  
	
	\subsection{Paths in Six-Vertex Ensembles}
	
	\label{Paths} 
	
	Here we set some notation for paths in a six-vertex ensemble on a \emph{rectangular domain}, namely, one of the form $[M_1, M_2] \times [N_1, N_2] \subset \mathbb{Z}^2$, for some integers $M_1 \le M_2$ and $N_1 \le N_2$. We set its \emph{south}, \emph{west}, \emph{north}, and \emph{east boundaries} to be $[M_1, M_2] \times \{ N_1 - 1\}$, $\{ M_1 - 1 \} \times [N_1, N_2]$, $[M_1, M_2] \times \{ N_2 + 1 \}$, and $\{ M_2 + 1 \} \times [N_1, N_2]$, respectively. Before proceeding further on six-vertex ensembles on rectangular domains, let us recall the precise notion of a path on an arbitrary domain $\Lambda \subseteq \mathbb{Z}^2$. 
	
	For any two points $z_1 = (x_1, y_1) \in \mathbb{R}^2$ and $z_2 = (x_2, y_2) \in \mathbb{R}^2$, we write $z_1 \ge z_2$, or equivalently $z_2 \le z_1$, if $x_1 \ge x_2$ and $y_1 \ge y_2$. A curve $\textbf{p} \subset \mathbb{R}^2$ is \emph{nondecreasing} if, for any $z_1, z_2 \in \textbf{p}$ we either have that $z_1 \ge z_2$ or $z_2 \ge z_1$. Given some domain $\Lambda \subseteq \mathbb{Z}^2$, a \emph{(directed) path} $\textbf{p}$ on $\Lambda$ is a continuous, nondecreasing curve in $\mathbb{R}^2$ connecting a sequence of adjacent vertices in $\overline{\Lambda}$ by edges of $\mathbb{Z}^2$, such that no edge of $\textbf{p}$ connects two vertices in $\partial \Lambda$. Any compact path $\textbf{p}$ has a \emph{starting point}, which is the unique $z \in \textbf{p}$ such that $z \le z'$ for any $z' \in \textbf{p}$. It also has an \emph{ending point}, which is the unique $z \in \textbf{p}$ such that $z \ge z'$ for any $z' \in \textbf{p}$. We say $\textbf{p}$ \emph{enters $\Lambda$ through $z$} if $z$ is the starting point for $\textbf{p}$ and $z \in \partial \Lambda$. Similarly, $\textbf{p}$ \emph{exits $\Lambda$ through $z$} if $z$ is the ending point for $\textbf{p}$ and $z \in \partial \Lambda$. 
	
	Given two paths $\textbf{p}_1$ and $\textbf{p}_2$ on $\Lambda$, we say that $\textbf{p}_1$ \emph{lies below} $\textbf{p}_2$, or equivalently that $\textbf{p}_2$ \emph{lies above} $\textbf{p}_1$, if the following two statements hold. First, for any $(x_1, y_1) \in \textbf{p}_1$, there exists some $(x_2, y_2) \in \textbf{p}_2$ such that $x_1 \ge x_2$ and $y_1 \le y_2$. Second, for any $(x_2, y_2) \in \textbf{p}_2$, there exists some $(x_1, y_1) \in \textbf{p}_1$ such that $x_1 \ge x_2$ and $y_1 \le y_2$. If these two properties are satisfied, then we write $\textbf{p}_1 \le \textbf{p}_2$. An ensemble of paths $\mathcal{P} = (\textbf{p}_1, \textbf{p}_2, \ldots , \textbf{p}_k)$ is called \emph{non-crossing} if $\textbf{p}_1 \le \textbf{p}_2 \le \cdots \le \textbf{p}_k$ and no two paths $\textbf{p}_i, \textbf{p}_j \in \mathcal{P}$ share an edge (although they may share vertices); see \Cref{pathsvertex}.

	\begin{figure}[t]
		
		\begin{center}
			
			\begin{tikzpicture}[
			>=stealth,
			scale = .7	
			]

			\draw[thick] (1, -1) node[below, scale = .8]{$u_0$} -- (1, 0) -- (2, 0) -- (2, 1) -- (3, 1) -- (3, 2) -- (3, 3) -- (4, 3) -- (4, 4) -- (4, 5) -- (5, 5) -- (5, 6) -- (6, 6) -- (6, 7) node[above, scale = .8]{$v_0$}; 
			\draw[thick] (2, -1) node[below, scale = .8]{$u_{-1}$} -- (2, 0) -- (3, 0) -- (4, 0) -- (4, 1) -- (4, 2) -- (5, 2) -- (5, 3) -- (6, 3) -- (6, 4) -- (7, 4) node[right, scale = .8]{$v_{-1}$};
			\draw[thick] (5, -1) node[below, scale = .8]{$u_{-2}$} -- (5, 0) -- (6, 0) -- (6, 1) -- (6, 2) -- (7, 2) node[right, scale = .8]{$v_{-2}$};
			
			\draw[thick] (0, 1) node[left, scale = .8]{$u_1$} -- (1, 1) -- (1, 2) -- (2, 2) -- (2, 3) -- (3, 3) -- (3, 4) -- (3, 5) -- (3, 6) -- (4, 6) -- (4, 7) node[above, scale = .8]{$v_1$};
			\draw[thick] (0, 3) node[left, scale = .8]{$u_2$} -- (1, 3) -- (1, 4) -- (1, 5) -- (2, 5) -- (2, 6) -- (3, 6) -- (3, 7) node[above, scale = .8]{$v_2$};
			\draw[thick] (0, 5) node[left, scale = .8]{$u_3$} -- (1, 5) -- (1, 6) -- (2, 6) -- (2, 7) node[above, scale = .8]{$v_3$};
			
			\draw[very thick] (0, -1) -- (7, -1) -- (7, 7) -- (0, 7) -- (0, -1);
			
			\draw[] (1, -.5) node[left, scale = .8]{$\textbf{p}_0$};
			\draw[] (2, -.5) node[right, scale = .8]{$\textbf{p}_{-1}$};
			\draw[] (5, -.5) node[right, scale = .8]{$\textbf{p}_{-2}$};
			\draw[] (.5, 1) node[above, scale = .8]{$\textbf{p}_1$};
			\draw[] (.5, 3) node[above, scale = .8]{$\textbf{p}_2$};
			\draw[] (.5, 5) node[above, scale = .8]{$\textbf{p}_3$};
			
			\draw[] (.3, 7) node[above, scale = 1.2]{$\Lambda$};

			\end{tikzpicture}
			
		\end{center}	
		
		\caption{\label{pathsvertex} A non-crossing ensemble is depicted above, with its paths and boundary data labeled.  }
	\end{figure}
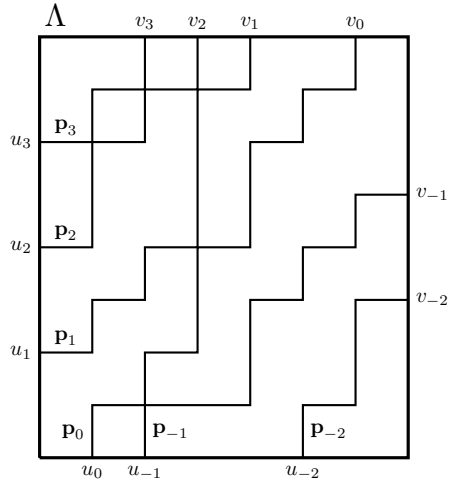

	Under this notation, six-vertex ensembles on a rectangular domain $\Lambda \subset \mathbb{Z}^2$ are in bijective correspondence with non-crossing path ensembles on $\Lambda$, each of whose paths enters $\Lambda$ through a vertex on its south or west boundary and exits $\Lambda$ through a vertex its north or east boundary. We index the paths in any such ensemble $\mathcal{P} = (\textbf{p}_{-B}, \textbf{p}_{1 - B}, \ldots , \textbf{p}_A)$, for some integers $A = A (\mathcal{P}) \ge 0$ and $B = B (\mathcal{P}) \ge -1$, so that $\textbf{p}_{-B}, \textbf{p}_{1 - B}, \ldots , \textbf{p}_0$ enter though the south boundary of $\Lambda$, and $\textbf{p}_1, \textbf{p}_2, \ldots , \textbf{p}_A$ enter through the west boundary. In this way, there are $A + B + 1$ paths in $\mathcal{P}$, and $\textbf{p}_1$ is the bottommost path entering through the west boundary of $\Lambda$. We refer to \Cref{pathsvertex} for a depiction of the non-crossing path ensemble associated with the six-vertex ensemble on the right side of \Cref{vertexdomain}. 	
	
	For each integer $i \in [-B, A]$, let $u_i \in \partial \Lambda$ denote the vertex through which the path $\textbf{p}_i$ enters $\Lambda$, and let $v_i \in \partial \Lambda$ denote the vertex through which $\textbf{p}_i$ exits $\Lambda$. We then refer to the $(A + B + 1)$-tuple $\textbf{u} = (u_{-B}, u_{1 - B}, \ldots , u_A) \subseteq \partial \Lambda$ as the \emph{entrance data} for $\mathcal{P}$ (or equivalently for the associated six-vertex ensemble) and to the $(A + B + 1)$-tuple $\textbf{v} = (v_{-B}, v_{1 - B}, \ldots , v_A) \subseteq \partial \Lambda$ as the \emph{exit data} for $\mathcal{P}$. \emph{Boundary data} for $\mathcal{P}$ consists of the union $\textbf{u} \cup \textbf{v}$ of its entrance and exit data. 
	
	Let $\mathfrak{E}_{\textbf{u}} (\Lambda) \subseteq \mathfrak{E} (\Lambda)$ denote the set of six-vertex ensembles on $\Lambda$ with entrance data given by $\textbf{u}$, and let $\mathfrak{E}_{\textbf{u}; \textbf{v}} (\Lambda) \subseteq \mathfrak{E}_{\textbf{u}} (\Lambda)$ denote the set of those with boundary data given by $\textbf{u} \cup \textbf{v}$. The \emph{six-vertex model with entrance data} $\textbf{u}$ (and free exit data) is the probability measure on $\mathfrak{E}_{\textbf{u}} (\Lambda)$ that assigns probability $\mathbb{P} [\mathcal{E}] = \mathbb{P}_{\textbf{u}; \Lambda} [\mathcal{E}] = Z_{\textbf{u}; \Lambda}^{-1} w (\mathcal{E})$ to any $\mathcal{E} \in \mathfrak{E}_{\textbf{u}} (\Lambda)$, where we recalled the weight $w (\mathcal{E})$ from \eqref{eweight} and defined the normalizing constant $Z_{\textbf{u}; \Lambda} = \sum_{\mathcal{E} \in \mathfrak{E}_{\textbf{u}} (\Lambda)} w(\mathcal{E})$ so that these probabilities sum to $1$. Similarly, the \emph{six-vertex model with boundary data} $\textbf{u} \cup \textbf{v}$ is the probability measure on $\mathfrak{E}_{\textbf{u}; \textbf{v}} (\Lambda)$ that assigns probability $\mathbb{P} [\mathcal{E}] = \mathbb{P}_{\textbf{u}; \textbf{v}; \Lambda} [\mathcal{E}] = Z_{\textbf{u}; \textbf{v}; \Lambda}^{-1} w (\mathcal{E})$ to any $\mathcal{E}\in \mathfrak{E}_{\textbf{u}; \textbf{v}} (\Lambda)$, where $Z_{\textbf{u} ;\textbf{v}; \Lambda}= \sum_{\mathcal{E} \in \mathfrak{E}_{\textbf{u}; \textbf{v}} (\Lambda)} w (\mathcal{E})$.

	\subsection{Regularity for Boundary Data} 

	\label{MeasuresExtremal}
	
	In this section we introduce the notion of regularity for boundary data of six-vertex ensembles and show that the boundary data induced by a pure state is likely regular. 
	
	To that end, we begin with the following two definitions. The first essentially states that a set is $(R; \eta)$-regular with slope $\rho$ on an interval $J$ if it $\eta$-approximates the Lebesgue measure of density $\rho$ on any subinterval $I \subseteq J$ of length $R$. The second essentially states that boundary data on a rectangular domain $\Lambda$ is regular if it is regular on each of the four boundaries of $\Lambda$.
	
	In what follows, for any interval $I \subseteq \mathbb{Z}^2$, we let $|I|$ denote the number of vertices in $I$. In particular, if $I = [A_1, A_2] \times \{ B \}$ or $I = \{ A \} \times [B_1, B_2]$ for some $A, A_1, A_2, B, B_1, B_2 \in \mathbb{Z}$ then $|I| = A_2 - A_1 + 1$ or $|I| = B_2 - B_1 + 1$, respectively.

	\begin{definition}	
		
		\label{boundary1}
		
		Fix real numbers $\eta, \rho \in (0, 1]$ and $R \ge 1$. For any subset $\textbf{u} \subseteq \mathbb{Z}^2$ and finite interval $J \subset \mathbb{Z}^2$ of the form $[A_1, A_2] \times \{ B \}$ or $\{A \} \times [B_1, B_2]$, we say $\textbf{u}$ is \emph{$(R; \eta)$-regular with slope $\rho$ on $J$} if the following holds. For any interval $I \subseteq J$ with $|I| \le R$, we have $\big| I \cap \textbf{u} - \rho |I| \big| \le \eta R$.

	\end{definition}

	\begin{definition}	
		
		\label{boundarystomega}
		
		Fix real numbers $\eta, s, t \in (0, 1]$ and $R \ge 1$. Let $\Lambda \subset \mathbb{Z}^2$ denote a rectangular domain, and let $\textbf{u} \cup \textbf{v}$ denote boundary data on $\Lambda$. We say $\textbf{u} \cup \textbf{v}$ is \emph{$(R; \eta)$-regular with slope $(s, t)$} if it is $(R; \eta)$-regular with slope $s$ along both the north and south boundaries of $\Lambda$, and it is $(R; \eta)$-regular with slope $t$ along both the west and east boundaries of $\Lambda$.

	\end{definition} 

	\begin{rem}
		
		\label{uetareta} 
		
		Suppose $\textbf{u} \subseteq \mathbb{Z}^2$ is $(R; \eta)$-regular with slope $s$ on some interval $J \subset \mathbb{Z}^2$. Then, for any interval $I \subseteq J$ such that $|I| \ge R$, we have $\big| |\textbf{u} \cap I| - s |I| \big| < 2 \eta |I|$. Indeed this follows from the fact that $I$ can be covered by a family of at most $2R^{-1} |I|$ intervals, each of length bounded above by $R$. 	 
		
	\end{rem} 

	The following lemma essentially states that boundary data induced by a pure state of slope $(s, t)$ is likely regular with this slope.

	\begin{lem}
		
		\label{rhosmu}
		
		Fix a pair $(s, t) \in (0, 1]^2$ and a translation-invariant, ergodic measure $\mu \in \mathscr{P} \big( \mathfrak{E} (\mathbb{Z}^2) \big)$ of slope $(s, t)$. For any real number $\eta \in (0, 1)$, there exists a constant $C_0 = C_0 (\eta, \mu) > 1$ such that the following holds. Let $R > C_0$ denote a real number and $M, N \in [R, \eta^{-1} R]$ denote integers; set $\Lambda = [1, M] \times [1, N] \subset \mathbb{Z}^2$. Sample a random six-vertex ensemble $\mathcal{E} \in \mathfrak{E}(\mathbb{Z}^2)$ under $\mu$, and let $\textbf{\emph{u}} \cup \textbf{\emph{v}}$ denote the boundary data induced on $\Lambda$ by the restriction $\mathcal{E}_{\mathbb{Z}^2 \setminus \Lambda}$. Then, with probability at least $1 - \eta$, the boundary data $\textbf{\emph{u}} \cup \textbf{\emph{v}}$ is $(R; \eta)$-regular with slope $(s, t)$.
		
	\end{lem}

	To establish this lemma, we first require the following one, which provides a law of large numbers for sums of the indicator functions $\chi^{(v)} (u) = \chi_{\mathcal{E}}^{(v)} (u)$ and $\chi^{(h)} (u) = \chi_{\mathcal{E}}^{(h)} (u)$ (from \Cref{Measures}) along horizontal and vertical lines.
	
	\begin{lem} 
		
		\label{sumchimu}
		
		Let $\mu \in \mathscr{P} \big( \mathfrak{E} (\mathbb{Z}^2) \big)$ denote a translation-invariant, ergodic measure of slope $(s, t) \in (0, 1]^2$, and let $\mathcal{E}\in \mathfrak{E}(\mathbb{Z}^2)$ denote a random six-vertex ensemble sampled under $\mu$. For any real number $\delta > 0$, there exists a constant $C_0 = C_0 (\delta; \mu) > 1$ such that
		\begin{flalign}
		\label{chisumestimate} 
		\mathbb{P}_{\mu} \Bigg[ \bigg| \displaystyle\frac{1}{N} \displaystyle\sum_{x = 0}^{N - 1} \chi_{\mathcal{E}}^{(v)} (x, 0) - s \bigg| < \delta \Bigg] \ge 1 - \delta; \qquad \mathbb{P}_{\mu} \Bigg[ \bigg| \displaystyle\frac{1}{N} \displaystyle\sum_{y = 0}^{N - 1} \chi_{\mathcal{E}}^{(v)} (0, y) - t \bigg| < \delta \Bigg] \ge 1 - \delta,
		\end{flalign} 
		
		\noindent for any integer $N > C_0$.  
	\end{lem}

	\begin{proof} 
		
		Let us only establish the first bound in \eqref{chisumestimate}, as the proof of the latter is entirely analogous. To that end, first define the map $f: \mathfrak{E} (\mathbb{Z}^2) \rightarrow \{ 0, 1 \}^{\mathbb{Z}^2}$ so that, for any six-vertex ensemble $\mathcal{E} \in \mathfrak{E}(\mathbb{Z}^2)$, the value of $f(\mathcal{E})$ at some vertex $u \in \mathbb{Z}^2$ is the indicator $\chi_{\mathcal{E}}^{(v)} (u) \in \{ 0, 1 \}$. Let $\lambda = f_* \mu$ denote the probability measure on $\{ 0, 1\}^{\mathbb{Z}^2}$ obtained as the pushforward of $\mu$ under $f$. 
		
		Next, let $\lambda_0$ denote the marginal of $\lambda$ on the $x$-axis $\mathbb{Z} \times \{ 0 \}$; in this way, $\lambda_0$ is a probability measure on $\{ 0, 1 \}^{\mathbb{Z}}$. Since $\mu$ is translation-invariant, as is $\lambda$, and therefore $\lambda_0$ is as well. Hence, since $\lambda_0$ prescribes the law of $\big( \chi_{\mathcal{E}}^{(v)} (x, 0) \big)_{x \in \mathbb{Z}} \in \{ 0, 1 \}^{\mathbb{Z}}$ under $\mu$, the strong ergodic theorem implies the almost sure (with respect to $\mu$) existence of the limit
		\begin{flalign}
		\label{sumchivse} 
		\displaystyle\lim_{N \rightarrow \infty} \displaystyle\frac{1}{N} \displaystyle\sum_{x = 0}^{N - 1} \chi_{\mathcal{E}}^{(v)} (x, 0) = S (\mathcal{E}).
		\end{flalign}
		
		Now, for any $N_1, N_2, y \in \mathbb{Z}$ with $N_1 \le N_2$ and any $\mathcal{E} \in \mathfrak{E} (\mathbb{Z}^2)$, observe that  
		\begin{flalign*}
		\Bigg| \displaystyle\sum_{x = N_1}^{N_2} \chi_{\mathcal{E}}^{(v)} (x, y + 1) - \displaystyle\sum_{x = N_1}^{N_2} \chi_{\mathcal{E}}^{(v)} (x, y) \Bigg| \le 1.
		\end{flalign*}
		
		\noindent Combining this with \eqref{sumchivse}, we deduce that 
		\begin{flalign}
		\label{sumchimnse}
		\displaystyle\lim_{M \rightarrow \infty} \left(\displaystyle\lim_{N \rightarrow \infty} \displaystyle\frac{1}{MN} \displaystyle\sum_{y = 0}^{M - 1} \displaystyle\sum_{x = 0}^{N - 1} \chi_{\mathcal{E}}^{(v)} (x, y) \right) = S (\mathcal{E}),
		\end{flalign}
		
		\noindent holds almost surely with respect to $\mu$. Then, since 
		\begin{flalign*} 
		\displaystyle\lim_{M \rightarrow \infty} \left( \displaystyle\lim_{N \rightarrow \infty} \displaystyle\frac{1}{MN} \Big| \partial \big( [0, N - 1] \times [0, M - 1] \big)\Big| \right) = 0,
		\end{flalign*}
		
		\noindent the event 
		\begin{flalign*}
		\Bigg\{ \displaystyle\lim_{M \rightarrow \infty} \bigg( \displaystyle\lim_{N \rightarrow \infty} \displaystyle\frac{1}{MN} \displaystyle\sum_{y = 0}^{M - 1} \displaystyle\sum_{x = 0}^{N - 1} \chi_{\mathcal{E}}^{(v)} (x, y) \bigg) \in I \Bigg\},
		\end{flalign*}
		
		\noindent is invariant with respect to translations by elements of $\mathbb{Z}^2$, for any interval $I \subseteq [0, 1]$. 
		
		Therefore, \eqref{sumchimnse} and the ergodicity of $\mu$ together imply $\mathbb{P}_{\mu} \big[ S (\mathcal{E}) \in I \big] \in \{ 0, 1 \}$, for any $I \subseteq [0, 1]$. It follows that there exists some $s_0 \in [0, 1]$ such that $S (\mathcal{E}) = s_0$ almost surely under $\mu$. Since $\mu$ has slope $(s, t)$, we have $s_0 = s$, and so the first estimate in \eqref{chisumestimate} follows from the almost sure limit \eqref{sumchivse}. As mentioned previously, the proof of the second is very similar and is therefore omitted. 
	\end{proof}

	Now we can establish \Cref{rhosmu}.
	
	\begin{proof}[Proof of \Cref{rhosmu}]
		
		Set $\delta = \frac{\eta^3}{12}$, and let $I \subset \partial \Lambda$ denote an interval satisfying $\frac{\eta R}{3} \le |I| \le \frac{\eta R}{2}$. If the constant $C_0 (\eta, \mu)$ here is chosen to be larger than the $3 \eta^{-1} C_0 (\delta, \mu)$ from \Cref{sumchimu}, then \eqref{chisumestimate} implies 
		\begin{flalign}
		\label{nuintervaln}
		\mathbb{P}_{\mu} \bigg[ \Big| \big| I \cap (\textbf{u} \cup \textbf{v}) \big| - s|I| \Big| < \delta R \bigg] \ge 1 - \delta, \quad \text{or} \quad 
		\mathbb{P}_{\mu} \bigg[ \Big| \big| I \cap (\textbf{u} \cup \textbf{v}) \big| - t|I| \Big| < \delta R \bigg] \ge 1 - \delta, 
		\end{flalign} 
		
		\noindent if $I$ lies on either the north or south boundary of $\Lambda$, or if $I$ lies on either the east or west boundary of $\Lambda$, respectively.  
		
		Next let $\mathcal{I}$ denote a set of mutually disjoint intervals of lengths between $\frac{\eta R}{3}$ and $\frac{\eta R}{2}$, whose union constitutes the north and south boundaries of $\Lambda$. Similarly, let $\mathcal{J}$ denote a set of mutually disjoint intervals of lengths between $\frac{\eta R}{3}$ and $\frac{\eta R}{2}$, whose union constitutes the east and west boundaries of $\Lambda$. Defining the event 
		\begin{flalign*}
		\mathcal{A} = \bigg\{ \displaystyle\max_{I \in \mathcal{I} } \Big| \big| I \cap (\textbf{u} \cup \textbf{v}) \big| - s |I| \Big| < \delta R \bigg\} \cap \bigg\{ \displaystyle\max_{I \in \mathcal{J} } \Big| \big| I \cap (\textbf{u} \cup \textbf{v}) \big| - t |I| \Big| < \delta R \bigg\},
		\end{flalign*} 
		
		\noindent observe that boundary data $\textbf{u} \cup \textbf{v}$ is $(R; \eta)$-regular with slope $(s, t)$ on $\mathcal{A}$. Indeed, suppose for instance $J$ is an interval on the west boundary of $\partial \Lambda$ of length at most $R$. If $|J| \le \eta R$, then $\big| |J \cap (\textbf{u} \cup \textbf{v})| - t |J| \big| \le |J| \le \eta R$. If instead $|J| > \eta R$, then $J$ admits a cover by at most $3 \eta^{-1}$ intervals in $\mathcal{J}$, and so on the event $\mathcal{A}$ we have $\big| |J \cap (\textbf{u} \cup \textbf{v})| - t |J| \big| \le 3 \eta^{-1} \delta R \le \eta R$. This establishes the bound $\big| |J \cap (\textbf{u} \cup \textbf{v})| - t |J| \big| \le \eta R$ in both cases and therefore verifies the $(R; \eta)$-regularity with slope of $\textbf{u} \cup \textbf{v}$ along the west boundary of $\partial \Lambda$. The proof that $\textbf{u} \cup \textbf{v}$ is $(R; \eta)$-regular (with the appropriate slope) along the other three boundaries is entirely analogous and is therefore omitted. 
		
		It therefore suffices to lower bound the probability $\mathbb{P}_{\mu} [\mathcal{A}]$, which we will do by estimating the sizes of $\mathcal{I}$ and $\mathcal{J}$. To that end, observe that since the total length of the north and south boundaries of $\Lambda$ is $2M \le 2 \eta^{-1} R$, we have $|\mathcal{I}| \le \frac{6M}{\eta R} \le 6 \eta^{-2}$. Similarly, $|\mathcal{J}| \le \frac{6N}{\eta R} \le 6 \eta^{-2}$. Thus, by a union bound over \eqref{nuintervaln}, we deduce that $\mathbb{P}_{\mu} [ \mathcal{A}] \ge 1 - 12 \eta^{-2} \delta = 1 - \eta$, which implies the lemma. 
	\end{proof}

	\subsection{Stochastic Six-Vertex Model} 

	\label{ModelStochastic} 	
	
	Fix real numbers $B_1, B_2 \in (0, 1)$. Introduced in \cite{SVMRSASH}, the \emph{$(B_1, B_2)$-stochastic six-vertex model} is the special case of the six-vertex model whose weights are given by
	\begin{flalign} 
	\label{bstochasticb}
	\begin{aligned}
	& w (0, 0; 0, 0) = 1; \qquad w (1, 0; 1, 0) = B_1; \qquad w (1, 0; 0, 1) = 1 - B_1; \\
	& w (1, 1; 1, 1) = 1; \qquad w (0, 1; 0, 1) = B_2; \qquad w (0, 1; 1, 0) = 1 - B_2.
	\end{aligned}
	\end{flalign}
	
	\noindent These weights are \emph{stochastic}, in that they satisfy the property that sum of the weights of all arrow configurations with a fixed set of incoming arrows is equal to $1$. Stated alternatively, we have $\sum_{i_2, j_2} w (i_1, j_1; i_2, j_2) = 1$ for any fixed $(i_1, j_1) \in \{ 0, 1 \} \times \{ 0, 1 \}$. 
	
	The choice \eqref{bstochasticb} enables a Markovian sampling procedure for the stochastic six-vertex model with fixed entrance data. Let us explain this sampling in more detail for the model on a rectangular domain $\Lambda = [M_1, M_2] \times [N_1, N_2] \subset \mathbb{Z}^2$; we will also permit $\infty \in \{ M_2, N_2 \}$, allowing $\Lambda$ to be a quadrant. We may assume by translation that $M_1 = 1 = N_1$, and so we set $M_2 = M$ and $N_2 = N$. For each integer $n \in [2, M + N]$, we will define a probability measure $\mathcal{P}_n = \mathcal{P}_{n; \Lambda} (B_1, B_2)$ on the set of six-vertex ensembles whose vertices are all contained in the subdomain $\mathcal{T}_n = \{ (x, y) \in \mathbb{Z}_{\ge 0}^2: x + y \le n \} \cap \Lambda \subseteq \Lambda$. The stochastic six-vertex model on $\Lambda$, denoted by $\mathcal{P} = \mathcal{P}_{\Lambda} (B_1, B_2)$, will then be set to $\mathcal{P}_{M + N}$. For each positive integer $n$, we define $\mathcal{P}_{n + 1}$ from $\mathcal{P}_n$ through the following Markovian update rules (in the case $n = 1$, $\mathcal{T}_n$ is empty, and so $\mathscr{P} \big( \mathfrak{E} (\mathcal{T}_1) \big)$ is empty).

	\begin{figure}[t]
		
		\begin{center}
			
			\begin{tikzpicture}[
			>=stealth,
			scale = .7	
			]

			\draw[thick, dotted] (.5, -.5) -- (5.75, -.5) -- (.5, 4.75) -- (.5, -.5);
			
			\draw[] (.25, -.75) circle[radius = 0] node[]{$\mathcal{T}_{n - 1}$};
			
			\draw[thick, dotted] (.25, 5.25) -- (6.25, -.75) -- (6.75, -.25) -- (.75, 5.75) -- (.25, 5.25);
			
			\draw[] (6.875, -.75) circle[radius = 0] node[]{$\mathcal{D}_n$};
			
			\draw[->, thick, black] (1, -1) -- (1, -.1); 
			\draw[->, thick, black] (2, -1) -- (2, -.1);
			\draw[->, thick, black] (5, -1) -- (5, -.1);
			
			\draw[->, thick, black] (1.1, 0) -- (1.9, 0);
			\draw[->, thick, black] (2.1, 0) -- (2.9, 0);
			\draw[->, thick, black] (3.1, 0) -- (3.9, 0);
			
			\draw[->, thick, black] (5.1, 0) -- (5.9, 0);
			
			\draw[->, thick, black] (2, .1) -- (2, .9);
			\draw[->, thick, black] (4, .1) -- (4, .9);
			\draw[->, thick, black] (4, 1.1) -- (4, 1.9);
			
			\draw[->, thick, black] (0, 1) -- (.9, 1);
			\draw[->, thick, black] (2.1, 1) -- (2.9, 1);
			
			\draw[->, thick, black] (1, 1.1) -- (1, 1.9); 
			\draw[->, thick, black] (3, 1.1) -- (3, 1.9);
			
			\draw[->, thick, black] (1.1, 2) -- (1.9, 2);
			
			\draw[->, thick, black] (2, 2.1) -- (2, 2.9); 
			\draw[->, thick, black] (3, 2.1) -- (3, 2.9);
			
			\draw[->, thick, black] (0, 3) -- (.9, 3);
			\draw[->, thick, black] (2.1, 3) -- (2.9, 3);
			
			\draw[->, thick, black] (1, 3.1) -- (1, 3.9);
			
			\draw[->, thick, black] (1, 4.1) -- (1, 4.9); 
			
			\draw[->, thick, black] (0, 5) -- (.9, 5);

			\draw[->, ultra thick, dashed] (4.1, 2) -- (4.9, 2);
			\draw[->, ultra thick, dashed] (3.1, 3) -- (3.9, 3);
			\draw[->, ultra thick, dashed] (3, 3.1) -- (3, 3.9);
			\draw[->, ultra thick, dashed] (1.1, 5) -- (1.9, 5);
			\draw[->, ultra thick, dashed] (1, 5.1) -- (1, 5.9); 
			\draw[->, ultra thick, dashed] (6, .1) -- (6, .9);

			\filldraw[fill=gray!50!white, draw=black] (1, 0) circle [radius=.1];
			\filldraw[fill=gray!50!white, draw=black] (1, 1) circle [radius=.1];
			\filldraw[fill=gray!50!white, draw=black] (1, 2) circle [radius=.1];
			\filldraw[fill=gray!50!white, draw=black] (1, 3) circle [radius=.1];
			\filldraw[fill=gray!50!white, draw=black] (1, 4) circle [radius=.1];
			\filldraw[fill=gray!50!white, draw=black] (1, 5) circle [radius=.1];
			\filldraw[fill=gray!50!white, draw=black] (1, 6) circle [radius=.1];
			
			\filldraw[fill=gray!50!white, draw=black] (2, 0) circle [radius=.1];
			\filldraw[fill=gray!50!white, draw=black] (2, 1) circle [radius=.1];
			\filldraw[fill=gray!50!white, draw=black] (2, 2) circle [radius=.1];
			\filldraw[fill=gray!50!white, draw=black] (2, 3) circle [radius=.1];
			\filldraw[fill=gray!50!white, draw=black] (2, 4) circle [radius=.1];
			\filldraw[fill=gray!50!white, draw=black] (2, 5) circle [radius=.1];
			\filldraw[fill=gray!50!white, draw=black] (2, 6) circle [radius=.1];
			
			\filldraw[fill=gray!50!white, draw=black] (3, 0) circle [radius=.1];
			\filldraw[fill=gray!50!white, draw=black] (3, 1) circle [radius=.1];
			\filldraw[fill=gray!50!white, draw=black] (3, 2) circle [radius=.1];
			\filldraw[fill=gray!50!white, draw=black] (3, 3) circle [radius=.1];
			\filldraw[fill=gray!50!white, draw=black] (3, 4) circle [radius=.1];
			\filldraw[fill=gray!50!white, draw=black] (3, 5) circle [radius=.1];
			\filldraw[fill=gray!50!white, draw=black] (3, 6) circle [radius=.1];
			
			\filldraw[fill=gray!50!white, draw=black] (4, 0) circle [radius=.1];
			\filldraw[fill=gray!50!white, draw=black] (4, 1) circle [radius=.1];
			\filldraw[fill=gray!50!white, draw=black] (4, 2) circle [radius=.1];
			\filldraw[fill=gray!50!white, draw=black] (4, 3) circle [radius=.1];
			\filldraw[fill=gray!50!white, draw=black] (4, 4) circle [radius=.1];
			\filldraw[fill=gray!50!white, draw=black] (4, 5) circle [radius=.1];
			\filldraw[fill=gray!50!white, draw=black] (4, 6) circle [radius=.1];
			
			\filldraw[fill=gray!50!white, draw=black] (5, 0) circle [radius=.1];
			\filldraw[fill=gray!50!white, draw=black] (5, 1) circle [radius=.1];
			\filldraw[fill=gray!50!white, draw=black] (5, 2) circle [radius=.1];
			\filldraw[fill=gray!50!white, draw=black] (5, 3) circle [radius=.1];
			\filldraw[fill=gray!50!white, draw=black] (5, 4) circle [radius=.1];
			\filldraw[fill=gray!50!white, draw=black] (5, 5) circle [radius=.1];
			\filldraw[fill=gray!50!white, draw=black] (5, 6) circle [radius=.1];
			
			\filldraw[fill=gray!50!white, draw=black] (6, 0) circle [radius=.1];
			\filldraw[fill=gray!50!white, draw=black] (6, 1) circle [radius=.1];
			\filldraw[fill=gray!50!white, draw=black] (6, 2) circle [radius=.1];
			\filldraw[fill=gray!50!white, draw=black] (6, 3) circle [radius=.1];
			\filldraw[fill=gray!50!white, draw=black] (6, 4) circle [radius=.1];
			\filldraw[fill=gray!50!white, draw=black] (6, 5) circle [radius=.1];
			\filldraw[fill=gray!50!white, draw=black] (6, 6) circle [radius=.1];

			\filldraw[dotted, fill = gray!40!white] (12, 0) -- (18, 0) -- (18, 1.5) -- (14, 1.5) -- (14, 6) -- (12, 6) -- (12, 0);

			\draw[very thick] (12, 0) -- (12, 6) -- (18, 6) -- (18, 0) -- (12, 0);
			\draw[thick] (14, 1.5) -- (16.5, 1.5) -- (16.5, 5) -- (14, 5) -- (14, 1.5);
			\draw[dashed] (14, 5) -- (14, 6);
			\draw[dashed] (16.5, 1.5) -- (18, 1.5);
			
			\draw[] (15, 6.5) circle[radius = 0] node[scale = 1.6]{$\Lambda$};
			\draw[] (15.25, 3.25) circle[radius = 0] node[scale = 1.3]{$\Lambda'$};
			\draw[] (13.25, .75) circle[radius = 0] node[scale = 1.3]{$\Lambda \setminus \Xi$};
			
			\end{tikzpicture}
			
		\end{center}	
		
		\caption{\label{vertexdomain2} The sampling procedure for the stochastic six-vertex model is depicted on the left. The Markov property described in \Cref{ensemblen} is depicted on the right.}
	\end{figure}
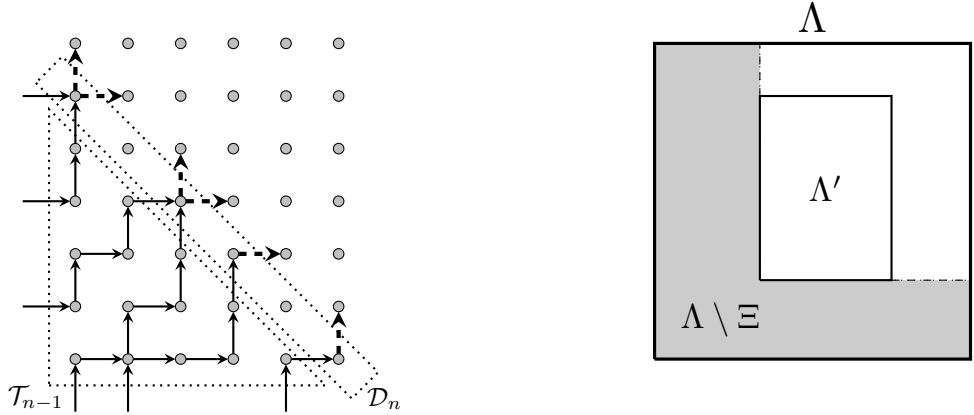

	Use $\mathcal{P}_n$ to sample a six-vertex ensemble $\mathcal{E}_n$ on $\mathcal{T}_n$, which assigns an arrow configuration to each vertex in $\Lambda$ strictly below the diagonal $\mathcal{D}_n = \{ (x, y) \in \mathbb{Z}_{> 0}^2: x + y = n \} \cap \Lambda$. Each vertex on $\mathcal{D}_n$ is also given ``half'' of an arrow configuration, in the sense that it is given the directions of all entering paths but no direction of any exiting path; we refer to the left side of \Cref{vertexdomain2} for a depiction. To extend $\mathcal{E}_n$ to an ensemble on $\mathcal{T}_{n + 1}$, we must ``complete'' the configurations (specify the exiting paths) at all vertices $v \in \mathcal{D}_n$. Any half-configuration can be completed in at most two ways; selecting between these completions is done randomly, according to the probabilities \eqref{bstochasticb}. All choices are mutually independent. 
	
	In this way, we obtain a random ensemble $\mathcal{E}_{n + 1}$ on $\mathcal{T}_{n + 1}$; the resulting probability measure on path ensembles with vertices in $\mathcal{T}_{n + 1}$ is denoted by $\mathcal{P}_{n + 1}$. Then define $\mathcal{P} = \mathcal{P}_{M + N}$ if $M + N$ is finite and $\mathcal{P} = \lim_{n \rightarrow \infty} \mathcal{P}_n$ if $M + N = \infty$. 
	
	\begin{rem}
	
	\label{ensemblen}
	
	Observe that a random six-vertex ensemble $\mathcal{E} \in \mathfrak{E} (\Lambda)$ sampled under $\mathcal{P}$ satisfies the following Markov property. Let $\Lambda' = [X_1, X_2] \times [Y_1, Y_2] \subseteq \Lambda$ be a rectangular subdomain; let $\textbf{u} = \textbf{u} (\mathcal{E})$ denote the entrance data for $\mathcal{E}_{\Lambda'}$ on $\Lambda'$; and set $\Xi = \big( [X_1, \infty) \times [Y_1, \infty) \big) \cap \Lambda$. Then, conditional on $\textbf{u}$, $\mathcal{E}_{\Lambda'}$ is independent of $\mathcal{E}_{\Lambda \setminus \Xi}$. We refer to the right side of \Cref{vertexdomain2} for a depiction.  
	
	\end{rem}

	\subsection{Pure States for the Stochastic Six-Vertex Model} 
	
	\label{Translation}
	
	In this section we describe a family of pure states $\mu (\rho)$ for the stochastic six-vertex model. We begin by explaining how to sample the restrictions of these measures to the positive quadrant $\mathbb{Z}_{>  0}^2$; the extension of these measures to all of $\mathbb{Z}^2$ is then done through a translation and limiting procedure. 
	
	To implement the former task, we require certain entrance data. For any $\rho_1, \rho_2 \in [0, 1]$, \emph{double-sided $(\rho_1, \rho_2)$-Bernoulli entrance data} is that in which sites on the $y$-axis are independently entrance sites for paths with probability $\rho_1$, and sites on the $x$-axis are independently entrance sites for paths with probability $\rho_2$. 
	
	Now, fix $0 < B_1 < B_2 < 1$ and define $\kappa > 1$ and $\varphi: [0, 1] \rightarrow [0, 1]$ by 
	\begin{flalign}
	\label{kappafunction}
	\kappa = \kappa_{B_1, B_2} = \displaystyle\frac{1 - B_1}{1 - B_2} > 1; \qquad \varphi (z) = \varphi_{B_1, B_2} (z) = \displaystyle\frac{\kappa z}{(\kappa - 1) z + 1}, \quad \text{for any $z \in [0, 1]$.}
	\end{flalign} 
	
	Further fix $\rho \in [0, 1]$, and consider the stochastic six-vertex model on the nonnegative quadrant with double-sided $\big( \varphi (\rho), \rho \big)$-Bernoulli entrance data; denote the associated measure on $\mathfrak{E} (\mathbb{Z}_{> 0}^2)$ by $\mu_0 = \mu_0 (\rho)$. It was shown in \cite{CFSAEPSSVMCL} that this measure is translation-invariant in the following sense. 
	
	\begin{lem}[{\cite[Lemma A.2]{CFSAEPSSVMCL}}]
		
		\label{murhochi}
		
	Fix $\rho \in [0, 1]$, and sample a six-vertex ensemble $\mathcal{E} \in \mathfrak{E} (\mathbb{Z}_{> 0}^2)$ randomly under $\mu_0 (\rho)$. Then, for any $(x, y) \in \mathbb{Z}_{\ge 0}^2$, the random variables $\big\{ \chi_{\mathcal{E}}^{(h)} (x, y + 1), \chi_{\mathcal{E}}^{(h)} (x, y + 2), \ldots \big\} \cup \big\{ \chi_{\mathcal{E}}^{(v)} (x + 1, y), \chi_{\mathcal{E}}^{(v)} (x + 2, y), \ldots \big\}$ are mutually independent. Furthermore, each $\chi_{\mathcal{E}}^{(h)} (x, y)$ and $\chi_{\mathcal{E}}^{(v)} (x, y)$ is a $0$-$1$ Bernoulli random variable with mean $\varphi (\rho)$ and $\rho$, respectively. 
	
	\end{lem}

	Observe in particular that, if $(x, y) = (0, 0)$, then this is the definition of double-sided $\big( \varphi (\rho), \rho \big)$-Bernoulli entrance data for $\mu_0$. The fact that it is also true for any $(x, y) \in \mathbb{Z}_{\ge 0}^2$ allows us to define a family of measures $\{ \mu_N \} = \big\{ \mu_N (\rho) \big\}_{N \ge 0}$ as follows. For each integer $N \ge 1$, let $\mu_N = \mu_N (\rho) = \mathfrak{T}_{(N, N)} \mu_0$ denote the measure on $\mathbb{Z}_{\ge -N}^2$ formed by translating $\mu_0$ by $(-N, -N)$ (that is, $N$ spaces down and to the left). Due to the translation-invariance of $\mu_0$ from \Cref{murhochi}, these measures are compatible in the sense that $\mu_M$ is the restriction of $\mu_N$ to $\mathbb{Z}_{> -M}^2$, for any integers $N \ge M \ge 0$. 
	
	Therefore, we can define the limit $\mu = \mu (\rho) = \lim_{N \rightarrow \infty} \mu_N (\rho)$ on all of $\mathbb{Z}^2$. By \Cref{murhochi}, this limit is invariant with respect to any vertical or horizontal shift. Thus, $\mu (\rho)$ is a translation-invariant Gibbs measure for the stochastic six-vertex model on $\mathbb{Z}^2$. The following proposition indicates that it is moreover a pure state.

	\begin{prop} 
		
		\label{murho1} 
		
		Fix real numbers $0 < B_1 < B_2 < 1$ and $\rho \in [0, 1]$. Then, the measure $\mu (\rho)$ is a pure state of slope $\big( \rho, \varphi (\rho) \big)$ for the $(B_1, B_2)$-stochastic six-vertex model.
	\end{prop} 

	\begin{rem}
		
		\label{rhoh} 
		
		Under the stochastic specialization $(a_1, a_2, b_1, b_2, c_1, c_2) = (1, 1, B_1, B_2, 1 - B_1, 1 - B_2)$ of weights, the condition $\mathfrak{h} (s, t) = 0$ (recall \eqref{h}) is equivalent to $t = \varphi (s)$. 
		
	\end{rem}

	To establish this proposition, we require the following lemma from \cite{LSLSSSVM} that essentially states the following. If the boundary data for two stochastic six-vertex models on the quadrant coincide in a union of two intervals on the $x$-axis, then these models can be coupled to coincide in rectangles above these intervals with high probability. Although the below result was stated in \cite{LSLSSSVM} in the case of one interval and for the stochastic six-vertex model on the discrete upper half-plane, it is quickly verified that the proof there also applies with multiple intervals for the model on the quadrant. 
	
	\begin{lem}[{\cite[Proposition 2.17]{LSLSSSVM}}]
		
		\label{modelcouple} 
		
		For any $0 < B_1 < B_2 < 1$, there exists a constant $c = c(B_2) > 0$ such that the following holds. Fix two entrance data $\textbf{\emph{u}}$ and $\textbf{\emph{u}}'$ on the quadrant $\mathbb{Z}_{> 0}^2$, and let $\mathfrak{S}$ and $\mathfrak{S}'$ denote the $(B_1, B_2)$-stochastic six-vertex models on $\mathbb{Z}^2$ with entrance data $\textbf{\emph{u}}$ and $\textbf{\emph{u}}'$, respectively. Further let $K_1, K_2, M, N > 0$ denote integers; assume that $N \le K_1, K_2$ and $N \ge \frac{3M}{1 - B_2}$. Denote the intervals $I_1 = [K_1 - N, K_1 + N] \times \{ 0 \}$ and $I_2 = [K_2 - N, K_2 + N] \times \{ 0 \}$, and additionally suppose that $\textbf{\emph{u}} \cap (I_1 \cup I_2) = \textbf{\emph{u}}' \cap (I_1 \cup I_2)$. 
		
		Randomly sample two six-vertex ensembles $\mathcal{E} \in \mathfrak{E}_{\textbf{\emph{u}}} (\mathbb{Z}_{> 0}^2)$ and $\mathcal{E}' \in \mathfrak{E}_{\textbf{\emph{u}}'} (\mathbb{Z}_{> 0}^2)$ under $\mathfrak{S}$ and $\mathfrak{S}'$, respectively, and set the domain $\Lambda = \big( [K_1 - M, K_1 + M] \cup [K_2 - M, K_2 + M] \big) \times [0, M]$. Then, it is possible to couple $\mathcal{E}$ and $\mathcal{E}'$ in such a way that $\mathcal{E}_{\Lambda} = \mathcal{E}_{\Lambda}'$ holds with probability at least $1 - c^{-1} e^{-cM}$.  
		
	\end{lem}

	Now we can establish \Cref{murho1}.
	
	\begin{proof}[Proof of \Cref{murho1}]
		
		As mentioned above, $\mu (\rho)$ is translation-invariant and has slope $\big( \rho, \varphi (\rho) \big)$ by \Cref{murhochi}, so it suffices to show that it is ergodic. To that end, recalling the translation map $\mathfrak{T}_u: \mathbb{Z}^2 \rightarrow \mathbb{Z}^2$ from \Cref{Measures} and abbreviating $\mathfrak{T}_X = \mathfrak{T}_{(-X, 0)}$, we will show that $\mu (\rho)$ satisfies the following mixing property. For any two events $\mathcal{A}_1$ and $\mathcal{A}_2$, we have that 
		\begin{flalign}
		\label{abprobabilityr} 
		\displaystyle\lim_{X \rightarrow \infty} \mathbb{P} \big[ \mathcal{A}_1 \cap \mathfrak{T}_X \mathcal{A}_2 \big] = \mathbb{P} [\mathcal{A}_1] \mathbb{P} [\mathcal{A}_2],
		\end{flalign} 
		
		\noindent where the probability measure is with respect to $\mu (\rho)$. 
		
		Assuming \eqref{abprobabilityr}, we can quickly establish the ergodicity of $\mu (\rho)$. Indeed, letting $\mathcal{A}$ denote a translation-invariant event, apply \eqref{abprobabilityr} with $\mathcal{A}_1 = \mathcal{A} = \mathcal{A}_2$. Then, $\mathfrak{T}_X \mathcal{A}_2 = \mathcal{A}$ by the translation-invariance of $\mathcal{A}$, and so \eqref{abprobabilityr} yields $\mathbb{P}[\mathcal{A}] = \mathbb{P} [\mathcal{A}]^2$. Hence, $\mathbb{P} [\mathcal{A}] \in \{ 0, 1 \}$, and so $\mu (\rho)$ is ergodic. 

		Thus, it remains to verify \eqref{abprobabilityr}. To that end, we may assume that $\mathcal{A}_1$ and $\mathcal{A}_2$ are both cylinder subsets. Then, there exist finite domains $\Lambda, \Lambda' \subset \mathbb{Z}^2$ and six-vertex ensembles $\mathcal{G}_1 \in \mathfrak{E}(\Lambda)$ and $\mathcal{G}_2 \in \mathfrak{E} (\Lambda')$ such that $\mathcal{A}_1 = \big\{ \mathcal{E} \in \mathfrak{E} (\mathbb{Z}^2): \mathcal{E}_{\Lambda} = \mathcal{G}_1 \big\}$ and $\mathcal{A}_2 = \big\{ \mathcal{E} \in \mathfrak{E} (\mathbb{Z}^2): \mathcal{E}_{\Lambda'} = \mathcal{G}_2 \big\}$. Letting $R > 0$ denote a sufficiently large integer such that $\Lambda, \Lambda' \subseteq [-R, R] \times [-R, R]$, we may then assume that $\Lambda = [-R, R] \times [-R, R] = \Lambda'$.
		
		Next, let $N \ge \frac{18R}{1 - B_2}$ and $X \ge 6N$ denote two integers. We consider two $(B_1, B_2)$-stochastic six-vertex models on the quadrant $\Gamma = (-2N, \infty) \times (-R, \infty) \subset \mathbb{Z}^2$ with different entrance data. The first, denoted $\mathfrak{S}$, has double-sided $\big( \varphi (\rho), \rho \big)$-Bernoulli entrance data $\textbf{u}$. Defining the intervals $I_1 = [-N, N] \times \{ -R \}$ and $I_2 = [X - N, X + N] \times \{ -R \}$, the second, denoted $\mathfrak{S}'$, has entrance data $\textbf{u} \cap (I_1 \cup I_2)$. Let $\mathcal{E}, \mathcal{E}' \in \mathfrak{E} (\Gamma)$ denote six-vertex ensembles sampled under $\mathfrak{S}$ and $\mathfrak{S}'$, respectively. 
		
		Observe from \Cref{murhochi} that the law of $\mathcal{E}$ is given by the marginal of $\mu (\rho)$ on $\mathfrak{E} (\Gamma)$. Thus, we may set the events $\mathcal{A}_1$ and $\mathcal{A}_2$ in \eqref{abprobabilityr} by
		\begin{flalign*}
		 & \mathcal{A}_1 = \big\{ \mathcal{E}_{\Lambda} = \mathcal{G}_1 \big\}; \qquad \mathcal{A}_2 = \big\{ \mathcal{E}_{\Lambda} = \mathcal{G}_2 \big\}.
		\end{flalign*}
		
		\noindent Let us further define the events $\mathcal{A}_1'$ and $\mathcal{A}_2'$ by 
		\begin{flalign*}
		& \mathcal{A}_1' = \big\{ \mathcal{E}_{\Lambda}' = \mathcal{G}_1 \big\}; \qquad \mathcal{A}_2' = \big\{ \mathcal{E}_{\Lambda}' = \mathcal{G}_2 \big\}.
		\end{flalign*}
		
		Now, the entrance data for $(\mathcal{E}, \mathcal{E}')$ coincide on $I_1 \cup I_2$. Therefore, the $(K_1, K_2, M, N) = \big( 2N, X + 2N, \big\lfloor \frac{(1 - B_2) N}{3} \big\rfloor, N \big)$ case of \Cref{modelcouple} (and the fact that $N \ge \frac{18R}{1 - B_2}$) yields the existence of a constant $c = c(B_2) > 0$ such that $\mathcal{E}$ and $\mathcal{E}'$ can be coupled to coincide on $\Lambda \cup \mathfrak{T}_X \Lambda$, off of an event of probability at most $c^{-1} e^{-cN}$. In particular, this implies the three bounds  
		\begin{flalign}
		\label{a1a2a1a2} 
		\begin{aligned}
		& \big| \mathbb{P}_{\mathfrak{S}} [\mathcal{A}_1] - \mathbb{P}_{\mathfrak{S}'} [\mathcal{A}_1'] \big| \le c^{-1} e^{-cN}; \qquad \big| \mathbb{P}_{\mathfrak{S}} [\mathfrak{T}_X \mathcal{A}_2] - \mathbb{P}_{\mathfrak{S}'} [\mathfrak{T}_X \mathcal{A}_2'] \big| \le c^{-1} e^{-cN}; \\
		& \big| \mathbb{P}_{\mathfrak{S}} [\mathcal{A}_1 \cap \mathfrak{T}_X \mathcal{A}_2] - \mathbb{P}_{\mathfrak{S}'} [\mathcal{A}_1' \cap \mathfrak{T}_X \mathcal{A}_2'] \big| \le c^{-1} e^{-cN}.
		\end{aligned} 
		\end{flalign} 
		
		Next let $\mathcal{B}$ denote the event that no path in $\mathcal{E}'$ intersects both $\Lambda$ and $\mathfrak{T}_X \Lambda$. Since the $\big\{ \chi^{(v)} (u) \big\}$ are mutually independent for $u \in I_1 \cup I_2$ (due to the choice of double-sided Bernoulli boundary data for $\textbf{u}$), it follows from the sampling procedure described in \Cref{ModelStochastic} that the events $\mathcal{A}_1' \cap \mathcal{B}$ and $\mathfrak{T}_X \mathcal{A}_2' \cap \mathcal{B}$ are independent after conditioning on $\mathcal{B}$. Hence,
		\begin{flalign*}
		\mathbb{P}_{\mathfrak{S}'} [\mathcal{A}_1' | \mathcal{B}] \mathbb{P}_{\mathfrak{S}'} [ \mathfrak{T}_X \mathcal{A}_2' | \mathcal{B}] = \mathbb{P}_{\mathfrak{S}'} [\mathcal{A}_1' \cap \mathfrak{T}_X \mathcal{A}_2' | \mathcal{B}],
		\end{flalign*}
		
		\noindent which implies that 
		\begin{flalign}
		\label{a1a21}
		\big| \mathbb{P}_{\mathfrak{S}'} [\mathcal{A}_1' \cap \mathfrak{T}_X \mathcal{A}_2'] - \mathbb{P}_{\mathfrak{S}'} [\mathcal{A}_1'] \mathbb{P}_{\mathfrak{S}'} [ \mathfrak{T}_X \mathcal{A}_2'] \big| \le 2 \mathbb{P}_{\mathfrak{S}'}[\mathcal{B}^c].
		\end{flalign}
		
		To bound $\mathbb{P}_{\mathfrak{S}'}[\mathcal{B}^c]$, we again apply \Cref{modelcouple}. To that end, let $\mathcal{E}_0 \in \mathfrak{E} (\Gamma)$ denote the six-vertex ensemble on $\Gamma$ with no paths. Then, since $X \ge 6N$, the entrance data for $(\mathcal{E}', \mathcal{E}_0)$ coincide on the interval $I_0 = [2N, 4N] \times \{ - R \}$. Thus, \Cref{modelcouple} implies that it is possible to couple $\mathcal{E}'$ with $\mathcal{E}_0$ on $[3N - R, 3N + R] \times [-R, R]$, off of an event of probability at most $c^{-1} e^{-cN}$. Since if $\mathcal{E}'$ contains no paths in $[3N - R, 3N + R] \times [-R, R]$ then $\mathcal{B}$ holds, it follows that $\mathbb{P}_{\mathfrak{S}'} [\mathcal{B}^c] \le c^{-1} e^{-cN}$. 
		
		Combining this with \eqref{a1a2a1a2}, \eqref{a1a21}, and the fact that $\mathbb{P}_{\mathfrak{S}} [\mathcal{A}_2] = \mathbb{P}_{\mathfrak{S}} [\mathfrak{T}_X \mathcal{A}_2]$ (by the translation-invariance of $\mu (\rho)$), we obtain 
		\begin{flalign*}
		\Big| \mathbb{P}_{\mathfrak{S}} \big[ \mathcal{A}_1 \cap \mathfrak{T}_X \mathcal{A}_2 \big] - \mathbb{P}_{\mathfrak{S}} [\mathcal{A}_1] \mathbb{P}_{\mathfrak{S}} [\mathcal{A}_2]  \Big| \le 6 c^{-1} e^{-cN},
		\end{flalign*}
		
		\noindent whenever $N \ge \frac{18R}{1 - B_2}$ and $X \ge 6N$. Letting $X$ tend to $\infty$ and then $N$ tend to $\infty$, we deduce \eqref{abprobabilityr} and therefore the proposition.
	\end{proof}

	\subsection{Local Statistics in the Stochastic Six-Vertex Model} 
	
	\label{Local}
	
	In this section we provide a convergence of local statistics result for the stochastic six-vertex model with free exit data that essentially states the following. Consider this model on the quadrant $\mathbb{Z}_{> 0}^2$ whose entrance data is regular with slope $\rho \in [0, 1]$ on an interval of the positive $x$-axis; then, the local statistics of this model slightly above this interval are approximately given by the measure $\mu (\rho)$ from \Cref{Translation}. This statement follows from results in \cite{LSLSSSVM}, but we will briefly outline how more precisely.

	\begin{lem} [{\cite{LSLSSSVM}}]
		
		\label{modellocal} 
		
		Fix real numbers $0 < B_1 < B_2 < 1$ and $\rho, \varepsilon \in (0, 1]$; an integer $k \ge 1$; and a six-vertex ensemble $\mathcal{G} \in \mathfrak{E} \big( [-k, k] \times [-k, k] \big)$. Then, there exist constants $\delta = \delta (B_1, B_2, \rho, \varepsilon, k) > 0$, $c_0 = c_0 (B_1, B_2, \rho, k) > 0$, and $C_0 = C_0 (B_1, B_2, \rho, \varepsilon, k) > 1$ such that the following holds for any integers $N > C_0$ and $M \in \big[ \frac{c_0 N}{2}, c_0 N \big]$. 
		
		Let $\textbf{\emph{u}}$ denote some entrance data on the quadrant $\mathbb{Z}_{> 0}^2$ that is $(\delta N; \delta)$-regular with slope $\rho$ on the interval $\big[ \frac{N}{2}, \frac{3N}{2} \big) \times \{ 0 \}$. Sample random six-vertex ensembles $\mathcal{E} \in \mathfrak{E}_{\textbf{\emph{u}}} (\mathbb{Z}_{> 0}^2)$ under the $(B_1, B_2)$-stochastic six-vertex model $\mathfrak{S}$ on $\mathbb{Z}_{> 0}^2$ with entrance data $\textbf{\emph{u}}$ and free exit data, and $\mathcal{F} \in \mathfrak{E} (\mathbb{Z}^2)$ under $\mu (\rho)$. Then, 
		\begin{flalign}
		\label{enkmk}
		\Big| \mathbb{P}_{\mathfrak{S}} \big[ \mathcal{E}_{[N - k, N + k] \times [M - k, M + k]} = \mathcal{G} \big] - \mathbb{P}_{\mu (\rho)} \big[ \mathcal{F}_{[-k, k] \times [-k, k]} = \mathcal{G} \big] \Big| < \varepsilon. 
		\end{flalign}

	\end{lem} 
	
	\begin{proof}[Proof (Outline)]
		
		This result will follow from Theorem 1.3 of \cite{LSLSSSVM} (which establishes \eqref{enkmk} on a discrete cylinder); Proposition 5.7 of \cite{LSLSSSVM} (which compares the stochastic six-vertex model on a discrete cylinder with that on the discrete upper half-plane); and \Cref{modelcouple} above (which enables a comparison between the stochastic six-vertex model on the discrete upper half-plane with that on the positive quadrant). Let us briefly outline how this proceeds. 
		
		First define the discrete torus $\mathbb{T}_N = \mathbb{Z} / N \mathbb{Z}$ and discrete cylinder $\mathfrak{C}_{N; 2M} = \mathbb{T}_N \times \{ 1, 2, \ldots , 2M \}$. Denoting the interval $J = \big( [ \frac{N}{2}, \frac{3N}{2} ) \times \{ 0 \} \big) \cap \mathbb{Z}^2$, we may identify $J$ with the lower end of $\partial \mathfrak{C}_{N; 2M}$; in this way, $\textbf{u} \cap J$ induces entrance data $\mathfrak{C}_{N; 2M}$. Let $\mathcal{E}'$ denote random six-vertex ensemble on $\mathfrak{C}_{N; 2M}$ sampled from the $(B_1, B_2)$-stochastic six-vertex model $\mathfrak{S}'$ on $\mathfrak{C}_{N; 2M}$ with this entrance data $\textbf{u} \cap J$ and free exit data. Then, the $\psi = \rho$ case of Theorem 1.3 of \cite{LSLSSSVM} yields for each real number $\omega > 0$ the existence of constants $\delta = \delta (B_1, B_2, \rho, \varepsilon, \omega, k) > 0$ and $C_1 = C_1 (B_1, B_2, \rho, \varepsilon, \omega, k) > 1$ such that
		\begin{flalign}
		\label{enkmk2}
		\Big| \mathbb{P}_{\mathfrak{S}'} \big[ \mathcal{E}_{[N - k, N + k] \times [M - k, M + k]}' = \mathcal{G} \big] - \mathbb{P}_{\mu (\rho)} \big[ \mathcal{F}_{[-k, k] \times [-k, k]} = \mathcal{G} \big] \Big| < \displaystyle\frac{\varepsilon}{3},
		\end{flalign}
		
		\noindent for any integers $N > C_1$ and $M \in \big[ \omega N, \omega^{-1} N \big]$, if $\textbf{u}$ is $(\delta N; \delta)$-regular. 
		
		Next, let $\mathcal{E}'' \in \mathfrak{E}(\mathbb{Z} \times \mathbb{Z}_{> 0})$ denote a random six-vertex ensemble sampled under the $(B_1, B_2)$-stochastic six-vertex model $\mathfrak{S}''$ on the discrete upper half-plane $\mathbb{Z} \times \mathbb{Z}_{> 0}$ with entrance data $\textbf{u} \cap J$. Then, Proposition 5.7 of \cite{LSLSSSVM} implies the existence of a constant $c_1 = c_1 (B_2) > 0$ and a coupling between $\mathcal{E}'$ and $\mathcal{E}''$ such that they coincide on $[N - M, N + M] \times [0, 2M]$ off of an event of probability at most $c_1^{-1} e^{-c_1 M}$, if $M < c_1 N$. Combined with \eqref{enkmk2}, this for sufficiently large $N$ yields
		\begin{flalign}
		\label{enkmk3}
		\Big| \mathbb{P}_{\mathfrak{S}''} \big[ \mathcal{E}_{[N - k, N + k] \times [M - k, M + k]}'' = \mathcal{G} \big] - \mathbb{P}_{\mu (\rho)} \big[ \mathcal{F}_{[-k, k] \times [-k, k]} = \mathcal{G} \big] \Big| < \displaystyle\frac{2\varepsilon}{3}, 
		\end{flalign}
		
		\noindent if $\omega N < M < c_1 N$. 
		
		Now, since no paths exist in $\mathcal{E}''$ left of the line $x = N$, we may equivalently view $\mathcal{E}''$ as sampled from the $(B_1, B_2)$-stochastic six-vertex ensemble on the quadrant $\mathbb{Z}_{> 0}^2$ with entrance data $\textbf{u} \cap J$. Thus, \Cref{modelcouple} again yields a coupling between $\mathcal{E}$ and $\mathcal{E}''$ such that they coincide on $[N - M, N + M] \times [0, 2M]$ off of an event of probability at most $c_1 e^{-c_1 M}$, if $M < c_1 N$. This, together with \eqref{enkmk3}, implies \eqref{enkmk}. 
	\end{proof}

	\section{Comparison With the Stochastic Six-Vertex Model}
	
	\label{MeasureModel}
	
	In this section we establish \Cref{hstate}, conditional on a certain statement (given by \Cref{mulowerprobability} below). To that end, we begin in \Cref{TransformStochastic} by implementing a gauge transformation to show that a pure state of any ferroelectric six-vertex model is equivalent to a pure state of a particular stochastic six-vertex model. Then in \Cref{Probability1} we introduce the notion of a partition function stochastic lower bound and state two results, given by \Cref{musmu} and \Cref{mulowerprobability}, classifying pure states satisfying this bound; we further establish \Cref{hstate} conditional on these two results. In \Cref{StochasticProbability} and \Cref{StateEstimate} we establish \Cref{musmu}; \Cref{mulowerprobability} will be established in \Cref{ProbabilityLower} below.

	\subsection{Gauge Equivalence With a Stochastic Pure State} 
	
	\label{TransformStochastic}
	
	We begin with the following result stating that the Gibbs property for the six-vertex model is invariant under a certain gauge transformation of its weights; similar results, with analogous proofs, were showed at the end of Section 3 of \cite{ICSVMT} and Appendix A.1 of \cite{CFSAEPSSVMCL} (see also Section 2 of \cite{DTSVMFFP}). 
	
	\begin{lem} 
		
		\label{wabcxir}
		
		Fix real numbers $a_1, a_2, b_1, b_2, c_1, c_2 > 0$ and $r, x, y, z > 0$. If some measure $\mu \in \mathscr{P} \big( \mathfrak{E} (\mathbb{Z}^2) \big)$ satisfies the Gibbs property for the six-vertex model with weights $(a_1, a_2, b_1, b_2, c_1, c_2)$, then it is also does for the six-vertex model with weights $(r a_1, rz a_2, ryz b_1, ry^{-1} b_2, rx c_1, rx^{-1} z c_2)$.
	\end{lem} 

	\begin{proof}
		
		Let $\mathcal{E}\in \mathfrak{E} (\mathbb{Z}^2)$ denote a random six-vertex ensemble on $\mathbb{Z}^2$ sampled under $\mu$. Fix a finite rectangular domain $\Lambda \subset \mathbb{Z}^2$, and condition on the restriction $\mathcal{E}_{\mathbb{Z}^2 \setminus \Lambda} = \mathcal{H}$ of $\mathcal{E}$ to the complement of $\Lambda$. Then $\mathcal{H}$ induces boundary conditions, denoted by $\textbf{u} \cup \textbf{v} = \textbf{u} (\mathcal{H}) \cup \textbf{v} (\mathcal{H})$, for the restriction $\mathcal{E}_{\Lambda}$ of $\mathcal{E}$ to $\Lambda$. Set $\textbf{u} = (u_{-B}, u_{1 - B}, \ldots , u_A)$ and $\textbf{v} = (v_{-B}, v_{1 - B}, \ldots , v_A)$, for some integers $A = A (\mathcal{H}) \ge 0$ and $B = B (\mathcal{H}) \ge -1$.
		
		Now, for any six-vertex ensemble $\mathcal{G} \in \mathfrak{E} (\Lambda)$ on $\Lambda$, let $N_1 (\mathcal{G})$, $N_2 (\mathcal{G})$, $N_3 (\mathcal{G})$, $N_4 (\mathcal{G})$, $N_5 (\mathcal{G})$, and $N_6 (\mathcal{G})$ denote the numbers of vertices in $\Lambda$ whose arrow configurations under $\mathcal{G}$ are given by $(0, 0; 0, 0)$, $(1, 1; 1, 1)$, $(1, 0; 1, 0)$, $(0, 1; 0, 1)$, $(1, 0; 0, 1)$, and $(0, 1; 1, 0)$, respectively. Then, since $\mu$ satisfies the Gibbs property for the six-vertex model with weights $(a_1, a_2, b_1, b_2, c_1, c_2)$, there exists a constant $Z = Z_{\mathcal{H}} > 0$ such that 
		\begin{flalign}
		\label{mulambdag123123}
		\begin{aligned}
		\mathbb{P} \big[ \mathcal{E}_{\Lambda} = \mathcal{G} \big| \mathcal{E}_{\mathbb{Z}^2 \setminus \Lambda} = \mathcal{H} \big] & = Z^{-1} a_1^{N_1} a_2^{N_2} b_1^{N_3} b_2^{N_4} c_1^{N_5} c_2^{N_6} \\
		& = Z^{-1} r^{-M_1}  x^{-M_2} y^{(M_4 - M_3) / 2} z^{(M_2 - M_3) / 2} \\
		& \qquad \times (ra_1)^{N_1} (rz a_2)^{N_2} (ryz b_1)^{N_3} (r y^{-1} b_2)^{N_4} (rx c_1)^{N_5} (rx^{-1} z c_2)^{N_6},
		\end{aligned} 
		\end{flalign}
		
		\noindent for any $\mathcal{G} \in \mathfrak{E}_{\textbf{u}; \textbf{v}} (\Lambda)$ with boundary data consistent with that of $\mathcal{H}$ (namely, $\textbf{u} \cup \textbf{v}$). Here, we have abbreviated $N_i = N_i (\mathcal{G})$ for each index $i \in [1, 6]$ and defined the $M_i = M_i (\mathcal{G})$ by 
		\begin{flalign*} 
		& M_1 = N_1 + N_2 + N_3 + N_4 + N_5 + N_6; \qquad M_2 = N_5 - N_6; \\
		& M_3 = 2 N_2 + 2 N_3 + N_5 + N_6; \qquad \qquad \qquad M_4 = 2N_2 + 2N_4 + N_5 + N_6.
		\end{flalign*}
		
		We claim that $M_1$, $M_2$, $M_3$, and $M_4$ are all independent of the six-vertex ensemble $\mathcal{G} \in \mathfrak{E}_{\textbf{u}; \textbf{v}} (\Lambda)$. Indeed, this holds for $M_1$, as it denotes the number of vertices in $\Lambda$. Furthermore, $2 N_2 + 2 N_3 + N_5 + N_6$ denotes the number of vertical edges in $\Lambda$ occupied by an arrow under $\mathcal{G}$. So, setting $\mathscr{Y} (x, y) = y$ for any $(x, y) \in \mathbb{Z}^2$, we have that $M_3 = \sum_{j = -B}^A \big( \mathscr{Y} (v_j) - \mathscr{Y} (u_j) \big)$, which only depends on $\textbf{u} \cup \textbf{v}$. Similarly, $2 N_2 + 2 N_4 + N_5 + N_6$ denotes the number of horizontal edges in $\Lambda$ occupied by an arrow under $\mathcal{G}$. So, setting $\mathscr{X} (x, y) = x$ for any $(x, y) \in \mathbb{Z}^2$, we have that $M_4 = \sum_{j = -B}^A \big( \mathscr{X} (v_j) - \mathscr{X} (u_j) \big)$, which too only depends on $\textbf{u} \cup \textbf{v}$. Moreover, letting $k = k (\mathcal{H}) \in [-B - 1, A]$ denote the maximal index of a path whose ending point $v_k$ is on the east boundary of $\Lambda$ (where we set $k = - B - 1$ if no path satisfies this property), we find that $M_2 = N_5 - N_6 = B - A + k + 1$ also is independent of $\mathcal{G}\in \mathfrak{E}_{\textbf{u}; \textbf{v}} (\Lambda)$. 
		
		Thus, the prefactor $Z^{-1} r^{-M_1}  x^{-M_2} y^{(M_4 - M_3) / 2} z^{(M_2 - M_3) / 2}$ on the right side of \eqref{mulambdag123123} only depends on $\mathcal{H}$, and so it follows that $\mu$ satisfies the Gibbs property for the six-vertex model with weights $(r a_1, rz a_2, ryz b_1, ry^{-1} b_2, rx c_1, rx^{-1} z c_2)$ on any rectangular subdomain of $\mathbb{Z}^2$. Since any finite domain $\Lambda \subset \mathbb{Z}^2$ is a subdomain of a rectangular one, we deduce that $\mu$ satisfies the Gibbs property for the six-vertex model with weights $(r a_1, rz a_2, ryz b_1, ry^{-1} b_2, rx c_1, rx^{-1} z c_2)$.
	\end{proof}

	The following corollary states that we can gauge transform the weights of any ferroelectric six-vertex model to those of a stochastic six-vertex model, while presevering the Gibbs property. Observe that the appearance of $\Delta - \sqrt{\Delta^2 - 1}$ in the below weights necessitates $\Delta \ge 1$, namely, that the six-vertex model is in its ferroelectric phase.

	\begin{cor}
		
		\label{aibicib1b2}
		
		Fix real numbers $a_1, a_2, b_1, b_2, c_1, c_2 > 0$, and set $\Delta$ as in \eqref{parameter1}. Assume that $\Delta \ge 1$ and that $b_1 b_2 < a_1 a_2$, and define 
		 \begin{flalign}
		 \label{b1b2}
		 B_1 = \big( \Delta - \sqrt{\Delta^2 - 1} \big) \sqrt{\displaystyle\frac{b_1 b_2}{a_1 a_2}}; \qquad B_2 = \big( \Delta + \sqrt{\Delta^2 - 1} \big) \sqrt{\displaystyle\frac{b_1 b_2}{a_1 a_2}}.
		 \end{flalign}	
		 
		 \noindent Then $\mu \in \mathscr{P} \big( \mathfrak{E} (\mathbb{Z}^2) \big)$ is a Gibbs measure for the six-vertex model with weights $(a_1, a_2, b_1, b_2, c_1, c_2)$ if and only if it is one for the $(B_1, B_2)$-stochastic six-vertex model. 
		
	\end{cor}

	\begin{proof}
	
	 Define the real numbers $r$, $x$, $y$, and $z$ by  
	\begin{flalign*}
	r = \displaystyle\frac{1}{a_1}; \quad  x = \displaystyle\frac{a_1}{c_1} \left( 1 - \big( \Delta - \sqrt{\Delta^2 - 1} \big) \sqrt{\displaystyle\frac{b_1 b_2}{a_1 a_2}} \right); \quad y = \big( \Delta - \sqrt{\Delta^2 - 1} \big) \sqrt{\displaystyle\frac{a_2 b_2}{a_1 b_1}}; \quad z = \displaystyle\frac{a_1}{a_2}, 
	\end{flalign*}

	\noindent which are all positive due to the bounds 
	\begin{flalign*} 
	a_1, a_2, b_1, b_2, c_1, c_2 > 0; \qquad 0 < \Delta - \sqrt{\Delta^2 - 1} \le 1 < \sqrt{\frac{a_1 a_2}{b_1 b_2}}.
	\end{flalign*} 
	
	\noindent Now the corollary follows from \Cref{wabcxir} and the fact that $(1, 1, B_1, B_2, 1 - B_1, 1 - B_2) = (ra_1, rz a_2, ryzb_1, ry^{-1} b_2, rx c_1, rx^{-1} z c_2)$.
	\end{proof}

	\begin{rem}
	
	\label{hb} 
	
	Under the choices of $0 < B_1 < B_2 < 1$ in \eqref{b1b2}, recalling $\mathfrak{h}$ from \eqref{h} and $\varphi$ from \eqref{kappafunction}, the conditions $t < \varphi (s)$ and $t = \varphi (s)$ are equivalent to $\mathfrak{h} (s, t) < 0$ and $\mathfrak{h} (s, t) = 0$, respectively.

	\end{rem} 
	
	In view of \Cref{aibicib1b2} and \Cref{hb}, we can now use the measures $\mu (\rho)$ from \Cref{Translation} to show existence of a pure state for the ferroelectric six-vertex model with any slope $(s, t) \in \partial \mathfrak{H}$. 
	
	\begin{cor}
		
		\label{murhoaibici} 
		
		Fix real numbers $a_1, a_2, b_1, b_2, c_1, c_2 > 0$; set $\Delta$ as in \eqref{parameter1}; assume that $\Delta > 1$ and $a_1 a_2 > b_1 b_2$; and let $(s, t) \in \partial \mathfrak{H}$. Then, there exists a pure state of slope $(s, t)$ for the six-vertex model with weights $(a_1, a_2, b_1, b_2, c_1, c_2)$.
		
	\end{cor}
	
	\begin{proof}
		
		Since $(s, t) \in \partial \mathfrak{H}$, we have $(s, t) \in \mathfrak{h}_1 \cup \mathfrak{h}_2$; let us first assume $(s, t) \in \mathfrak{h}_1$, so that $\mathfrak{h} (s, t) = 0$. Setting $0 < B_1 < B_2 < 1$ as in \eqref{b1b2} and defining $\varphi$ as in \eqref{kappafunction}, \Cref{hb} implies that the condition $\mathfrak{h} (s, t) = 0$ is equivalent to $t = \varphi (s)$. Then, by \Cref{murho1}, the measure $\mu (s)$ from \Cref{Translation} is a pure state of slope $\big( s, \varphi (s) \big) = (s, t)$ for the $(B_1, B_2)$-stochastic six-vertex model. Thus, \Cref{aibicib1b2} implies it is also a pure state of slope $(s, t)$ for the six-vertex model with weights $(a_1, a_2, b_1, b_2, c_1, c_2)$, establishing the corollary if $(s, t) \in \mathfrak{h}_1$. 
		
		If instead $(s, t) \in \mathfrak{h}_2$, then $(t, s) \in \mathfrak{h}_1$. Therefore, the above implies that $\mu (t)$ is a pure state of slope $(t, s)$ for the six-vertex model with weights $(a_1, a_2, b_2, b_1, c_2, c_1)$ (where we were permitted to interchange weights in the pairs $(b_2, b_1)$ and $(c_2, c_1)$, since the definitions \eqref{h} of $\mathfrak{h}$ and \eqref{b1b2} of $B_1, B_2$ only depend on $(a_1, a_2, b_1, b_2, c_1, c_2)$ through the three products $(a_1 a_2, b_1 b_2, c_1 c_2)$). Let $\widetilde{\mu} (t) \in \mathscr{P} \big( \mathfrak{E} (\mathbb{Z}^2) \big)$ denote the ``reflection'' of this measure into the line $y = x$, namely, the law of $\widetilde{\mathcal{E}} \in \mathfrak{E} (\mathbb{Z}^2)$ that is the ensemble obtained by first sampling $\mathcal{E} \in \mathfrak{E} (\mathbb{Z}^2)$ under $\mu (t)$ and then reflecting $\mathcal{E}$ into the line $y = x$. Since this reflection changes the six-vertex weights $(a_1, a_2, b_2, b_1, c_2, c_1)$ to $(a_1, a_2, b_1, b_2, c_1, c_2)$, and changes any slope $(s', t')$ of a pure state to $(t', s')$, $\widetilde{\mu} (t)$ is a pure state of slope $(s, t)$ for the six-vertex model with weights $(a_1, a_2, b_1, b_2, c_1, c_2)$. This shows the existence of a pure state for this model of any slope $(s, t) \in \mathfrak{h}_1 \cup \mathfrak{h}_2 = \partial \mathfrak{H}$, thereby establishing the corollary. 
	\end{proof}

	\subsection{Partition Function Stochastic Lower Bounds} 
	
	\label{Probability1}

	We begin this section with the following definition for when a pure state of the stochastic six-vertex model has the property that its partition function on an $N \times N$ domain is likely at least $e^{- o (N^2)}$. 
	
	\begin{definition}
		
		\label{estimateprobabilitylower}
		
		Fix real numbers $0 < B_1 < B_2 < 1$ and a translation-invariant Gibbs measure $\mu \in \mathscr{P} \big( \mathfrak{E} (\mathbb{Z}^2) \big)$ for the $(B_1, B_2)$-stochastic six-vertex model. We say that $\mu$ satisfies a \emph{partition function stochastic lower bound} if, for any real number $\delta > 0$, there exists a constant $C_0 = C_0 (\delta, \mu) > 1$ such that the following holds for any integer $N > C_0$. 
		
		Set $\Lambda_N = [1, N] \times [1, N]$, and let $\nu_N \in \mathscr{P} \big( \mathfrak{E} (\mathbb{Z}^2 \setminus \Lambda_N) \big)$ denote the marginal distribution of $\mu$ on $\mathfrak{E} (\mathbb{Z}^2 \setminus \Lambda_N)$ (induced by restricting six-vertex ensembles on $\mathbb{Z}^2$ to ones on $\mathbb{Z}^2 \setminus \Lambda_N$). Any six-vertex ensemble $\mathcal{H}\in \mathfrak{E}(\mathbb{Z}^2 \setminus \Lambda_N)$ that can be obtained as the restriction of one on $\mathbb{Z}^2$ induces boundary conditions $\textbf{u} (\mathcal{H}) \cup \textbf{v} (\mathcal{H})$ on $\Lambda$. Recalling the weights $w (\mathcal{E})$ from \eqref{eweight} for the $(B_1, B_2)$-stochastic six-vertex model and defining the partition function $Z (\mathcal{H}) = \sum_{\mathcal{E} \in \mathfrak{E}_{\textbf{u} (\mathcal{H}); \textbf{v} (\mathcal{H})}} w (\mathcal{E})$, we have 
		\begin{flalign*}
		\mathbb{P}_{\nu_N} \big[  Z (\mathcal{H}) \ge e^{-\delta N^2} \big] \ge 1 - \delta. 
		\end{flalign*}

	\end{definition} 

	Before describing the use of this notion, let us state two results. The first, which will be established in \Cref{StateEstimate} below, classifies any pure state satisfying a partition function stochastic lower bound as one of the measures $\mu (\rho)$ from \Cref{Translation}. The second, which will be established in \Cref{Lower1} below, states that any pure state with a certain slope satisfies a partition function stochastic lower bound.

	\begin{prop}
		
		\label{musmu}
			
		Fix real numbers $0 < B_1 < B_2 < 1$ and a pair $(s, t) \in (0, 1]^2$. If $\mu \in \mathscr{P} (\mathbb{Z}^2)$ is a pure state of slope $(s, t)$ for the $(B_1, B_2)$-stochastic six-vertex model that satisfies a partition function stochastic lower bound, then $\mu = \mu (s)$. 
		
	\end{prop}

	\begin{prop}
		
		\label{mulowerprobability} 
		
		Fix real numbers $0 < B_1 < B_2 < 1$ and a pair $(s, t) \in (0, 1]^2$. Let $\mu$ denote a pure state of slope $(s, t)$ for the $(B_1, B_2)$-stochastic six-vertex model. If $s \le t \le \varphi (s)$, then $\mu$ satisfies a partition function stochastic lower bound.
		
	\end{prop}

	Given \Cref{musmu} and \Cref{mulowerprobability}, we can quickly establish \Cref{hstate}. 
	
	\begin{proof}[Proof of \Cref{hstate} Assuming \Cref{musmu} and \Cref{mulowerprobability}]
		
		By \Cref{murhoaibici}, there exists a pure state of any slope $(s, t) \in \partial \mathfrak{H}$ for the six-vertex model with weights $(a_1, a_2, b_1, b_2, c_1, c_2)$. Thus, it remains to show that the pure state of this slope is unique, and that no such pure state can exist if $(s, t) \in \mathfrak{H}$. We may assume in what follows that $s, t \in (0, 1]$, since otherwise $(s, t) = (0, 0)$, and there is a unique pure state of this slope (which deterministically assigns arrow configuration $(0, 0; 0, 0)$ to each vertex of $\mathbb{Z}^2$). 
		
		So, fix $(s, t) \in \overline{\mathfrak{H}}$ with $s, t > 0$, and let $\mu$ denote a pure state of slope $(s, t)$ for the six-vertex model with weights $(a_1, a_2, b_1, b_2, c_1, c_2)$. By replacing $\mu$ with its reflection into the line $x = y$ if necessary (as in the proof of \Cref{murhoaibici}), we may assume that $t \ge s$. Then, define $0 < B_1 < B_2 < 1$ as in \eqref{b1b2} and $\varphi = \varphi_{B_1, B_2}$ as in \eqref{kappafunction}. By \Cref{aibicib1b2}, $\mu$ is a pure state for the $(B_1, B_2)$-stochastic six-vertex model with slope $(s, t)$. 
		
		Since $(s, t) \in \overline{\mathfrak{H}}$, we have that $\mathfrak{h} (s, t) \le 0$. By \Cref{hb}, this yields $t \le \varphi (s)$, and so $s \le t \le \varphi (s)$. Hence, \Cref{mulowerprobability} implies $\mu$ satisfies a partition function stochastic lower bound. Thus, by \Cref{musmu}, $\mu$ is equal to the pure state $\mu (s)$ from \Cref{Translation}, which has slope $\big( s, \varphi (s) \big) \in \mathfrak{h}_1 \subset \partial \mathfrak{H}$. This uniquely determines $\mu$ if $(s, t) \in \partial \mathfrak{H}$ and shows that $\mu$ cannot exist if $(s, t) \in \mathfrak{H}$, thereby establishing the theorem. 
	\end{proof}

	The benefit to translation-invariant Gibbs measures satisfying a partition function stochastic lower bound is that they can be compared to a stochastic six-vertex model with free exit data, which is sometimes more amenable to direct analysis since it is a Markov process. This comparison is made more precise through the following lemma essentially stating that, if $\mu$ satisfies a partition function stochastic lower bound, then an event exponentially unlikely under the stochastic six-vertex model with free exit data is also unlikely under $\mu$.

	\begin{lem} 
	
	\label{modelmodelstochastic}
	
	For any real numbers $\varepsilon, \gamma \in (0, 1]$ and $0 < B_1 < B_2 < 1$, and translation-invariant Gibbs measure $\mu \in \mathscr{P} (\mathbb{Z}^2)$ for the $(B_1, B_2)$-stochastic six-vertex model satisfying a partition function stochastic lower bound, there exists a constant $C = C(\mu, \varepsilon, \gamma, B_1, B_2) > 1$ such that the following holds for any integer $N > C$. Recall $\Lambda_N = [1, N] \times [1, N]$ and $\nu_N \in \mathscr{P} \big( \mathfrak{E}(\mathbb{Z}^2 \setminus \Lambda_N) \big)$ from \Cref{estimateprobabilitylower}; sample $\mathcal{H} \in \mathfrak{E}(\mathbb{Z}^2 \setminus \Lambda)$ under $\nu_N$; and denote its boundary data by $\textbf{\emph{u}} (\mathcal{H}) \cup \textbf{\emph{v}} (\mathcal{H})$. 
	
	Further sample random six-vertex ensembles $\mathcal{E}\in \mathfrak{E} (\mathbb{Z}^2)$ under $\mu$, and $\mathcal{F} \in \mathfrak{E}_{\textbf{\emph{u}} (\mathcal{H})} (\Lambda_N)$ under the $(B_1, B_2)$-stochastic six-vertex model $\mathfrak{S}$ on $\Lambda$ with entrance data $\textbf{\emph{u}} (\mathcal{H})$ and free exit data. Then, for any subset $\mathfrak{D} \subset \mathfrak{E} (\Lambda_N)$ such that $\mathbb{E}_{\nu_N} \big[ \mathbb{P}_{\mathfrak{S}} [\mathcal{F} \in \mathfrak{D}] \big] \le e^{-\gamma N^2}$, we have $\mathbb{P}_{\mu} [\mathcal{E} \in \mathfrak{D}] < \varepsilon$. 
	
	\end{lem}

	\begin{proof} 
		
		Throughout this proof, we abbreviate $\Lambda = \Lambda_N$, $\nu = \nu_N$ and $\textbf{u} \cup \textbf{v} = \textbf{u} (\mathcal{H}) \cup \textbf{v} (\mathcal{H})$ (the last of which is random). We further recall the partition function $Z (\mathcal{H})$ from \Cref{estimateprobabilitylower}, and define the event 
		\begin{flalign*} 
		\mathcal{A} = \mathcal{A}_N (\gamma) = \big\{ Z (\mathcal{H}) \ge e^{-\gamma N^2 / 2} \big\}.
		\end{flalign*} 
		
		\noindent Then, the fact that $\mu$ satisfies a partition function stochastic lower bound yields a constant $C_0 = C_0 (\mu, \gamma, \varepsilon) > 0$ such that $\mathbb{P}_{\nu} [\mathcal{A}] \ge 1 - \frac{\varepsilon}{2}$ holds whenever $N > C_0$. Thus, a union bound gives 
		\begin{flalign}
		\label{dmue}
		\mathbb{P}_{\mu} [\mathcal{E} \in \mathfrak{D}] \le \mathbb{P}_{\mu} \big[ \{ \mathcal{E} \in \mathfrak{D} \} \cap \mathcal{A} \big] + \mathbb{P}_{\nu} [\mathcal{A}^c] \le \mathbb{P}_{\mu} \big[ \{ \mathcal{E} \in \mathfrak{D} \} \cap \mathcal{A} \big] + \displaystyle\frac{\varepsilon}{2}. 
		\end{flalign}
		
		\noindent Next, since $\mathcal{A}$ is measurable with respect to $\mathcal{E}_{\mathbb{Z}^2 \setminus \Lambda}$, we have that 
		\begin{flalign}
		\label{probabilitymu} 
		\begin{aligned} 
		\mathbb{P}_{\mu} \big[ \{ \mathcal{E} \in \mathfrak{D} \} \cap \mathcal{A} \big] = \mathbb{E}_{\nu} \big[ \textbf{1}_{\mathcal{A}} \mathbb{P}_{\mu} [ \mathcal{E} \in \mathfrak{D} | \mathcal{E}_{\mathbb{Z}^2 \setminus \Lambda} = \mathcal{H}] \big] & = \mathbb{E}_{\nu} \Bigg[ \displaystyle\frac{\textbf{1}_{\mathcal{A}}}{Z (\mathcal{H})} \displaystyle\sum_{\mathcal{E}' \in \mathfrak{D} \cap \mathfrak{E}_{\textbf{u}; \textbf{v}} (\Lambda)} w (\mathcal{E}') \Bigg] \\
		& \le e^{\gamma N^2 / 2} \mathbb{E}_{\nu} \Bigg[ \displaystyle\sum_{\mathcal{E}' \in \mathfrak{D} \cap \mathfrak{E}_{\textbf{u}} (\Lambda)} w(\mathcal{E}') \Bigg],
		\end{aligned} 
		\end{flalign}
		
		\noindent since $Z (\mathcal{H}) \ge e^{-\gamma N^2 / 2}$ on $\mathcal{A}$. Next, due to the stochasticity of the six-vertex weights of $\mathfrak{S}$, we have the deterministic identity $Z_{\textbf{u}} (\Lambda) = \sum_{\mathcal{E}' \in \mathfrak{E}_{\textbf{u}} (\Lambda)} w (\mathcal{E}') = 1$. Thus, 
		\begin{flalign*}
		\displaystyle\sum_{\mathcal{E}' \in \mathfrak{D} \cap \mathfrak{E}_{\textbf{u}} (\Lambda)} w(\mathcal{E}') = \displaystyle\frac{1}{Z_{\textbf{u}} (\Lambda)} \displaystyle\sum_{\mathcal{E}' \in \mathfrak{D} \cap \mathfrak{E}_{\textbf{u}} (\Lambda)} w(\mathcal{E}') = \mathbb{P}_{\mathfrak{S}} [\mathcal{F} \in \mathfrak{D}].
		\end{flalign*}
		
		\noindent Taking the expectation of this bound with respect to $\nu$, and also applying \eqref{probabilitymu} and the bound $\mathbb{E}_{\nu} \big[ \mathbb{P}_{\mathfrak{S}} [\mathcal{F} \in \mathfrak{D}] \big] \le e^{-\gamma N^2}$, it follows for $N$ sufficiently large that 
		\begin{flalign*}
		\mathbb{P}_{\mu} \big[ \{ \mathcal{E} \in \mathfrak{D} \} \cap \mathcal{A} \big] \le e^{\gamma N^2 / 2} \mathbb{E}_{\nu} \big[ \mathbb{P}_{\mathfrak{S}} [\mathcal{F} \in \mathfrak{D}] \big] \le e^{-\gamma N^2 / 2} < \displaystyle\frac{\varepsilon}{2}.
		\end{flalign*}
		
		\noindent This, together with \eqref{dmue}, implies the lemma. 
	\end{proof}

	Now, to establish \Cref{musmu}, we must show that $\mathbb{E}_{\mu} [\psi] = \mathbb{E}_{\mu (s)} [\psi]$ holds, for any local function $\psi: \mathfrak{E} (\mathbb{Z}^2) \rightarrow \mathbb{R}$. To that end, we will consider shift-averages of $\psi$ over a large square grid. First, as \Cref{thetaistochastic} below, we will show the probability that this average differs non-negligibly from $\mathbb{E}_{\mu (s)} [\psi]$ decays exponentially in the grid size under the stochastic six-vertex model with free exit data. Next, by \Cref{modelmodelstochastic}, it will follow that these shift averages likely converge to $\mathbb{E}_{\mu (s)} [\psi]$ under $\mu$; see \Cref{musumpsiig} below. Then \Cref{musmu} will follow from the fact that these shift averages under $\mu$ converge to $\mathbb{E}_{\mu} [\psi]$ (by the ergodic theorem).

	\subsection{Shift-Averages Under the Stochastic Six-Vertex Model} 
	
	\label{StochasticProbability}
	
	Throughout the remainder of this paper, we fix real numbers $0 < B_1 <  B_2 < 1$ and will allow constants to depend on them, even when not explicitly mentioned. 
	
	Let us set some additional notation that will be used in this section. Suppose that we are given
	\begin{flalign*}
	s \in (0, 1]; \quad K, M \in \mathbb{Z}_{> 0}; \quad Y \in [0, M] \cap \mathbb{Z}; \quad k \in \mathbb{Z}_{> 0}; \quad \mathcal{G} \in \mathfrak{E} \big( [-k, k] \times [-k, k] \big).
	\end{flalign*}	
	
	\noindent From these parameters, define 
	\begin{flalign} 
	\label{xnlambda} 
	X = \bigg\lceil \displaystyle\frac{M}{2} \bigg\rceil; \qquad N = KM; \qquad \Lambda = \Lambda_N = [1, N] \times [1, N] \subseteq \mathbb{Z}^2.
	\end{flalign}
	
	Now let us define a partition $\Lambda = \bigcup_{i = 1}^{K^2} \Omega_i$ into $K^2$ subdomains as follows. For each index $i \in [1, K^2]$, let $j = j(i) \in [0, K - 1]$ denote the integer such that $i = jK + r + 1$ for some $r = r (i) \in [0, K - 1]$. Then define the subdomain $\Omega_i = [rM + 1, rM + M] \times [jM + 1, jM + M] \subseteq \Lambda$ and define the vertex $z_i = (rM + X, jM + Y) \in \Omega_i$. We refer to \Cref{omegai} for a depiction.

	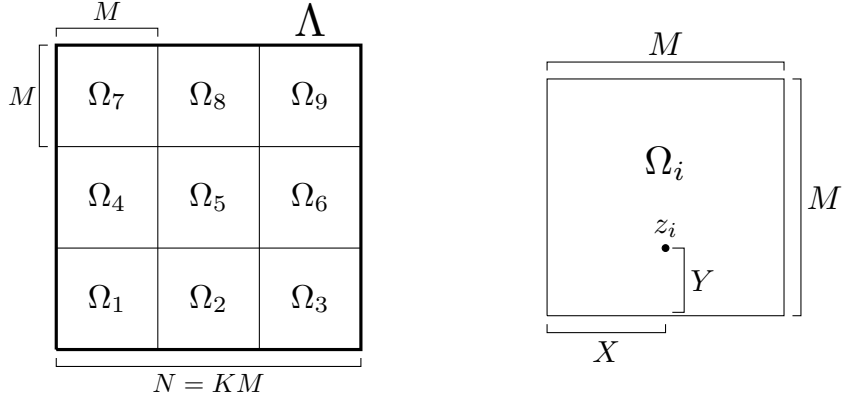
\begin{figure}[t]
		
		\begin{center}
			
			\begin{tikzpicture}[
			>=stealth,
			scale = .45
			]
			
			\draw[very thick] (0, 0) -- (9, 0) -- (9, 9) -- (0, 9) -- (0, 0);
			\draw[] (0, 3) -- (9, 3); 
			\draw[] (0, 6) -- (9, 6); 
			\draw[] (3, 0) -- (3, 9); 
			\draw[] (6, 0) -- (6, 9);
			 
			 \draw[] (1.5, 1.5) circle [radius = 0] node[scale = 1.25]{$\Omega_1$};
			 \draw[] (4.5, 1.5) circle [radius = 0] node[scale = 1.25]{$\Omega_2$};
			 \draw[] (7.5, 1.5) circle [radius = 0] node[scale = 1.25]{$\Omega_3$};
			 \draw[] (1.5, 4.5) circle [radius = 0] node[scale = 1.25]{$\Omega_4$};
			 \draw[] (4.5, 4.5) circle [radius = 0] node[scale = 1.25]{$\Omega_5$};
			 \draw[] (7.5, 4.5) circle [radius = 0] node[scale = 1.25]{$\Omega_6$};
			 \draw[] (1.5, 7.5) circle [radius = 0] node[scale = 1.25]{$\Omega_7$};
			 \draw[] (4.5, 7.5) circle [radius = 0] node[scale = 1.25]{$\Omega_8$};
			 \draw[] (7.5, 7.5) circle [radius = 0] node[scale = 1.25]{$\Omega_9$};
			 \draw[] (7.5, 9.75) circle [radius = 0] node[scale = 1.75]{$\Lambda$};
			 
			 \draw[] (0, -.25) -- (0, -.5) -- (9, -.5) -- (9, -.25);
			 \draw[] (-.25, 6) -- (-.5, 6) -- (-.5, 9) -- (-.25, 9);
			 \draw[] (0, 9.25) -- (0, 9.5) -- (3, 9.5) -- (3, 9.25);
			
			\draw[] (4.5, -1) circle [radius = 0] node[]{$N = KM$};
			\draw[] (-1, 7.5) circle [radius = 0] node[]{$M$};
			\draw[] (1.5, 10) circle [radius = 0] node[]{$M$};

			\draw[] (14.5, 1) -- (21.5, 1) -- (21.5, 8) -- (14.5, 8) -- (14.5, 1);

			\draw[] (14.5, 8.25) -- (14.5, 8.5) -- (21.5, 8.5) -- (21.5, 8.25);
			\draw[] (21.75, 1) -- (22, 1) -- (22, 8) -- (21.75, 8);
			\draw[] (14.5, .75) -- (14.5, .5) -- (18, .5) -- (18, .75);
			\draw[] (18.2, 1.1) -- (18.55, 1.1) -- (18.55, 3) -- (18.2, 3);

			\filldraw[] (18, 3) circle [radius = .1] node[above, scale = 1.15]{$z_i$};

			\draw[] (16.25, -.05) circle [radius = 0] node[scale = 1.15]{$X$};
			\draw[] (19.1, 2.05) circle [radius = 0] node[scale = 1.15]{$Y$};
			\draw[] (18, 9.05) circle [radius = 0] node[scale = 1.25]{$M$};
			\draw[] (22.7, 4.5) circle [radius = 0] node[scale = 1.25]{$M$};
			\draw[] (18, 5.5) circle [radius = 0] node[scale = 1.5]{$\Omega_i$};
			
			\end{tikzpicture}
			
		\end{center}	
		
		\caption{\label{omegai} The domains $\Omega_i$ are depicted to the left (where there $K = 3$), and the vertex $z_i \in \Omega_i$ is depicted to the right.} 
		
	\end{figure}

	Next, for any six-vertex ensemble $\mathcal{E} \in \mathfrak{E} (\Lambda)$, we will define indicator functions $\psi_i (\mathcal{E}) \in \{ 0, 1 \}$ for the event that $\mathcal{G}$ ``locally appears around'' $z_i$ in $\mathcal{E}$. More specifically, for each index $i \in [1, K^2]$, define (recalling the translation operator $\mathfrak{T}_u$ from \Cref{Measures})
	\begin{flalign}
	\label{epsii} 
	\psi_i (\mathcal{E}) = \psi_i^{(\mathcal{G})} (\mathcal{E}) = \textbf{1} \big( (\mathfrak{T}_{z_i} \mathcal{E})_{[-k, k] \times [-k, k]}= \mathcal{G} \big). 
	\end{flalign}

	\noindent Moreover, for any real number $\eta \in (0, 1)$, let $\Theta_i (\eta)$ denote the event on which the entrance data for the ensemble $\mathcal{E}_{\Omega_i}$ is $(\eta M; \eta)$-regular with slope $s$ along the south boundary of $\Omega_i$. 
	
	The following lemma provides an exponential probability concentration estimate for the sum of the $\psi_i (\mathcal{E})$ over some index set $\mathcal{I} \subseteq [1, K^2]$ for the stochastic six-vertex model with free exit data (after restricting to the event $\bigcap_{i \in \mathcal{I}} \Theta_i (\eta)$, for sufficiently small $\eta$).

	\begin{lem}
		
	\label{thetaistochastic}	
	
	For any real number $\varepsilon \in (0, 1]$, there exist constants $\delta = \delta (\varepsilon, s, k) > 0$, $c_1 (s, k) > 0$, $c_2 = c_2 (\varepsilon) > 0$, and $C = C (\varepsilon, s, k) > 1$ such that the following holds. Adopt the notation above, and assume that $M > C$ and $\frac{c_1 M}{2} < Y < c_1 M$. Let $\eta \in (0, \delta)$ denote a real number and $\mathcal{I} \subseteq [1, K^2]$ denote a nonempty subset. Further fix some entrance data $\textbf{\emph{u}}$ on $\Lambda$; let $\mathcal{E} \in \mathfrak{E}_{\textbf{\emph{u}}} (\Lambda)$ denote a random six-vertex ensemble sampled under the $(B_1, B_2)$-stochastic six-vertex model $\mathfrak{S}$ on $\Lambda$ with entrance data $\textbf{\emph{u}}$ and free exit data; and let $\mathcal{F} \in \mathfrak{E} (\mathbb{Z}^2)$ denote a random six-vertex ensemble sampled under $\mu (s)$. Then,  
	\begin{flalign}
	\label{sumpsii}
	\mathbb{P}_{\mathfrak{S}} \Bigg[ \bigg\{ \Big| \displaystyle\frac{1}{|\mathcal{I}|} \displaystyle\sum_{i \in \mathcal{I}} \psi_i (\mathcal{E}) - \mathbb{P}_{\mu (s)}\big[ \mathcal{F}_{[-k, k] \times [-k, k]} = \mathcal{G} \big] \Big| > \varepsilon \bigg\} \cap \bigcap_{i \in \mathcal{I}} \Theta_i (\eta) \Bigg] < e^{-c_2 |\mathcal{I}|}.
	\end{flalign}

	\end{lem} 	
	
	\begin{proof}
		
	Throughout this proof, we set $\zeta = \mathbb{P}_{\mu (s)} \big[ \mathcal{F}_{[-k, k] \times [-k, k]} = \mathcal{G} \big]$ and, for each index $i \in [1, K^2]$, we let $\textbf{u}^{(i)}$ denote the (random) entrance data on $\Omega_i$ for $\mathcal{E}_{\Omega_i}$. In this way, $\textbf{u}^{(i)}$ is $(\eta M; \eta)$-regular along the south boundary of $\Omega_i$ on the event $\Theta_i (\eta)$. Thus, \Cref{modellocal} yields constants $\delta = \delta (\varepsilon, s, k) > 0$, $c_1 = c_1 (s, k) > 0$, and $C_1 = C_1 (\varepsilon, s, k) > 1$ such that
	\begin{flalign}
	\label{thetaiestimatepsiizeta} 
	\begin{aligned} 
	\textbf{1}_{\Theta_i (\eta)} \Big| \mathbb{P}_{\mathfrak{S}} \big[ \psi_i (\mathcal{E}) = 1 \big|  \textbf{u}^{(i)}  \big] - \zeta  \Big| & = \textbf{1}_{\Theta_i (\eta)} \Big| \mathbb{P} \big[ (\mathfrak{T}_{z_i} \mathcal{E})_{[-k, k] \times [-k, k]} = \mathcal{G} \big|  \textbf{u}^{(i)}  \big] - \mathbb{P} \big[ \mathcal{F}_{[-k, k] \times [-k, k]} = \mathcal{G} \big]  \Big| \\
	& < \displaystyle\frac{\varepsilon}{2},
	\end{aligned} 
	\end{flalign}	
	
	\noindent holds whenever $M > C_1$, $\frac{c_1 M}{2} < Y < c_1 M$, and $\eta < \delta$.
	
	Next, \Cref{ensemblen} implies for any $i \in [1, K^2]$ that $\mathcal{E}_i$ is independent of $\bigcup_{j = 1}^{i - 1} \mathcal{E}_j$, after conditioning on $\textbf{u}^{(i)}$. Hence, we obtain from \eqref{thetaiestimatepsiizeta} that 
	\begin{flalign*}
	\textbf{1}_{\Theta_i (\eta)} \Bigg| \mathbb{P}_{\mathfrak{S}} \bigg[ \psi_i (\mathcal{E}) = 1 \bigg| \bigcup_{j = 1}^{i - 1} \mathcal{E}_j  \bigg] - \zeta  \Bigg| < \displaystyle\frac{\varepsilon}{2}.
	\end{flalign*}
	
	\noindent This, together with the Chernoff estimate (or, alternatively, the Azuma--Hoeffding inequality) for sums of $0$-$1$ Bernoulli random variables, yields a constant $c_2 = c_2 (\varepsilon) > 0$ such that 
	\begin{flalign*}
	\mathbb{P}_{\mathfrak{S}} \Bigg[ \bigg\{  \Big| \displaystyle\sum_{i \in \mathcal{I}} \psi_i (\mathcal{E}) - \zeta |\mathcal{I}| \Big| > \varepsilon |\mathcal{I}| \bigg\} \cap \bigcap_{i \in \mathcal{I}} \Theta_i (\eta) \Bigg| < e^{-c_2 |\mathcal{I}|},
	\end{flalign*}
	
	\noindent from which we deduce the lemma. 
	\end{proof}

	\subsection{Proof of \Cref{musmu}}
	
	\label{StateEstimate}
	
	The following lemma essentially states that shift-averages of the local functions $\psi_i$ from \eqref{epsii} under a pure state of slope $(s, t)$ satisfying a partition function stochastic lower bound converge to their expectations under $\mu (s)$.

	\begin{lem} 
		
	\label{musumpsiig} 
	
	Fix	 an integer $k > 0$; a real number $\varepsilon > 0$; a pair $(s, t) \in (0, 1]^2$, and a pure state $\mu \in \mathscr{P} \big( \mathfrak{E} (\mathbb{Z}^2) \big)$ of slope $(s, t)$ for the $(B_1, B_2)$-stochastic six-vertex model that satisfies a partition function stochastic lower bound in the sense of \Cref{estimateprobabilitylower}. Then, there exist $c = c (s, k) > 0$ and $C_1 = C_1 (\varepsilon, \mu, k) > 1$ so that, for any integers $M > C_1$ and $Y \in \big( \frac{cM}{2}, cM \big)$, there is a constant $C_2 = C_2 (M, \varepsilon, \mu, k) > 1$, such that the following holds for any integer $K > C_2$. 
	
	Fix a six-vertex ensemble $\mathcal{G}\in \mathfrak{E} \big( [-k, k] \times [-k, k] \big)$; set $X$, $N$, and $\Lambda$ as in \eqref{xnlambda}; and set $\psi_i$ as in \eqref{epsii}. Further let $\mathcal{E} \in \mathfrak{E}(\Lambda)$ denote a random six-vertex ensemble sampled under (the marginal on $\mathfrak{E} (\Lambda)$ of) $\mu$, and let $\mathcal{F} \in \mathfrak{E}(\mathbb{Z}^2)$ denote a random six-vertex ensemble sampled under $\mu (s)$. Then, 
	\begin{flalign}
	\label{musumpsimus}
	\mathbb{P}_{\mu} \Bigg[ \bigg| \displaystyle\frac{1}{K^2} \displaystyle\sum_{i = 1}^{K^2} \psi_i (\mathcal{E}) - \mathbb{P}_{\mu(s)} \big[ \mathcal{F}_{[-k, k] \times [-k, k]} = \mathcal{G} \big] \bigg| > \varepsilon \Bigg] < \varepsilon. 
	\end{flalign}
		
	\end{lem}
	
	\begin{proof} 
		
		For any nonempty index set $\mathcal{I}\subseteq [1, K^2]$, real number $\eta > 0$, and six-vertex ensemble $\mathcal{E}_0 \in \mathfrak{E} (\Lambda)$, define the event
		\begin{flalign}
		\label{deta}
		\mathfrak{D}_{\mathcal{I}} (\varepsilon; \eta) = \Bigg\{ \bigg| \displaystyle\frac{1}{|\mathcal{I}|} \displaystyle\sum_{i \in \mathcal{I}} \psi_i (\mathcal{E}_0) - \mathbb{P}_{\mu (s)}\big[ \mathcal{F}_{[-k, k] \times [-k, k]} = \mathcal{G} \big] \bigg| > \displaystyle\frac{\varepsilon}{2} \Bigg\} \cap \bigcap_{i \in \mathcal{I}} \Theta_i (\eta),
		\end{flalign}
		
		\noindent similar to the one appearing in \eqref{sumpsii} (where we have recalled $\Theta_i (\eta)$ from below \eqref{epsii}).
			
		Next, recalling the measure $\nu = \nu_N \in \mathscr{P} \big( \mathfrak{E}(\mathbb{Z}^2 \setminus \Lambda_N) \big)$ from \Cref{estimateprobabilitylower}, sample $\mathcal{H} \in \mathfrak{E}(\mathbb{Z}^2 \setminus \Lambda)$ under $\nu$ and denote its boundary data by $\textbf{u} (\mathcal{H}) \cup \textbf{v} (\mathcal{H})$. Let $\mathcal{E}' \in \mathfrak{E}_{\textbf{u} (\mathcal{H})} (\Lambda)$ denote a random six-vertex ensemble sampled under the $(B_1, B_2)$-stochastic six-vertex model $\mathfrak{S}$ on $\Lambda$ with entrance data $\textbf{u} (\mathcal{H})$ and free exit data. Then, \Cref{thetaistochastic} yields (after taking the expectation there over $\textbf{u}$ with respect to $\nu$) constants $\delta = \delta (\varepsilon, s, k) > 0$, $c = c (s, k) > 0$, $c_0 = c_0 (\varepsilon) > 0$, and $C = C (\varepsilon, s, k) > 1$ such that 
		\begin{flalign}
		\label{cidieta}
		\mathbb{E}_{\nu} \Big[ \mathbb{P}_{\mathfrak{S}} \big[ \mathfrak{D}_{\mathcal{I}} (\varepsilon; \eta) \big] \Big] < e^{-c_0 |\mathcal{I}|},
		\end{flalign} 
		
		\noindent for any real number $\eta \in (0, \delta)$; integers $M > C$ and $\frac{c M}{2} < Y < c M$; and nonempty set $\mathcal{I} \subseteq [1, K^2]$. 
		
		Fix parameters $\eta, M, Y$ satisfying these properties. We would eventually like to bound the probability of the events $\mathfrak{D}_{\mathcal{I}}$ but without the $\Theta_i (\eta)$ appearing in their definitions \eqref{deta}. To do this, it will be useful to show an exponential bound for the probability of the union of $\mathfrak{D}_{\mathcal{I}} (\varepsilon; \eta)$ over all $\mathcal{I}$ of sufficiently large size. To that end, further fix a real number $\omega \in \big( 0, \frac{1}{2} \big)$, later to be chosen sufficiently small. Defining the set 
		\begin{flalign*}
		\mathfrak{I} (\omega) = \big\{ \mathcal{I} \subseteq [1, K^2]: |\mathcal{I}| \ge (1 - \omega) K^2 \big\},
		\end{flalign*} 
		
		\noindent taking a union bound in \eqref{cidieta} over $\mathcal{I}\in \mathfrak{I} (\omega)$ yields 
		\begin{flalign}
		\label{omegaidestimate}
		\mathbb{E}_{\nu} \Bigg[ \mathbb{P}_{\mathfrak{S}} \bigg[ \bigcup_{\mathcal{I} \in \mathfrak{I}(\omega)} \mathfrak{D}_{\mathcal{I}} (\varepsilon; \eta) \bigg] \Bigg] < e^{c_0 (\omega - 1) K^2} \big| \mathfrak{I} (\omega) \big| < e^{c_0 (\omega - 1) K^2} \left( \displaystyle\frac{4}{\omega} \right)^{2 \omega K^2}.
		\end{flalign}
		
		\noindent Here, we have used the fact that 
		\begin{flalign*} 
		\big| \mathfrak{I} (\omega) )\big| = \displaystyle\sum_{r \ge (1 - \omega) K^2} \binom{K^2}{r} = \displaystyle\sum_{r \le \omega K^2} \binom{K^2}{r} \le \displaystyle\sum_{r \le \omega K^2} \bigg( \displaystyle\frac{4K^2}{r} \bigg)^r \le \left( \displaystyle\frac{4}{\omega} \right)^{\omega K^2} \displaystyle\sum_{r = 0}^{\infty} 2^{-r} = 2 \left( \displaystyle\frac{4}{\omega} \right)^{\omega K^2},
		\end{flalign*} 
		
		\noindent where the third statement holds since $r! \ge \big( \frac{r}{4} \big)^r$ and the fourth holds since $2 \big( \frac{4K^2}{r} \big)^r < \big( \frac{4K^2}{r + 1} \big)^{r + 1}$ for $r \in \big[0, \frac{K^2}{2} \big]$. Selecting $\omega = \omega (\varepsilon, s, k) > 0$ sufficiently small so that 
		\begin{flalign*}
		c_0 \omega + 2 \omega \log \bigg( \displaystyle\frac{4}{\omega} \bigg) < \displaystyle\frac{c_0}{2},
		\end{flalign*}
	
		\noindent we deduce from \eqref{omegaidestimate} and the fact that $N = KM$ that 
		\begin{flalign*}
		\mathbb{E}_{\nu} \Bigg[ \mathbb{P}_{\mathfrak{S}} \bigg[ \bigcup_{\mathcal{I} \in \mathfrak{I}(\omega)} \mathfrak{D}_{\mathcal{I}} (\varepsilon; \eta) \bigg] \Bigg] < e^{c_0 (\omega - 1) K^2} \big| \mathfrak{I} (\omega) \big| < e^{- c_0 N^2 / 2 M^2}. 
		\end{flalign*}
		
		\noindent Hence, by the $\gamma = \frac{c_0}{2M^2}$ case of \Cref{modelmodelstochastic}, there exists a constant $C_0 = C_0 (M, \varepsilon, \mu, k) > 1$ such that for $K > C_0$ we have 
		\begin{flalign}
		\label{omegaidmu} 
		\mathbb{P}_{\mu} \Bigg[ \bigcup_{\mathcal{I} \in \mathfrak{I}(\omega)} \mathfrak{D}_{\mathcal{I}} (\varepsilon; \eta) \Bigg] < \displaystyle\frac{\varepsilon}{2}.
		\end{flalign}

		Now, observe if 
		\begin{flalign*}
		\Bigg| \displaystyle\frac{1}{K^2} \displaystyle\sum_{i = 1}^{K^2} \psi_i (\mathcal{E}) - \mathbb{P}_{\mathfrak{S}} \big[ \mathcal{F}_{[-k, k] \times [-k, k]} = \mathcal{G} \big] \Bigg| > \varepsilon,
		\end{flalign*}
		
		\noindent then for any $\mathcal{I} \in \mathfrak{I} (\omega)$ we have
		\begin{flalign*}
		 \Bigg| \displaystyle\sum_{i \in \mathcal{I}} \psi_i (\mathcal{E}) - |\mathcal{I}| \mathbb{P}_{\mu (s)} \big[ \mathcal{F}_{[-k, k] \times [-k, k]} = \mathcal{G} \big] \Bigg| & \ge \Bigg| \displaystyle\sum_{i = 1}^{K^2} \psi_i (\mathcal{E}) - K^2 \mathbb{P}_{\mu (s)} \big[ \mathcal{F}_{[-k, k] \times [-k, k]} = \mathcal{G} \big] \Bigg| - 2 \omega K^2 \\
		 & \ge (\varepsilon - 2 \omega) K^2 \ge (\varepsilon - 2 \omega) |\mathcal{I}|. 
		\end{flalign*}
		
		\noindent Hence, if we further select $\omega = \omega (\varepsilon, s, k) > 0$ sufficiently small so that $\omega < \frac{\varepsilon}{4}$, then
		\begin{flalign*}
		\Bigg\{ \bigg| \displaystyle\frac{1}{K^2} \displaystyle\sum_{i = 1}^{K^2} \psi_i (\mathcal{E}) - \mathbb{P}_{\mathfrak{S}} \big[ \mathcal{F}_{[-k, k] \times [-k, k]} = \mathcal{G} \big] \bigg| > \varepsilon \Bigg\} \cap \bigcup_{\mathcal{I} \in \mathfrak{I} (\omega)} \bigcap_{i \in \mathcal{I}} \Theta_i (\eta) \subseteq \bigcup_{\mathcal{I} \in \mathfrak{I}(\omega)} \mathfrak{D}_{\mathcal{I}} (\varepsilon; \eta), 
		\end{flalign*}
		
		\noindent and so \eqref{omegaidmu} yields 
		\begin{flalign*}
		\mathbb{P}_{\mu} \Bigg[ \bigg\{ \Big| \displaystyle\frac{1}{K^2} \displaystyle\sum_{i = 1}^{K^2} \psi_i (\mathcal{E}) - \mathbb{P}_{\mathfrak{S}} \big[ \mathcal{F}_{[-k, k] \times [-k, k]} = \mathcal{G} \big] \Big| > \varepsilon \bigg\}	 \cap \bigcup_{\mathcal{I} \in \mathfrak{I} (\omega)} \bigcap_{i \in \mathcal{I}} \Theta_i (\eta) \Bigg] < \displaystyle\frac{\varepsilon}{2}.
		\end{flalign*}

		\noindent Therefore, by a union bound, it suffices to show for sufficiently large $M$ that 
		\begin{flalign*}
		\mathbb{P}_{\mu} \Bigg[ \bigcup_{\mathcal{I} \in \mathfrak{I} (\omega)} \bigcap_{i \in \mathcal{I}} \Theta_i (\eta) \Bigg] \ge 1 - \displaystyle\frac{\varepsilon}{2}, 
		\end{flalign*}
		
		\noindent or equivalently that 
		\begin{flalign}
		\label{dmuomegathetaieta}
		\mathbb{P}_{\mu} \Bigg[ \displaystyle\sum_{i = 1}^{K^2} \textbf{1}_{\Theta_i (\eta)} \ge (1 - \omega) K^2 \Bigg] \ge 1 - \displaystyle\frac{\varepsilon}{2}. 
		\end{flalign}
		
		\noindent To that end, first observe that \Cref{rhosmu} yields a constant $C_1 = C_1 (\varepsilon, \mu, s, k) > 1$ such that $\mathbb{P} \big[\Theta_i (\eta)^c \big] < \frac{\varepsilon \omega}{2}$ holds for any index $i \in [1, K^2]$ whenever $M > C_1$ (since $\eta$ and $\omega$ only depend on $\varepsilon$, $s$, and $k$). Thus, a Markov estimate yields 
		\begin{flalign*} 
		\mathbb{P}_{\mu} \Bigg[ \sum_{i = 1}^{K^2} \textbf{1}_{\Theta_i (\eta)} \le (1 - \omega) K^2 \Bigg] & = \mathbb{P}_{\mu} \Bigg[ \sum_{i = 1}^{K^2} \textbf{1}_{\Theta_i (\eta)^c} \ge \omega K^2 \Bigg] \\
		& \le \displaystyle\frac{1}{\omega K^2}  \mathbb{E}_{\mu} \Bigg[ \sum_{i = 1}^{K^2} \textbf{1}_{\Theta_i (\eta)^c} \Bigg] = \displaystyle\frac{1}{\omega K^2} \displaystyle\sum_{i = 1}^{K^2} \mathbb{P}_{\mu} \big[ \Theta_i (\eta)^c \big] \le \displaystyle\frac{\varepsilon}{2}, 
		\end{flalign*} 
		
		\noindent which implies \eqref{dmuomegathetaieta} and therefore the lemma. 		
	\end{proof}

	Now we can quickly establish \Cref{musmu}.

	\begin{proof}[Proof of \Cref{musmu}]
		
		Let $\mathcal{E}, \mathcal{F} \in \mathfrak{E} (\mathbb{Z}^2)$ denote random six-vertex ensembles sampled under $\mu$ and $\mu (s)$, respectively. It suffices to show that, for any integer $k > 0$ and six-vertex ensemble $\mathcal{G} \in \mathfrak{E} \big( [-k, k] \times [-k, k] \big)$, we have 
		\begin{flalign}
		\label{muemusf} 
		\mathbb{P}_{\mu} \big[ \mathcal{E}_{[-k, k] \times [-k, k]} = \mathcal{G} \big] = \mathbb{P}_{\mu (s)} \big[ \mathcal{F}_{[-k, k] \times [-k, k]} = \mathcal{G} \big].
		\end{flalign}
		
		To that end, fix $k \in \mathbb{Z}_{> 0}$ and $\mathcal{G} \in \mathfrak{E} \big( [-k, k] \times [-k, k] \big)$, and recall the function $\psi_i (\mathcal{E})$ from \eqref{epsii}. Since $(M \mathbb{Z})^2$ is an amenable group and $\mu$ is invariant under its action, it follows from the pointwise ergodic theorem for amenable group actions (see, for instance, part (ii) of Theorem 3.3 of \cite{ETSSAG}) that, for any integer $M > 0$, the limit
		\begin{flalign*}
		H(M) = \displaystyle\lim_{K \rightarrow \infty} \displaystyle\frac{1}{K^2} \displaystyle\sum_{i = 1}^{K^2} \psi_i (\mathcal{E}),
		\end{flalign*} 
		
		\noindent exists almost surely under $\mu$, and its expectation is given by 
		\begin{flalign}
		\label{hmfg}
		\mathbb{E}_{\mu} \big[ H(M) \big] = \mathbb{P}_{\mu} \big[ \mathcal{E}_{[-k, k] \times [-k, k]} = \mathcal{G} \big].
		\end{flalign}
		
		Now, \Cref{musumpsiig} implies for any $\varepsilon > 0$ that 
		\begin{flalign*}
		\displaystyle\lim_{M \rightarrow \infty} \mathbb{P}_{\mu} \bigg[ \Big| H(M) - \mathbb{P}_{\mu (s)} \big[ \mathcal{F}_{[-k, k] \times [-k, k]} = \mathcal{G} \big] \Big| > \varepsilon \bigg] = 0,
		\end{flalign*}
		
		\noindent from which \eqref{muemusf} follows by taking expectation and applying \eqref{hmfg}. 
	\end{proof}

	\section{Partition Function Estimates}
	
	\label{ProbabilityLower} 
	
	In this section we establish \Cref{mulowerprobability}, which we do in \Cref{Lower1}, after introducing a sparsification procedure in \Cref{RestrictionLK} and an ensemble extension property in \Cref{LowerProbability} that will be used in its proof. Throughout this section, we fix real numbers $0 < B_1 <  B_2 < 1$ and will allow constants to depend on them, even when not explicitly mentioned. 
	
	\subsection{\texorpdfstring{$(L; K)$}{}-Restrictions}
	
	\label{RestrictionLK}
	
	In this section we introduce and describe properties of a certain way of ``sparsifying'' six-vertex ensembles, which we refer to as $(L; K)$-restriction. This procedure has the benefit of simultaneously altering the slope of a six-vertex ensemble (see \Cref{restrictionregular} below), while not reducing the associated partition function on an $N \times N$ square by more than $e^{- o(N^2)}$, assuming $K \gg 1$ (see \Cref{rholrho} below). This will eventually enable us in the proof of \Cref{mulowerprobability} to compare partition functions of a pure state of slope $(s, t) \in \overline{\mathfrak{H}}$ to one from \Cref{Translation} of slope $(s_0, t_0) \in \partial \mathfrak{H}$, whose partition function equals to $1$ (as it is induced by a stochastic model with free exit data). 
	
	This $(L; K)$-restriction procedure essentially removes $K - L$ out of every $K$ consecutive paths in a six-vertex ensemble (and retains the remaining $L$ ones). This is made more precise through the following definition. 

		\begin{figure}[t]
		
		\begin{center}
			
			\begin{tikzpicture}[
			>=stealth,
			scale = .45
			]

			\draw[->, thick] (0, 10) -- (.9, 10);
			\draw[->, thick] (1, 10) -- (1.9, 10);
			\draw[->, thick] (2, 10) -- (2.9, 10);
			\draw[->, thick] (3, 10) -- (3.9, 10);
			\draw[->, thick] (4, 10) -- (4.9, 10);
			\draw[->, thick] (5, 10) -- (5, 10.75);
			
			\draw[->, thick] (0, 7) -- (.9, 7);
			\draw[->, thick] (1, 7) -- (1.9, 7);
			\draw[->, thick] (2, 7) -- (2.9, 7);
			\draw[->, thick] (3, 7) -- (3, 7.9);
			\draw[->, thick] (3, 8) -- (3.9, 8);
			\draw[->, thick] (4, 8) -- (4.9, 8);
			\draw[->, thick] (5, 8) -- (5, 8.9); 
			\draw[->, thick] (5, 9) -- (5.9, 9);
			\draw[->, thick] (6, 9) -- (6, 9.9); 
			\draw[->, thick] (6, 10) -- (6.9, 10);
			\draw[->, thick] (7, 10) -- (7, 10.75);
			
			\draw[->, thick] (0, 4) -- (.9, 4);
			\draw[->, thick] (1, 4) -- (1.9, 4);
			\draw[->, thick] (2, 4) -- (2, 4.9);
			\draw[->, thick] (2, 5) -- (2.9, 5);
			\draw[->, thick] (3, 5) -- (3, 5.9);
			\draw[->, thick] (3, 6) -- (3.9, 6);
			\draw[->, thick] (4, 6) -- (4.9, 6);
			\draw[->, thick] (5, 6) -- (5.9, 6);
			\draw[->, thick] (6, 6) -- (6, 6.9);
			\draw[->, thick] (6, 7) -- (6, 7.9);
			\draw[->, thick] (6, 8) -- (6.9, 8); 
			\draw[->, thick] (7, 8) -- (7, 8.9);
			\draw[->, thick] (7, 9) -- (7.9, 9); 
			\draw[->, thick] (8, 9) -- (8, 9.9);
			\draw[->, thick] (8, 10) -- (8.9, 10); 
			\draw[->, thick] (9, 10) -- (9, 10.75);

			\draw[->, thick] (0, 1) -- (.9, 1);
			\draw[->, thick] (1, 1) -- (1.9, 1);
			\draw[->, thick] (2, 1) -- (2.9, 1);
			\draw[->, thick] (3, 1) -- (3, 1.9);
			\draw[->, thick] (3, 2) -- (3.9, 2);
			\draw[->, thick] (4, 2) -- (4, 2.9);
			\draw[->, thick] (4, 3) -- (4, 3.9);
			\draw[->, thick] (4, 4) -- (4, 4.9);
			\draw[->, thick] (4, 5) -- (4.9, 5);
			\draw[->, thick] (5, 5) -- (5.9, 5); 
			\draw[->, thick] (6, 5) -- (6, 5.9); 
			\draw[->, thick] (6, 6) -- (6.9, 6);
			\draw[->, thick] (7, 6) -- (7.9, 6); 
			\draw[->, thick] (8, 6) -- (8, 6.9);
			\draw[->, thick] (8, 7) -- (8, 7.9);
			\draw[->, thick] (8, 8) -- (8, 8.9);
			\draw[->, thick] (8, 9) -- (8.9, 9);
			\draw[->, thick] (9, 9) -- (9.9, 9);
			\draw[->, thick] (10, 9) -- (10, 9.9);
			\draw[->, thick] (10, 10) -- (10.75, 10);

			\draw[->, thick] (4, 0) -- (4, .9); 
			\draw[->, thick] (4, 1) -- (4, 1.9);
			\draw[->, thick] (4, 2) -- (4.9, 2); 
			\draw[->, thick] (5, 2) -- (5.9, 2);
			\draw[->, thick] (6, 2) -- (6, 2.9); 
			\draw[->, thick] (6, 3) -- (6.9, 3); 
			\draw[->, thick] (7, 3) -- (7, 3.9);
			\draw[->, thick] (7, 4) -- (7, 4.9);
			\draw[->, thick] (7, 5) -- (7.9, 5); 
			\draw[->, thick] (8, 5) -- (8, 5.9);
			\draw[->, thick] (8, 6) -- (8.9, 6); 
			\draw[->, thick] (9, 6) -- (9, 6.9);
			\draw[->, thick] (9, 7) -- (9, 7.9);
			\draw[->, thick] (9, 8) -- (9.9, 8);
			\draw[->, thick] (10, 8) -- (10, 8.9);
			\draw[->, thick] (10, 9) -- (10.75, 9);

			\draw[->, thick] (5, 0) -- (5, .9);
			\draw[->, thick] (5, 1) -- (5.9, 1);
			\draw[->, thick] (6, 1) -- (6.9, 1);
			\draw[->, thick] (7, 1) -- (7, 1.9);
			\draw[->, thick] (7, 2) -- (7, 2.9);
			\draw[->, thick] (7, 3) -- (7.9, 3);
			\draw[->, thick] (8, 3) -- (8, 3.9);
			\draw[->, thick] (8, 4) -- (8.9, 4);
			\draw[->, thick] (9, 4) -- (9, 4.9);
			\draw[->, thick] (9, 5) -- (9, 5.9);
			\draw[->, thick] (9, 6) -- (9.9, 6);
			\draw[->, thick] (10, 6) -- (10, 6.9);
			\draw[->, thick] (10, 7) -- (10, 7.9);
			\draw[->, thick] (10, 8) -- (10.75, 8);
			
			\draw[->, thick] (7, 0) -- (7, .9); 
			\draw[->, thick] (7, 1) -- (7.9, 1); 
			\draw[->, thick] (8, 1) -- (8, 1.9);
			\draw[->, thick] (8, 2) -- (8.9, 2);
			\draw[->, thick] (9, 2) -- (9, 2.9); 
			\draw[->, thick] (9, 3) -- (9, 3.9); 
			\draw[->, thick] (9, 4) -- (9.9, 4);
			\draw[->, thick] (10, 4) -- (10, 4.9);
			\draw[->, thick] (10, 5) -- (10, 5.9);
			\draw[->, thick] (10, 6) -- (10.75, 6);

			\draw[->, thick] (9, 0) -- (9, .9);
			\draw[->, thick] (9, 1) -- (9.9, 1);
			\draw[->, thick] (10, 1) -- (10, 1.9); 
			\draw[->, thick] (10, 2) -- (10, 2.9); 
			\draw[->, thick] (10, 3) -- (10.75, 3);

			\filldraw[fill=gray!50!black, draw=black] (0, 1) circle [radius=.25];
			\filldraw[fill=gray!50!white, draw=black] (0, 2) circle [radius=.1];
			\filldraw[fill=gray!50!white, draw=black] (0, 3) circle [radius=.1];
			\filldraw[fill=gray!50!black, draw=black] (0, 4) circle [radius=.25];
			\filldraw[fill=gray!50!white, draw=black] (0, 5) circle [radius=.1];
			\filldraw[fill=gray!50!white, draw=black] (0, 6) circle [radius=.1];
			\filldraw[fill=gray!50!black, draw=black] (0, 7) circle [radius=.25];
			\filldraw[fill=gray!50!white, draw=black] (0, 8) circle [radius=.1];
			\filldraw[fill=gray!50!white, draw=black] (0, 9) circle [radius=.1];
			\filldraw[fill=gray!50!black, draw=black] (0, 10) circle [radius=.25];

			\filldraw[fill=gray!50!white, draw=black] (1, 0) circle [radius=.1];
			\filldraw[fill=gray!50!white, draw=black] (1, 1) circle [radius=.1];
			\filldraw[fill=gray!50!white, draw=black] (1, 2) circle [radius=.1];
			\filldraw[fill=gray!50!white, draw=black] (1, 3) circle [radius=.1];
			\filldraw[fill=gray!50!white, draw=black] (1, 4) circle [radius=.1];
			\filldraw[fill=gray!50!white, draw=black] (1, 5) circle [radius=.1];
			\filldraw[fill=gray!50!white, draw=black] (1, 6) circle [radius=.1];
			\filldraw[fill=gray!50!white, draw=black] (1, 7) circle [radius=.1];
			\filldraw[fill=gray!50!white, draw=black] (1, 8) circle [radius=.1];
			\filldraw[fill=gray!50!white, draw=black] (1, 9) circle [radius=.1];
			\filldraw[fill=gray!50!white, draw=black] (1, 10) circle [radius=.1];
			\filldraw[fill=gray!50!white, draw=black] (1, 11) circle [radius=.1];
			
			\filldraw[fill=gray!50!white, draw=black] (2, 0) circle [radius=.1];
			\filldraw[fill=gray!50!white, draw=black] (2, 1) circle [radius=.1];
			\filldraw[fill=gray!50!white, draw=black] (2, 2) circle [radius=.1];
			\filldraw[fill=gray!50!white, draw=black] (2, 3) circle [radius=.1];
			\filldraw[fill=gray!50!white, draw=black] (2, 4) circle [radius=.1];
			\filldraw[fill=gray!50!white, draw=black] (2, 5) circle [radius=.1];
			\filldraw[fill=gray!50!white, draw=black] (2, 6) circle [radius=.1];
			\filldraw[fill=gray!50!white, draw=black] (2, 7) circle [radius=.1];
			\filldraw[fill=gray!50!white, draw=black] (2, 8) circle [radius=.1];
			\filldraw[fill=gray!50!white, draw=black] (2, 9) circle [radius=.1];
			\filldraw[fill=gray!50!white, draw=black] (2, 10) circle [radius=.1];
			\filldraw[fill=gray!50!white, draw=black] (2, 11) circle [radius=.1];
			
			\filldraw[fill=gray!50!white, draw=black] (3, 0) circle [radius=.1];
			\filldraw[fill=gray!50!white, draw=black] (3, 1) circle [radius=.1];
			\filldraw[fill=gray!50!white, draw=black] (3, 2) circle [radius=.1];
			\filldraw[fill=gray!50!white, draw=black] (3, 3) circle [radius=.1];
			\filldraw[fill=gray!50!white, draw=black] (3, 4) circle [radius=.1];
			\filldraw[fill=gray!50!white, draw=black] (3, 5) circle [radius=.1];
			\filldraw[fill=gray!50!white, draw=black] (3, 6) circle [radius=.1];
			\filldraw[fill=gray!50!white, draw=black] (3, 7) circle [radius=.1];
			\filldraw[fill=gray!50!white, draw=black] (3, 8) circle [radius=.1];
			\filldraw[fill=gray!50!white, draw=black] (3, 9) circle [radius=.1];
			\filldraw[fill=gray!50!white, draw=black] (3, 10) circle [radius=.1];
			\filldraw[fill=gray!50!white, draw=black] (3, 11) circle [radius=.1];
			
			\filldraw[fill=gray!50!black, draw=black] (4, 0) circle [radius=.25];
			\filldraw[fill=gray!50!white, draw=black] (4, 1) circle [radius=.1];
			\filldraw[fill=gray!50!white, draw=black] (4, 2) circle [radius=.1];
			\filldraw[fill=gray!50!white, draw=black] (4, 3) circle [radius=.1];
			\filldraw[fill=gray!50!white, draw=black] (4, 4) circle [radius=.1];
			\filldraw[fill=gray!50!white, draw=black] (4, 5) circle [radius=.1];
			\filldraw[fill=gray!50!white, draw=black] (4, 6) circle [radius=.1];
			\filldraw[fill=gray!50!white, draw=black] (4, 7) circle [radius=.1];
			\filldraw[fill=gray!50!white, draw=black] (4, 8) circle [radius=.1];
			\filldraw[fill=gray!50!white, draw=black] (4, 9) circle [radius=.1];
			\filldraw[fill=gray!50!white, draw=black] (4, 10) circle [radius=.1];
			\filldraw[fill=gray!50!white, draw=black] (4, 11) circle [radius=.1];
			
			\filldraw[fill=gray!50!black, draw=black] (5, 0) circle [radius=.25];
			\filldraw[fill=gray!50!white, draw=black] (5, 1) circle [radius=.1];
			\filldraw[fill=gray!50!white, draw=black] (5, 2) circle [radius=.1];
			\filldraw[fill=gray!50!white, draw=black] (5, 3) circle [radius=.1];
			\filldraw[fill=gray!50!white, draw=black] (5, 4) circle [radius=.1];
			\filldraw[fill=gray!50!white, draw=black] (5, 5) circle [radius=.1];
			\filldraw[fill=gray!50!white, draw=black] (5, 6) circle [radius=.1];
			\filldraw[fill=gray!50!white, draw=black] (5, 7) circle [radius=.1];
			\filldraw[fill=gray!50!white, draw=black] (5, 8) circle [radius=.1];
			\filldraw[fill=gray!50!white, draw=black] (5, 9) circle [radius=.1];
			\filldraw[fill=gray!50!white, draw=black] (5, 10) circle [radius=.1];
			\filldraw[fill=gray!50!black, draw=black] (5, 11) circle [radius=.25];
			
			\filldraw[fill=gray!50!white, draw=black] (6, 0) circle [radius=.1];
			\filldraw[fill=gray!50!white, draw=black] (6, 1) circle [radius=.1];
			\filldraw[fill=gray!50!white, draw=black] (6, 2) circle [radius=.1];
			\filldraw[fill=gray!50!white, draw=black] (6, 3) circle [radius=.1];
			\filldraw[fill=gray!50!white, draw=black] (6, 4) circle [radius=.1];
			\filldraw[fill=gray!50!white, draw=black] (6, 5) circle [radius=.1];
			\filldraw[fill=gray!50!white, draw=black] (6, 6) circle [radius=.1];
			\filldraw[fill=gray!50!white, draw=black] (6, 7) circle [radius=.1];
			\filldraw[fill=gray!50!white, draw=black] (6, 8) circle [radius=.1];
			\filldraw[fill=gray!50!white, draw=black] (6, 9) circle [radius=.1];
			\filldraw[fill=gray!50!white, draw=black] (6, 10) circle [radius=.1];
			\filldraw[fill=gray!50!white, draw=black] (6, 11) circle [radius=.1];

			\filldraw[fill=gray!50!black, draw=black] (7, 0) circle [radius=.25];
			\filldraw[fill=gray!50!white, draw=black] (7, 1) circle [radius=.1];
			\filldraw[fill=gray!50!white, draw=black] (7, 2) circle [radius=.1];
			\filldraw[fill=gray!50!white, draw=black] (7, 3) circle [radius=.1];
			\filldraw[fill=gray!50!white, draw=black] (7, 4) circle [radius=.1];
			\filldraw[fill=gray!50!white, draw=black] (7, 5) circle [radius=.1];
			\filldraw[fill=gray!50!white, draw=black] (7, 6) circle [radius=.1];
			\filldraw[fill=gray!50!white, draw=black] (7, 7) circle [radius=.1];
			\filldraw[fill=gray!50!white, draw=black] (7, 8) circle [radius=.1];
			\filldraw[fill=gray!50!white, draw=black] (7, 9) circle [radius=.1];
			\filldraw[fill=gray!50!white, draw=black] (7, 10) circle [radius=.1];
			\filldraw[fill=gray!50!black, draw=black] (7, 11) circle [radius=.25];
			
			\filldraw[fill=gray!50!white, draw=black] (8, 0) circle [radius=.1];
			\filldraw[fill=gray!50!white, draw=black] (8, 1) circle [radius=.1];
			\filldraw[fill=gray!50!white, draw=black] (8, 2) circle [radius=.1];
			\filldraw[fill=gray!50!white, draw=black] (8, 3) circle [radius=.1];
			\filldraw[fill=gray!50!white, draw=black] (8, 4) circle [radius=.1];
			\filldraw[fill=gray!50!white, draw=black] (8, 5) circle [radius=.1];
			\filldraw[fill=gray!50!white, draw=black] (8, 6) circle [radius=.1];
			\filldraw[fill=gray!50!white, draw=black] (8, 7) circle [radius=.1];
			\filldraw[fill=gray!50!white, draw=black] (8, 8) circle [radius=.1];
			\filldraw[fill=gray!50!white, draw=black] (8, 9) circle [radius=.1];
			\filldraw[fill=gray!50!white, draw=black] (8, 10) circle [radius=.1];
			\filldraw[fill=gray!50!white, draw=black] (8, 11) circle [radius=.1];
			
			\filldraw[fill=gray!50!black, draw=black] (9, 0) circle [radius=.25];
			\filldraw[fill=gray!50!white, draw=black] (9, 1) circle [radius=.1];
			\filldraw[fill=gray!50!white, draw=black] (9, 2) circle [radius=.1];
			\filldraw[fill=gray!50!white, draw=black] (9, 3) circle [radius=.1];
			\filldraw[fill=gray!50!white, draw=black] (9, 4) circle [radius=.1];
			\filldraw[fill=gray!50!white, draw=black] (9, 5) circle [radius=.1];
			\filldraw[fill=gray!50!white, draw=black] (9, 6) circle [radius=.1];
			\filldraw[fill=gray!50!white, draw=black] (9, 7) circle [radius=.1];
			\filldraw[fill=gray!50!white, draw=black] (9, 8) circle [radius=.1];
			\filldraw[fill=gray!50!white, draw=black] (9, 9) circle [radius=.1];
			\filldraw[fill=gray!50!white, draw=black] (9, 10) circle [radius=.1];
			\filldraw[fill=gray!50!black, draw=black] (9, 11) circle [radius=.25];
			
			\filldraw[fill=gray!50!white, draw=black] (10, 0) circle [radius=.1];
			\filldraw[fill=gray!50!white, draw=black] (10, 1) circle [radius=.1];
			\filldraw[fill=gray!50!white, draw=black] (10, 2) circle [radius=.1];
			\filldraw[fill=gray!50!white, draw=black] (10, 3) circle [radius=.1];
			\filldraw[fill=gray!50!white, draw=black] (10, 4) circle [radius=.1];
			\filldraw[fill=gray!50!white, draw=black] (10, 5) circle [radius=.1];
			\filldraw[fill=gray!50!white, draw=black] (10, 6) circle [radius=.1];
			\filldraw[fill=gray!50!white, draw=black] (10, 7) circle [radius=.1];
			\filldraw[fill=gray!50!white, draw=black] (10, 8) circle [radius=.1];
			\filldraw[fill=gray!50!white, draw=black] (10, 9) circle [radius=.1];
			\filldraw[fill=gray!50!white, draw=black] (10, 10) circle [radius=.1];
			\filldraw[fill=gray!50!white, draw=black] (10, 11) circle [radius=.1];
			
			\filldraw[fill=gray!50!white, draw=black] (11, 1) circle [radius=.1];
			\filldraw[fill=gray!50!white, draw=black] (11, 2) circle [radius=.1];
			\filldraw[fill=gray!50!black, draw=black] (11, 3) circle [radius=.25];
			\filldraw[fill=gray!50!white, draw=black] (11, 4) circle [radius=.1];
			\filldraw[fill=gray!50!white, draw=black] (11, 5) circle [radius=.1];
			\filldraw[fill=gray!50!black, draw=black] (11, 6) circle [radius=.25];
			\filldraw[fill=gray!50!white, draw=black] (11, 7) circle [radius=.1];
			\filldraw[fill=gray!50!black, draw=black] (11, 8) circle [radius=.25];
			\filldraw[fill=gray!50!black, draw=black] (11, 9) circle [radius=.25];
			\filldraw[fill=gray!50!black, draw=black] (11, 10) circle [radius=.25];

			\filldraw[fill=gray!50!black, draw=black] (11, 12) circle [radius=0] node[scale = 1.5]{$\mathcal{E}$};

			\draw[] (-.5, 1) -- (-.85, 1) -- (-.85, 10) -- (-.5, 10);
			\draw[] (-1.35, 5.5) circle [radius = 0] node[]{$N$};
			
			\draw[] (1, -.5) -- (1, -.85) -- (10, -.85) -- (10, -.5);
			\draw[] (5.5, -1.35) circle [radius = 0] node[]{$N$};

			\draw[->, thick] (15, 10) -- (15.9, 10);
			\draw[->, thick] (16, 10) -- (16.9, 10);
			\draw[->, thick] (17, 10) -- (17.9, 10);
			\draw[->, thick] (18, 10) -- (18.9, 10);
			\draw[->, thick] (19, 10) -- (19.9, 10);
			\draw[->, thick] (20, 10) -- (20, 10.75);
			
			\draw[->, thick] (15, 4) -- (15.9, 4);
			\draw[->, thick] (16, 4) -- (16.9, 4);
			\draw[->, thick] (17, 4) -- (17, 4.9);
			\draw[->, thick] (17, 5) -- (17.9, 5);
			\draw[->, thick] (18, 5) -- (18, 5.9);
			\draw[->, thick] (18, 6) -- (18.9, 6);
			\draw[->, thick] (19, 6) -- (19.9, 6);
			\draw[->, thick] (20, 6) -- (20.9, 6);
			\draw[->, thick] (21, 6) -- (21, 6.9);
			\draw[->, thick] (21, 7) -- (21, 7.9);
			\draw[->, thick] (21, 8) -- (21.9, 8); 
			\draw[->, thick] (22, 8) -- (22, 8.9);
			\draw[->, thick] (22, 9) -- (22.9, 9); 
			\draw[->, thick] (23, 9) -- (23, 9.9);
			\draw[->, thick] (23, 10) -- (23.9, 10); 
			\draw[->, thick] (24, 10) -- (24, 10.75);

			\draw[->, thick] (15, 1) -- (15.9, 1);
			\draw[->, thick] (16, 1) -- (16.9, 1);
			\draw[->, thick] (17, 1) -- (17.9, 1);
			\draw[->, thick] (18, 1) -- (18, 1.9);
			\draw[->, thick] (18, 2) -- (18.9, 2);
			\draw[->, thick] (19, 2) -- (19, 2.9);
			\draw[->, thick] (19, 3) -- (19, 3.9);
			\draw[->, thick] (19, 4) -- (19, 4.9);
			\draw[->, thick] (19, 5) -- (19.9, 5);
			\draw[->, thick] (20, 5) -- (20.9, 5); 
			\draw[->, thick] (21, 5) -- (21, 5.9); 
			\draw[->, thick] (21, 6) -- (21.9, 6);
			\draw[->, thick] (22, 6) -- (22.9, 6); 
			\draw[->, thick] (23, 6) -- (23, 6.9);
			\draw[->, thick] (23, 7) -- (23, 7.9);
			\draw[->, thick] (23, 8) -- (23, 8.9);
			\draw[->, thick] (23, 9) -- (23.9, 9);
			\draw[->, thick] (24, 9) -- (24.9, 9);
			\draw[->, thick] (25, 9) -- (25, 9.9);
			\draw[->, thick] (25, 10) -- (25.75, 10);

			\draw[->, thick] (20, 0) -- (20, .9);
			\draw[->, thick] (20, 1) -- (20.9, 1);
			\draw[->, thick] (21, 1) -- (21.9, 1);
			\draw[->, thick] (22, 1) -- (22, 1.9);
			\draw[->, thick] (22, 2) -- (22, 2.9);
			\draw[->, thick] (22, 3) -- (22.9, 3);
			\draw[->, thick] (23, 3) -- (23, 3.9);
			\draw[->, thick] (23, 4) -- (23.9, 4);
			\draw[->, thick] (24, 4) -- (24, 4.9);
			\draw[->, thick] (24, 5) -- (24, 5.9);
			\draw[->, thick] (24, 6) -- (24.9, 6);
			\draw[->, thick] (25, 6) -- (25, 6.9);
			\draw[->, thick] (25, 7) -- (25, 7.9);
			\draw[->, thick] (25, 8) -- (25.75, 8);
			
			\draw[->, thick] (22, 0) -- (22, .9); 
			\draw[->, thick] (22, 1) -- (22.9, 1); 
			\draw[->, thick] (23, 1) -- (23, 1.9);
			\draw[->, thick] (23, 2) -- (23.9, 2);
			\draw[->, thick] (24, 2) -- (24, 2.9); 
			\draw[->, thick] (24, 3) -- (24, 3.9); 
			\draw[->, thick] (24, 4) -- (24.9, 4);
			\draw[->, thick] (25, 4) -- (25, 4.9);
			\draw[->, thick] (25, 5) -- (25, 5.9);
			\draw[->, thick] (25, 6) -- (25.75, 6);

			\filldraw[fill=gray!50!black, draw=black] (15, 1) circle [radius=.25];
			\filldraw[fill=gray!50!white, draw=black] (15, 2) circle [radius=.1];
			\filldraw[fill=gray!50!white, draw=black] (15, 3) circle [radius=.1];
			\filldraw[fill=gray!50!black, draw=black] (15, 4) circle [radius=.25];
			\filldraw[fill=gray!50!white, draw=black] (15, 5) circle [radius=.1];
			\filldraw[fill=gray!50!white, draw=black] (15, 6) circle [radius=.1];
			\filldraw[fill=gray!50!white, draw=black] (15, 7) circle [radius=.1];
			\filldraw[fill=gray!50!white, draw=black] (15, 8) circle [radius=.1];
			\filldraw[fill=gray!50!white, draw=black] (15, 9) circle [radius=.1];
			\filldraw[fill=gray!50!black, draw=black] (15, 10) circle [radius=.25];

			\filldraw[fill=gray!50!white, draw=black] (16, 0) circle [radius=.1];
			\filldraw[fill=gray!50!white, draw=black] (16, 1) circle [radius=.1];
			\filldraw[fill=gray!50!white, draw=black] (16, 2) circle [radius=.1];
			\filldraw[fill=gray!50!white, draw=black] (16, 3) circle [radius=.1];
			\filldraw[fill=gray!50!white, draw=black] (16, 4) circle [radius=.1];
			\filldraw[fill=gray!50!white, draw=black] (16, 5) circle [radius=.1];
			\filldraw[fill=gray!50!white, draw=black] (16, 6) circle [radius=.1];
			\filldraw[fill=gray!50!white, draw=black] (16, 7) circle [radius=.1];
			\filldraw[fill=gray!50!white, draw=black] (16, 8) circle [radius=.1];
			\filldraw[fill=gray!50!white, draw=black] (16, 9) circle [radius=.1];
			\filldraw[fill=gray!50!white, draw=black] (16, 10) circle [radius=.1];
			\filldraw[fill=gray!50!white, draw=black] (16, 11) circle [radius=.1];
			
			\filldraw[fill=gray!50!white, draw=black] (17, 0) circle [radius=.1];
			\filldraw[fill=gray!50!white, draw=black] (17, 1) circle [radius=.1];
			\filldraw[fill=gray!50!white, draw=black] (17, 2) circle [radius=.1];
			\filldraw[fill=gray!50!white, draw=black] (17, 3) circle [radius=.1];
			\filldraw[fill=gray!50!white, draw=black] (17, 4) circle [radius=.1];
			\filldraw[fill=gray!50!white, draw=black] (17, 5) circle [radius=.1];
			\filldraw[fill=gray!50!white, draw=black] (17, 6) circle [radius=.1];
			\filldraw[fill=gray!50!white, draw=black] (17, 7) circle [radius=.1];
			\filldraw[fill=gray!50!white, draw=black] (17, 8) circle [radius=.1];
			\filldraw[fill=gray!50!white, draw=black] (17, 9) circle [radius=.1];
			\filldraw[fill=gray!50!white, draw=black] (17, 10) circle [radius=.1];
			\filldraw[fill=gray!50!white, draw=black] (17, 11) circle [radius=.1];
			
			\filldraw[fill=gray!50!white, draw=black] (18, 0) circle [radius=.1];
			\filldraw[fill=gray!50!white, draw=black] (18, 1) circle [radius=.1];
			\filldraw[fill=gray!50!white, draw=black] (18, 2) circle [radius=.1];
			\filldraw[fill=gray!50!white, draw=black] (18, 3) circle [radius=.1];
			\filldraw[fill=gray!50!white, draw=black] (18, 4) circle [radius=.1];
			\filldraw[fill=gray!50!white, draw=black] (18, 5) circle [radius=.1];
			\filldraw[fill=gray!50!white, draw=black] (18, 6) circle [radius=.1];
			\filldraw[fill=gray!50!white, draw=black] (18, 7) circle [radius=.1];
			\filldraw[fill=gray!50!white, draw=black] (18, 8) circle [radius=.1];
			\filldraw[fill=gray!50!white, draw=black] (18, 9) circle [radius=.1];
			\filldraw[fill=gray!50!white, draw=black] (18, 10) circle [radius=.1];
			\filldraw[fill=gray!50!white, draw=black] (18, 11) circle [radius=.1];
			
			\filldraw[fill=gray!50!white, draw=black] (19, 0) circle [radius=.1];
			\filldraw[fill=gray!50!white, draw=black] (19, 1) circle [radius=.1];
			\filldraw[fill=gray!50!white, draw=black] (19, 2) circle [radius=.1];
			\filldraw[fill=gray!50!white, draw=black] (19, 3) circle [radius=.1];
			\filldraw[fill=gray!50!white, draw=black] (19, 4) circle [radius=.1];
			\filldraw[fill=gray!50!white, draw=black] (19, 5) circle [radius=.1];
			\filldraw[fill=gray!50!white, draw=black] (19, 6) circle [radius=.1];
			\filldraw[fill=gray!50!white, draw=black] (19, 7) circle [radius=.1];
			\filldraw[fill=gray!50!white, draw=black] (19, 8) circle [radius=.1];
			\filldraw[fill=gray!50!white, draw=black] (19, 9) circle [radius=.1];
			\filldraw[fill=gray!50!white, draw=black] (19, 10) circle [radius=.1];
			\filldraw[fill=gray!50!white, draw=black] (19, 11) circle [radius=.1];
			
			\filldraw[fill=gray!50!black, draw=black] (20, 0) circle [radius=.25];
			\filldraw[fill=gray!50!white, draw=black] (20, 1) circle [radius=.1];
			\filldraw[fill=gray!50!white, draw=black] (20, 2) circle [radius=.1];
			\filldraw[fill=gray!50!white, draw=black] (20, 3) circle [radius=.1];
			\filldraw[fill=gray!50!white, draw=black] (20, 4) circle [radius=.1];
			\filldraw[fill=gray!50!white, draw=black] (20, 5) circle [radius=.1];
			\filldraw[fill=gray!50!white, draw=black] (20, 6) circle [radius=.1];
			\filldraw[fill=gray!50!white, draw=black] (20, 7) circle [radius=.1];
			\filldraw[fill=gray!50!white, draw=black] (20, 8) circle [radius=.1];
			\filldraw[fill=gray!50!white, draw=black] (20, 9) circle [radius=.1];
			\filldraw[fill=gray!50!white, draw=black] (20, 10) circle [radius=.1];
			\filldraw[fill=gray!50!black, draw=black] (20, 11) circle [radius=.25];
			
			\filldraw[fill=gray!50!white, draw=black] (21, 0) circle [radius=.1];
			\filldraw[fill=gray!50!white, draw=black] (21, 1) circle [radius=.1];
			\filldraw[fill=gray!50!white, draw=black] (21, 2) circle [radius=.1];
			\filldraw[fill=gray!50!white, draw=black] (21, 3) circle [radius=.1];
			\filldraw[fill=gray!50!white, draw=black] (21, 4) circle [radius=.1];
			\filldraw[fill=gray!50!white, draw=black] (21, 5) circle [radius=.1];
			\filldraw[fill=gray!50!white, draw=black] (21, 6) circle [radius=.1];
			\filldraw[fill=gray!50!white, draw=black] (21, 7) circle [radius=.1];
			\filldraw[fill=gray!50!white, draw=black] (21, 8) circle [radius=.1];
			\filldraw[fill=gray!50!white, draw=black] (21, 9) circle [radius=.1];
			\filldraw[fill=gray!50!white, draw=black] (21, 10) circle [radius=.1];
			\filldraw[fill=gray!50!white, draw=black] (21, 11) circle [radius=.1];
			
			\filldraw[fill=gray!50!black, draw=black] (22, 0) circle [radius=.25];
			\filldraw[fill=gray!50!white, draw=black] (22, 1) circle [radius=.1];
			\filldraw[fill=gray!50!white, draw=black] (22, 2) circle [radius=.1];
			\filldraw[fill=gray!50!white, draw=black] (22, 3) circle [radius=.1];
			\filldraw[fill=gray!50!white, draw=black] (22, 4) circle [radius=.1];
			\filldraw[fill=gray!50!white, draw=black] (22, 5) circle [radius=.1];
			\filldraw[fill=gray!50!white, draw=black] (22, 6) circle [radius=.1];
			\filldraw[fill=gray!50!white, draw=black] (22, 7) circle [radius=.1];
			\filldraw[fill=gray!50!white, draw=black] (22, 8) circle [radius=.1];
			\filldraw[fill=gray!50!white, draw=black] (22, 9) circle [radius=.1];
			\filldraw[fill=gray!50!white, draw=black] (22, 10) circle [radius=.1];
			\filldraw[fill=gray!50!white, draw=black] (22, 11) circle [radius=.1];
			
			\filldraw[fill=gray!50!white, draw=black] (23, 0) circle [radius=.1];
			\filldraw[fill=gray!50!white, draw=black] (23, 1) circle [radius=.1];
			\filldraw[fill=gray!50!white, draw=black] (23, 2) circle [radius=.1];
			\filldraw[fill=gray!50!white, draw=black] (23, 3) circle [radius=.1];
			\filldraw[fill=gray!50!white, draw=black] (23, 4) circle [radius=.1];
			\filldraw[fill=gray!50!white, draw=black] (23, 5) circle [radius=.1];
			\filldraw[fill=gray!50!white, draw=black] (23, 6) circle [radius=.1];
			\filldraw[fill=gray!50!white, draw=black] (23, 7) circle [radius=.1];
			\filldraw[fill=gray!50!white, draw=black] (23, 8) circle [radius=.1];
			\filldraw[fill=gray!50!white, draw=black] (23, 9) circle [radius=.1];
			\filldraw[fill=gray!50!white, draw=black] (23, 10) circle [radius=.1];
			\filldraw[fill=gray!50!white, draw=black] (23, 11) circle [radius=.1];
			
			\filldraw[fill=gray!50!white, draw=black] (24, 0) circle [radius=.1];
			\filldraw[fill=gray!50!white, draw=black] (24, 1) circle [radius=.1];
			\filldraw[fill=gray!50!white, draw=black] (24, 2) circle [radius=.1];
			\filldraw[fill=gray!50!white, draw=black] (24, 3) circle [radius=.1];
			\filldraw[fill=gray!50!white, draw=black] (24, 4) circle [radius=.1];
			\filldraw[fill=gray!50!white, draw=black] (24, 5) circle [radius=.1];
			\filldraw[fill=gray!50!white, draw=black] (24, 6) circle [radius=.1];
			\filldraw[fill=gray!50!white, draw=black] (24, 7) circle [radius=.1];
			\filldraw[fill=gray!50!white, draw=black] (24, 8) circle [radius=.1];
			\filldraw[fill=gray!50!white, draw=black] (24, 9) circle [radius=.1];
			\filldraw[fill=gray!50!white, draw=black] (24, 10) circle [radius=.1];
			\filldraw[fill=gray!50!white, draw=black] (24, 11) circle [radius=.1];
			
			\filldraw[fill=gray!50!white, draw=black] (25, 0) circle [radius=.1];
			\filldraw[fill=gray!50!white, draw=black] (25, 1) circle [radius=.1];
			\filldraw[fill=gray!50!white, draw=black] (25, 2) circle [radius=.1];
			\filldraw[fill=gray!50!white, draw=black] (25, 3) circle [radius=.1];
			\filldraw[fill=gray!50!white, draw=black] (25, 4) circle [radius=.1];
			\filldraw[fill=gray!50!white, draw=black] (25, 5) circle [radius=.1];
			\filldraw[fill=gray!50!white, draw=black] (25, 6) circle [radius=.1];
			\filldraw[fill=gray!50!white, draw=black] (25, 7) circle [radius=.1];
			\filldraw[fill=gray!50!white, draw=black] (25, 8) circle [radius=.1];
			\filldraw[fill=gray!50!white, draw=black] (25, 9) circle [radius=.1];
			\filldraw[fill=gray!50!white, draw=black] (25, 10) circle [radius=.1];
			\filldraw[fill=gray!50!white, draw=black] (25, 11) circle [radius=.1];
			
			\filldraw[fill=gray!50!white, draw=black] (26, 1) circle [radius=.1];
			\filldraw[fill=gray!50!white, draw=black] (26, 2) circle [radius=.1];
			\filldraw[fill=gray!50!white, draw=black] (26, 3) circle [radius=.1];
			\filldraw[fill=gray!50!white, draw=black] (26, 4) circle [radius=.1];
			\filldraw[fill=gray!50!white, draw=black] (26, 5) circle [radius=.1];
			\filldraw[fill=gray!50!black, draw=black] (26, 6) circle [radius=.25];
			\filldraw[fill=gray!50!white, draw=black] (26, 7) circle [radius=.1];
			\filldraw[fill=gray!50!black, draw=black] (26, 8) circle [radius=.25];
			\filldraw[fill=gray!50!white, draw=black] (26, 9) circle [radius=.1];
			\filldraw[fill=gray!50!black, draw=black] (26, 10) circle [radius=.25];

			\filldraw[fill=gray!50!black, draw=black] (26, 12) circle [radius=0] node[scale = 1.5]{$\mathcal{E}'$};

			\draw[] (26.5, 1) -- (26.85, 1) -- (26.85, 10) -- (26.5, 10);
			\draw[] (27.35, 5.5) circle [radius = 0] node[]{$N$};
			
			\draw[] (16, -.5) -- (16, -.85) -- (25, -.85) -- (25, -.5);
			\draw[] (20.5, -1.35) circle [radius = 0] node[]{$N$};

			\end{tikzpicture}
			
		\end{center}	
		
		\caption{\label{restrictionboundary} Depicted above, $\mathcal{E}' \in \mathfrak{E}(\Lambda_{10})$ is the $(2; 3)$-restriction of $\mathcal{E} \in \mathfrak{E} (\Lambda_{10})$. Vertices in the boundary data for these ensembles are drawn darker and larger than are the other vertices in $\overline{\Lambda}_{10}$.}
	\end{figure}
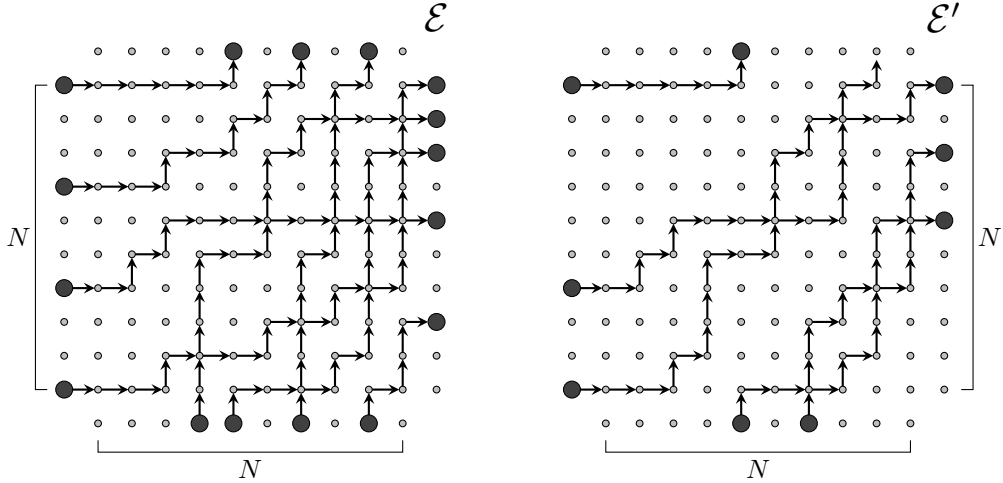

	\begin{definition} 
		
		\label{nmklrestrction} 
		
		Fix an integer $N > 0$; define the domain $\Lambda_N = [1, N] \times [1, N] \subset \mathbb{Z}^2$; and let $\textbf{u} \cup \textbf{v} = (u_{-B}, u_{1 - B}, \ldots , u_A) \cup (v_{-B}, v_{1 - B}, \ldots , v_A)$ denote boundary data  on $\Lambda_N$ for a six-vertex ensemble $\mathcal{E} \in \mathfrak{E} (\Lambda_N)$. For any integers $K > 0$ and $L \in [0, K]$, we define the \emph{$(L; K)$-restriction of} $\textbf{u} \cup \textbf{v}$ to be the boundary data $\textbf{u}' \cup \textbf{v}' = (u_{-B'}', u_{1 - B'}', \ldots , u_{A'}') \cup (v_{-B'}', v_{1 - B'}', \ldots , v_{A'}')$ obtained by setting $u_i' \in \textbf{u}'$ and $v_i' \in \textbf{v}'$ if and only if there exist $m \in \mathbb{Z}$ and $r \in [1, L]$ such that $u_i' = u_{mK + r}$ and $v_i' = v_{mK + r}$, respectively. 
		
		Similarly, the $(L; K)$-restriction $\mathcal{E}' \in \mathfrak{E} (\Lambda_N)$ of $\mathcal{E}$ is the six-vertex ensemble defined as follows. Denoting the non-crossing path ensembles associated with $\mathcal{E}$ and $\mathcal{E}'$ by $\mathcal{P} = (\textbf{p}_{-B}, \textbf{p}_{1 - B}, \ldots , \textbf{p}_A)$ and $\mathcal{P}' = (\textbf{p}_{-B'}', \textbf{p}_{1 - B'}', \ldots , \textbf{p}_{A'}')$, respectively, we have $\textbf{p}_i' \in \mathcal{P}'$ if and only if there exist $m \in \mathbb{Z}$ and $r \in [1, L]$ such that $\textbf{p}_i' = \textbf{p}_{mK + r}$. We refer to \Cref{restrictionboundary} for a depiction. 
		
	\end{definition}

	The following lemma essentially states that $(L; K)$-restricting regular boundary data of slope $(s_0, t_0)$ largely preserves its regularity but ``reduces'' its slope to $(s, t) \approx \big( \vartheta s_0, \vartheta t_0 \big)$, where $\vartheta = \frac{L}{K}$. 
		
	\begin{lem} 
		
		\label{restrictionregular}
		
		Fix real numbers $\eta \in (0, 1)$ and $R \ge 1$; two pairs $(s_0, t_0), (s, t) \in (0, 1]^2$; and integers $N \ge K \ge L > 0$. Assume that  
		\begin{flalign*}
		\bigg|  \displaystyle\frac{s_0 L}{K} - s \bigg| < \eta; \qquad \bigg| \displaystyle\frac{t_0 L}{K} - t \bigg| < \eta; \qquad R \le L; \qquad \eta < \displaystyle\frac{s_0 t_0}{4}.
		\end{flalign*}
		
		\noindent Define the domain $\Lambda = \Lambda_N = [1, N] \times [1, N] \subseteq \mathbb{Z}^2$, and fix boundary data $\textbf{\emph{u}} \cup \textbf{\emph{v}}$ on $\Lambda$; let $\textbf{\emph{u}}' \cup \textbf{\emph{v}}'$ denote the $(L; K)$-restriction of $\textbf{\emph{u}} \cup \textbf{\emph{v}}$. If $\textbf{\emph{u}} \cup \textbf{\emph{v}}$ is $(R; \eta)$-regular with slope $(s_0, t_0)$, then $\textbf{\emph{u}}' \cup \textbf{\emph{v}}'$ is $\big( \frac{2K}{s_0 t_0 \omega}; \frac{4 (\eta + \omega)}{s_0 t_0} \big)$-regular with slope $(s, t)$, for any real number $\omega > 0$. 
		
	\end{lem}  
	
	\begin{proof}
		
		Let $I \subset \partial \Lambda$ denote an interval with $|I| \le \frac{2K}{s_0 t_0 \omega}$. It suffices to show that if $I$ lies on the north or south boundary of $\Lambda$, then $\big| I \cap (\textbf{u}' \cup \textbf{v}') - s|I| \big| \le 8 (\eta + \omega) (s_0^2 t_0^2 \omega)^{-1} K$ and, if $I$ lies on the east or west boundary of $\Lambda$, then $\big| I \cap (\textbf{u}' \cup \textbf{v}') - t|I| \big| \le 8 (\eta + \omega) (s_0^2 t_0^2 \omega)^{-1} K$. Let us assume that $I \subset [1, N] \times \{ 0 \}$ lies on the west boundary of $\Lambda$, as the remaining cases are entirely analogous. 
		
		In this case, let $\textbf{u} = (u_{-B}, u_{1 - B}, \ldots , u_A)$ and $\textbf{v} = (v_{-B}, v_{1 - B}, \ldots , v_A)$; for each integer $i \ge 1$, set $u_i = (0, y_i)$. Further define the intervals $I_m = \{ 0 \} \times [y_{mK + 1}, y_{mK + K}] \subset \partial \Lambda$ and $J_m = \{ 0 \} \times [y_{mK + 1}, y_{mK + L}] \subseteq I_m$ along the west boundary of $\Lambda$, for each integer $m \ge 0$ for which they exist. Since $\textbf{u}$ is $(R; \eta)$-regular of slope $(s, t)$ and $|I_m| \ge K \ge L \ge R$, by \Cref{uetareta} we have that $K = |\textbf{u} \cap I_m| \ge (s_0 - 2 \eta) |I_m| \ge \frac{s_0 |I_m|}{2}$ (as $\eta < \frac{s_0 t_0}{4}$). Thus, $|I_m| \le 2s_0^{-1} K$. 
		
		Now, we may assume that $|I| \ge \frac{8K}{s_0^2 t_0^2}$, for otherwise $\big| I \cap (\textbf{u}' \cup \textbf{v}') - s|I| \big| \le |I| \le 8 (\eta + \omega) (s_0^2 t_0^2 \omega)^{-1} K$. Then $|I| > 4s_0^{-1} K$ so, since $|I_m| \le 2 s_0^{-1} K$, there exist integers $0 \le m_1 \le m_2$ such that $\bigcup_{m = m_1}^{m_2} I_m \subseteq I$ and $\big| I \setminus \bigcup_{m = m_1}^{m_2} I_m \big| \le 4 s_0^{-1} K$. Since $\textbf{u}'$ is the $(L; K)$-restriction of $\textbf{u}$, this yields 
		\begin{flalign}
		\label{iuti} 
		\begin{aligned} 
		\big| |I \cap \textbf{u}'| - t|I| \big| & \le  \Bigg| \bigg| \bigcup_{m = m_1}^{m_2} (J_m \cap \textbf{u}) \bigg| - t|I| \Bigg| + 4 s_0^{-1} K \\
		& \le \Bigg| (m_2 - m_1 + 1) L - t \bigg| \bigcup_{m = m_1}^{m_2} I_m \bigg| \Bigg| + 8 s_0^{-1} K.
		\end{aligned} 
		\end{flalign}
		
		Moreover, by \Cref{uetareta}, the $(R; \eta)$-regularity of $\textbf{u} \cap \textbf{v}$ and the fact that $|I_m \cap \textbf{u}| = K$ imply $\big| K - t_0 |I_m| \big| \le 2 \eta |I_m|$. Summing over $m \in [m_1, m_2]$, we obtain 
		\begin{flalign*}
		\Bigg| \bigg| \bigcup_{m = m_1}^{m_2} I_m \bigg| - t_0^{-1} (m_2 - m_1 + 1) K \Bigg| \le 2 t_0^{-1} \eta |I|.
		\end{flalign*}
		
		\noindent Together with \eqref{iuti} and the bounds $|L - t_0^{-1} t K| \le t_0^{-1}\eta K$ and $(m_2 - m_1 + 1) K \le |I| \le 2 (s_0 t_0 \omega)^{-1} K$, this gives
		\begin{flalign*}
		\big| |I \cap \textbf{u}'| - t|I| \big| \le (m_2 - m_1 + 1) |L - t_0^{-1} t K| + 2 t_0^{-1} t \eta |I| + 8 s_0^{-1} K & \le 3 t_0^{-1} \eta |I| + 8 s_0^{-1} K \\
		& \le 8 (\eta + \omega) (s_0^2 t_0^2 \omega)^{-1} K,
		\end{flalign*}
		
		\noindent which, as mentioned above, implies the lemma.
	\end{proof}

	We next have the following proposition that compares probabilities between two stochastic six-vertex models, the latter of whose boundary data is the $(L; K)$-restriction of that of the former. Observe in the below that the prefactor $\big( (1 - B_1) (1 - B_2) \big)^{4MN}$ appearing on the right side of \eqref{sgesge} is $e^{- o(N^2)}$ if $K \gg 1$, indicating in this case that $(L; K)$-restriction cannot reduce a partition function by more than $e^{o(N^2)}$.  
	
	\begin{prop} 
		
		\label{rholrho}
		
		Fix integers $N, K > 0$ and $L \in [0, K]$. Set $M = \big\lceil \frac{N}{K} \big\rceil$, define $\Lambda = [1, N] \times [1, N] \subseteq \mathbb{Z}^2$, and fix some entrance data $\textbf{\emph{u}}$ for a six-vertex ensemble on $\Lambda$. Let $\textbf{\emph{u}}'$ denote the $(L; K)$-restriction of $\textbf{\emph{u}}$, and fix some six-vertex ensemble $\mathcal{G} \in \mathfrak{E}_{\textbf{\emph{u}}'} (\Lambda)$ on $\Lambda$ with entrance data $\textbf{\emph{u}}'$. 
		
		Consider two $(B_1, B_2)$-stochastic six-vertex models on $\Lambda$, denoted by $\mathfrak{S}$ and $\mathfrak{S}'$, with entrance data $\textbf{\emph{u}}$ and $\textbf{\emph{u}}'$, respectively, and both with free exit data. Let $\mathcal{E} \in \mathfrak{E}(\Lambda)$ and $\mathcal{E}' \in \mathfrak{E}(\Lambda)$ denote random six-vertex ensembles sampled under $\mathfrak{S}$ and $\mathfrak{S}'$, respectively, and let $\mathcal{F} \in \mathfrak{E} (\Lambda)$ denote the $(L; K)$-restriction of $\mathcal{E}$. Then, 
		\begin{flalign}
		\label{sgesge}
		\mathbb{P}_{\mathfrak{S}'} [ \mathcal{E}' = \mathcal{G} ] \ge \big( (1 - B_1) (1 - B_2) \big)^{4 MN} \mathbb{P}_{\mathfrak{S}} [\mathcal{F} = \mathcal{G}].
		\end{flalign}
	
	\end{prop}  
	
	\begin{proof}
	
		For each $n \ge 1$, recall from  \Cref{ModelStochastic} the subdomain $\mathcal{T}_n = \{ (x, y) \in \mathbb{Z}_{\ge 0}^2: x + y \le n \} \cap \Lambda \subseteq \Lambda$ and diagonal $\mathcal{D}_n = \{ (x, y) \in \mathbb{Z}_{> 0}^2: x + y = n \} \cap \Lambda$. We will show for each integer $n \in [2, 2N]$ that 
		\begin{flalign}
		\label{gefe}
		\mathbb{P}_{\mathfrak{S}'} \big[ \mathcal{E}_{\mathcal{T}_n}' = \mathcal{G}_{\mathcal{T}_n}	 | \mathcal{E}_{\mathcal{T}_{n - 1}}' = \mathcal{G}_{\mathcal{T}_{n - 1}} \big] \ge \big( (1 - B_1) (1 - B_2) \big)^{2M} \mathbb{P}_{\mathfrak{S}} \big[ \mathcal{F}_{\mathcal{T}_n} = \mathcal{G}_{\mathcal{T}_n} | \mathcal{F}_{\mathcal{T}_{n - 1}} = \mathcal{G}_{\mathcal{T}_{n - 1}} \big].
		\end{flalign}
		
		\noindent Given \eqref{gefe} we deduce from the facts that $\mathcal{T}_{2N} = \Lambda$ and that $\mathcal{E}_{\mathcal{T}_1}' = \mathcal{G}_{\mathcal{T}_1}$ and $\mathcal{F}_{\mathcal{T}_1} = \mathcal{G}_{\mathcal{T}_1}$ both hold deterministically that 
		\begin{flalign*}
		\mathbb{P}_{\mathfrak{S}'} [\mathcal{E}' = \mathcal{G}] = \mathbb{P}_{\mathfrak{S}'} \big[ \mathcal{E}_{\mathcal{T}_{2N}}' = \mathcal{G}_{\mathcal{T}_{2N}} \big] & = \displaystyle\prod_{n = 2}^{2N} \mathbb{P}_{\mathfrak{S}'} \big[ \mathcal{E}_{\mathcal{T}_n}' = \mathcal{G}_{\mathcal{T}_n} | \mathcal{E}_{\mathcal{T}_{n - 1}}' = \mathcal{G}_{\mathcal{T}_{n - 1}} \big] \\
		& \ge \big( (1 - B_1) (1 - B_2) \big)^{4MN} \displaystyle\prod_{n = 2}^{2N}  \mathbb{P}_{\mathfrak{S}} \big[ \mathcal{F}_{\mathcal{T}_n} = \mathcal{G}_{\mathcal{T}_n} | \mathcal{F}_{\mathcal{T}_{n - 1}} = \mathcal{G}_{\mathcal{T}_{n - 1}} \big] \\
		& = \big( (1 - B_1) (1 - B_2) \big)^{4MN} \mathbb{P}_{\mathfrak{S}} \big[ \mathcal{F}_{\mathcal{T}_{2N}} = \mathcal{F}_{\mathcal{G}_{2N}} \big] \\
		& = \big( (1 - B_1) (1 - B_2) \big)^{4MN} \mathbb{P}_{\mathfrak{S}} [\mathcal{F} = \mathcal{G}],
		\end{flalign*}
		
		\noindent which yields \eqref{sgesge}. Thus, it suffices to establish \eqref{gefe}.
		
		To that end, we begin with some notation. Denote $\textbf{u} = (u_{-B}, u_{1 - B}, \ldots , u_A)$, and let $\mathcal{I} = [-B, A] \cap \bigcup_{m \in \mathbb{Z}} [mK + 1, mK + L]$. We further define the sets $\mathcal{R} = (r_1, r_2, \ldots , r_x) = [-B, A] \cap \bigcup_{m \in \mathbb{Z}} \{ mK + L \}$ and $\mathcal{S} = (s_1, s_2, \ldots , s_y) = [-B, A] \cap \bigcup_{m \in \mathbb{Z}} \{ mK + 1 \}$. In this way, $\mathcal{R}$ constitutes potential indices $r \in \mathcal{I}$ for which $u_r \in \textbf{u}'$ and $\textbf{u}_{r + 1} \notin \textbf{u}'$, and $\mathcal{S}$ constitutes potential indices $s \in \mathcal{I}$ for which $u_s \in \textbf{u}'$ and $\textbf{u}_{s - 1} \notin \textbf{u}'$. Observe under this notation that $x, y \le 2M$. For the example depicted in \Cref{restrictionboundary}, we have that $(A, B) = (4, 3)$, $(L, K) = (2, 3)$, $\mathcal{I} = (-2, -1, 1, 2, 4)$, $\mathcal{R} = (-1, 2)$, $\mathcal{S} = (-2, 1, 4)$, $x = 2$, and $y = 3$. 
		
		Now, let us describe a coupled sampling of $(\mathcal{E}_{\mathcal{T}_n}, \mathcal{E}_{\mathcal{T}_n}')$ given $(\mathcal{E}_{\mathcal{T}_{n - 1}}, \mathcal{E}_{\mathcal{T}_{n - 1}}')$ and that $\mathcal{E}_{\mathcal{T}_{n - 1}}' = \mathcal{G}_{\mathcal{T}_{n - 1}} = \mathcal{F}_{\mathcal{T}_{n - 1}}$. To do this, denote the non-crossing path ensemble associated with $\mathcal{E}$ by $\mathcal{P} = (\textbf{p}_{-B}, \textbf{p}_{1 - B}, \ldots , \textbf{p}_A)$. Our conditioning on $(\mathcal{E}_{\mathcal{T}_{n - 1}}, \mathcal{E}_{\mathcal{T}_{n - 1}}')$ prescribes for each $v \in \mathcal{D}_n$ all indices $i \in [-B, A]$ such that $v \in \textbf{p}_i$ (that is, the indices of all paths passing through $v$). Denoting the arrow configurations at any $v \in \mathcal{T}_n$ under $\mathcal{E}_{\mathcal{T}_n}$ and $\mathcal{E}_{\mathcal{T}_n}'$ by $\big( i_1 (v), j_1 (v); i_2 (v), j_2 (v) \big)$ and $\big( i_1' (v), j_1' (v); i_2' (v), j_2' (v) \big)$, respectively, we have that $\big( i_1 (v), j_1 (v) \big) = \big( i_1' (v), j_1' (v) \big)$ unless there exists some $m \notin \mathcal{I}$ for which $v \in \textbf{p}_m$. 
		
		For any $v \in \mathcal{D}_n$, let us randomly define $\big( i_2 (v), j_2 (v) \big)$ and $\big( i_2' (v), j_2' (v) \big)$ (given $\big( i_1 (v), j_1 (v) \big)$ and $\big( i_1' (v), j_1' (v) \big)$ from $(\mathcal{E}_{\mathcal{T}_{n - 1}}, \mathcal{E}_{\mathcal{T}_{n - 1}}')$) as follows. In the below, all choices over $v \in \mathcal{D}_n$ are mutually independent. 
		
		\begin{enumerate}
			
			\item If $v \notin \bigcup_{m \notin \mathcal{I}} \textbf{p}_m$, then couple $\big( i_2 (v), j_2 (v) \big) = \big( i_2' (v), j_2' (v) \big)$ under the probabilities from \eqref{bstochasticb}. Specifically, for any $i_2, j_2 \in \{ 0,1 \}$ set $\big( i_2 (v), j_2 (v) \big) = (i_2, j_2) = \big( i_2' (v), j_2' (v) \big)$ with probability $w\big( i_1 (v), j_1 (v); i_2, j_2 \big) = w\big( i_1' (v), j_1' (v); i_2, j_2 \big)$ (where the latter equality holds since $\big( i_1 (v), j_1 (v) \big) = \big( i_1' (v), j_1' (v) \big)$ if $v \notin \bigcup_{m \notin \mathcal{I}} \textbf{p}_m$). 
			
			\item Otherwise, set $\big( i_2 (v), j_2 (v) \big)$ and $\big( i_2' (v), j_2' (v) \big)$ independently, according to the probabilities in \eqref{bstochasticb}. 
		\end{enumerate} 
		
		\noindent This provides a sampling of $(\mathcal{E}_{\mathcal{T}_n}, \mathcal{E}_{\mathcal{T}_n}')$ given $(\mathcal{E}_{\mathcal{T}_{n - 1}}, \mathcal{E}_{\mathcal{T}_{n - 1}}')$. Defining for each $v \in \mathcal{D}_n$ the events
		\begin{flalign*}
		& \Upsilon^{(1)} (v) = \{ j_2' (v) \ge i_1' (v) \}; \qquad \Upsilon^{(2)} (v) = \big\{ i_2' (v) \ge j_1' (v) \big\}  \\
		& \Upsilon_n = \big\{ \mathcal{F}_{\mathcal{T}_n} = \mathcal{G}_{\mathcal{T}_n} \big\} \cap \bigcap_{r \in \mathcal{R}}  \bigcap_{v \in \textbf{p}_r \cap \textbf{p}_{r + 1} \cap \mathcal{D}_n} \Upsilon^{(1)} (v) \cap \bigcap_{s \in \mathcal{S}} \bigcap_{v \in \textbf{p}_{s - 1} \cap \textbf{p}_s \cap \mathcal{D}_n} \Upsilon^{(2)} (v),
		\end{flalign*} 
		
		\noindent we claim that $\mathcal{E}_{\mathcal{T}_n}' = \mathcal{G}_{\mathcal{T}_n}$ holds on $\Upsilon_n$ and that 
		\begin{flalign}
		\label{omegan1} 
		\mathbb{P} [\Upsilon_n] \ge \big( (1 - B_1) (1 - B_2) \big)^{2M} \mathbb{P}_{\mathfrak{S}} \big[ \mathcal{F}_{\mathcal{T}_n} = \mathcal{G}_{\mathcal{T}_n} | \mathcal{F}_{\mathcal{T}_{n - 1}} = \mathcal{G}_{\mathcal{T}_{n - 1}} \big],
		\end{flalign} 
		
		\noindent which would together imply \eqref{gefe}.
		
		Let us first establish the latter claim \eqref{omegan1}. To that end, observe for any $v \in \mathcal{D}_n \cap \bigcup_{r \in \mathcal{R}} (\textbf{p}_r \cap \textbf{p}_{r + 1})$ or $v \in \mathcal{D}_n \cap \bigcup_{s \in \mathcal{S}} (\textbf{p}_{s - 1} \cap \textbf{p}_s)$ that 
		\begin{flalign*} 
		\mathbb{P} \big[ \Upsilon^{(1)} (v) \big| \mathcal{F}_{\mathcal{T}_n} = \mathcal{G}_{\mathcal{T}_n} \big] \ge w (1, 0; 0, 1) = 1 - B_1; \quad \mathbb{P} \big[ \Upsilon^{(2)} (v) \big| \mathcal{E}_{\mathcal{T}_{n - 1}}, \mathcal{E}_{\mathcal{T}_{n - 1}}' \big] \ge w (0, 1; 1, 0) = 1 - B_2,
		\end{flalign*} 
		
		\noindent since then $\big( i_2 (v), j_2 (v) \big)$ and $\big( i_2' (v), j_2' (v) \big)$ are independent under the coupling described above. This, together with the above mentioned bounds $|\mathcal{R}| = x \le 2M$ and $|\mathcal{S}| = y \le 2M$ and the mutual independence between the $\Upsilon^{(i)} (v)$ for $v \in \bigcup_{m \notin \mathcal{I}} (\textbf{p}_m \cap \mathcal{D}_n)$, yields \eqref{omegan1}.
		
		It therefore remains to verify $\mathcal{E}_{\mathcal{T}_n}' = \mathcal{G}_{\mathcal{T}_n}$ on $\Upsilon_n$. So, let us restrict to $\Upsilon_n$ and denote the arrow configuration at $v \in \Lambda$ under $\mathcal{G}$ by $\big( I_1 (v), J_1 (v); I_2 (v), J_2 (v) \big)$; it suffices to show that $\big( i_1' (v), j_1' (v); i_2' (v), j_2' (v) \big) = \big( I_1 (v), J_1 (v); I_2 (v), J_2 (v) \big)$, for any vertex $v \in \mathcal{D}_n$. To do this, we separately consider cases depending on $v$. 
		
		If $v \notin \bigcup_{i \in \mathcal{I}}\textbf{p}_i$, then $\big( I_1 (v), J_1 (v); I_2 (v), J_2 (v) \big) = (0, 0; 0, 0)$, which since $\mathcal{E}_{\mathcal{T}_{n - 1}}' = \mathcal{G}_{\mathcal{T}_{n - 1}}$ implies $\big( i_1' (v), j_1' (v) \big) = (0, 0)$. Thus, $\big( i_1' (v), j_1' (v); i_2' (v), j_2' (v) \big) = (0, 0; 0, 0) = \big( I_1 (v), J_1 (v); I_2 (v), J_2 (v) \big)$. 
		
		So, suppose instead that $v \in \bigcup_{i \in \mathcal{I}} \textbf{p}_i$. If $v \notin \bigcup_{m \notin \mathcal{I}} \textbf{p}_m$, then $\big( i_1(v), j_1 (v); i_2 (v), j_2 (v) \big) = \big( I_1 (v), J_1 (v); I_2 (v), J_2 (v) \big)$, since $\mathcal{F}_{\mathcal{T}_n} = \mathcal{G}_{\mathcal{T}_n}$ on $\Upsilon_n$ and no path through $v$ is removed from $\mathcal{E}$ upon passing to its $(L; K)$-restriction $\mathcal{F}$. Moreover, $\big( i_1 (v), j_1 (v); i_2 (v), j_2 (v) \big) = \big( i_1' (v), j_1' (v); i_2' (v), j_2' (v) \big)$, due to the coupling between $\big( i_2 (v), j_2 (v) \big)$ and $\big( i_2' (v), j_2' (v) \big)$ for $v \notin \bigcup_{m \notin \mathcal{I}} \textbf{p}_m$. It therefore again follows that $\big( i_1' (v), j_1' (v); i_2' (v), j_2' (v) \big) = \big( I_1 (v), J_1 (v); I_2 (v), J_2 (v) \big)$. 
		
		Next, we consider the case $v \in \bigcup_{i \in \mathcal{I}} \textbf{p}_i \cap \bigcup_{m \notin \mathcal{I}} \textbf{p}_m$. Then $\big( i_1(v), j_1 (v); i_2 (v), j_2 (v) \big) = (1, 1; 1, 1)$, since $v$ is the in the intersection $\textbf{p}_i \cap \textbf{p}_m$, for some $i \in \mathcal{I}$ and $m \notin \mathcal{I}$. In particular, there either exists some index $r \in \mathcal{R}$ or $s \in \mathcal{S}$ such that $v \in \textbf{p}_r \cap \textbf{p}_{r + 1}$ or $v \in \textbf{p}_{s - 1} \cap \textbf{p}_s$, respectively. In the former case, $\textbf{p}_{r + 1}$ is removed from $\mathcal{E}$ when passing to $\mathcal{F}$, and so $\big( I_1 (v), J_1 (v); I_2 (v), J_2 (v) \big) = (1, 0; 0, 1)$; in the latter case, $\textbf{p}_{s - 1}$ is removed from $\mathcal{E}$ when passing to $\mathcal{F}$, and so $\big( I_1 (v), J_1 (v); I_2 (v), J_2 (v) \big) = (0, 1; 1, 0)$. Since $\big( i_1' (v), j_1' (v) \big) = \big( I_1 (v), J_1 (v) \big)$, we have in the former case that $\big( i_1' (v), j_1' (v); i_2' (v), j_2' (v) \big) = (1, 0; 0, 1)$ on $\Upsilon^{(1)} (v)$, and in the latter case that $\big( i_1' (v), j_1' (v); i_2' (v), j_2' (v) \big) = (0, 1; 1, 0)$ on $\Upsilon^{(2)} (v)$. This again implies on $\Upsilon_n$ that $\big( i_1' (v), j_1' (v); i_2' (v), j_2' (v) \big) = \big( I_1 (v), J_1 (v); I_2 (v), J_2 (v) \big)$, from which we deduce $\mathcal{E}_{\mathcal{T}_n}' = \mathcal{G}_{\mathcal{T}_n}$ and therefore the proposition. 
	\end{proof}

	\subsection{Extension of Six-Vertex Ensembles}
	
	\label{LowerProbability} 

	In this section we establish the following lemma that provides a condition for when it is possible to ``extend'' a six-vertex ensemble on a square to one on a larger square with given boundary data; see the left side of \Cref{lambdalambdae} for a depiction. This condition essentially states that the boundary data for these ensembles along the smaller and larger squares are regular with the same slope. 
	
	\begin{lem} 
		
	\label{ensemblee0e} 
		
	Fix real numbers $\eta, s, t \in (0, 1)$ and integers $N, W, R > 0$; assume that 
	\begin{flalign}
	\label{nwr}
	\min \{ sW, tW \} \ge 50 \eta N; \qquad R \le \eta N; \qquad 50 (s^{-1} + t^{-1}) R \le W \le N.
	\end{flalign} 
	
	 \noindent Define the domains $\Lambda = \Lambda_{N + 2W} = [1, N + 2W] \times [1, N + 2W] \subset \mathbb{Z}^2$ and $\Lambda' = [W + 1, N + W] \times [W + 1, N + W]$. Let $\textbf{\emph{u}} \cup \textbf{\emph{v}}$ and $\textbf{\emph{u}}' \cup \textbf{\emph{v}}'$ denote boundary data on $\Lambda$ and $\Lambda'$ for six-vertex ensembles $\mathcal{F} \in \mathfrak{E} (\Lambda)$ and $\mathcal{F}' \in \mathfrak{E} (\Lambda')$, respectively. If $\textbf{\emph{u}} \cup \textbf{\emph{v}}$ and $\textbf{\emph{u}}' \cup \textbf{\emph{v}}'$ are both $(R; \eta)$-regular, then there exists a six-vertex ensemble $\mathcal{E} \in \mathfrak{E}_{\textbf{\emph{u}}; \textbf{\emph{v}}} (\Lambda)$ such that $\mathcal{E}_{\Lambda'} = \mathcal{F}'$. 
	
	\end{lem}

	\begin{figure}[t]
		
		\begin{center}
			
			\begin{tikzpicture}[
			>=stealth,
			scale = .45
			]

			\draw[->, very thick] (0, 10) -- (.9, 10);
			\draw[->, thick, dotted] (1, 10) -- (1.9, 10);
			\draw[->, thick, dotted] (2, 10) -- (2.9, 10);
			\draw[->, thick, dotted] (3, 10) -- (3.9, 10);
			\draw[->, thick, dotted] (4, 10) -- (4.9, 10);
			\draw[->, very thick] (5, 10) -- (5, 11);
			
			\draw[->, very thick] (0, 7) -- (.9, 7);
			\draw[->, thick, dotted] (1, 7) -- (1.9, 7);
			\draw[->, very thick] (2, 7) -- (2.9, 7);
			\draw[->] (3, 7) -- (3, 7.9);
			\draw[->] (3, 8) -- (3.9, 8);
			\draw[->] (4, 8) -- (4.9, 8);
			\draw[->, very thick] (5, 8) -- (5, 8.9); 
			\draw[->, thick, dotted] (5, 9) -- (5.9, 9);
			\draw[->, thick, dotted] (6, 9) -- (6, 9.9); 
			\draw[->, thick, dotted] (6, 10) -- (6.9, 10);
			\draw[->, very thick] (7, 10) -- (7, 11);
			
			\draw[->, very thick] (0, 4) -- (.9, 4);
			\draw[->, thick, dotted] (1, 4) -- (1.9, 4);
			\draw[->, thick, dotted] (2, 4) -- (2, 4.9);
			\draw[->, very thick] (2, 5) -- (2.9, 5);
			\draw[->] (3, 5) -- (3, 5.9);
			\draw[->] (3, 6) -- (3.9, 6);
			\draw[->] (4, 6) -- (4.9, 6);
			\draw[->] (5, 6) -- (5.9, 6);
			\draw[->] (6, 6) -- (6, 6.9);
			\draw[->] (6, 7) -- (6, 7.9);
			\draw[->] (6, 8) -- (6.9, 8); 
			\draw[->, very thick] (7, 8) -- (7, 8.9);
			\draw[->, thick, dotted] (7, 9) -- (7.9, 9); 
			\draw[->, thick, dotted] (8, 9) -- (8, 9.9);
			\draw[->, thick, dotted] (8, 10) -- (8.9, 10); 
			\draw[->, very thick] (9, 10) -- (9, 11);

			\draw[->, very thick] (0, 1) -- (.9, 1);
			\draw[->, thick, dotted] (1, 1) -- (1.9, 1);
			\draw[->, thick, dotted] (2, 1) -- (2.9, 1);
			\draw[->, thick, dotted] (3, 1) -- (3, 1.9);
			\draw[->, thick, dotted] (3, 2) -- (3.9, 2);
			\draw[->, very thick] (4, 2) -- (4, 2.9);
			\draw[->] (4, 3) -- (4, 3.9);
			\draw[->] (4, 4) -- (4, 4.9);
			\draw[->] (4, 5) -- (4.9, 5);
			\draw[->] (5, 5) -- (5.9, 5); 
			\draw[->] (6, 5) -- (6, 5.9); 
			\draw[->] (6, 6) -- (6.9, 6);
			\draw[->] (7, 6) -- (7.9, 6); 
			\draw[->] (8, 6) -- (8, 6.9);
			\draw[->] (8, 7) -- (8, 7.9);
			\draw[->, very thick] (8, 8) -- (8, 8.9);
			\draw[->, thick, dotted] (8, 9) -- (8.9, 9);
			\draw[->, thick, dotted] (9, 9) -- (9.9, 9);
			\draw[->, thick, dotted] (10, 9) -- (10, 9.9);
			\draw[->, very thick] (10, 10) -- (11, 10);

			\draw[->, very thick] (4, 0) -- (4, .9); 
			\draw[->, thick, dotted] (4, 1) -- (4, 1.9);
			\draw[->, thick, dotted] (4, 2) -- (4.9, 2); 
			\draw[->, thick, dotted] (5, 2) -- (5.9, 2);
			\draw[->, very thick] (6, 2) -- (6, 2.9); 
			\draw[->] (6, 3) -- (6.9, 3); 
			\draw[->] (7, 3) -- (7, 3.9);
			\draw[->] (7, 4) -- (7, 4.9);
			\draw[->] (7, 5) -- (7.9, 5); 
			\draw[->] (8, 5) -- (8, 5.9);
			\draw[->, very thick] (8, 6) -- (8.9, 6); 
			\draw[->, thick, dotted] (9, 6) -- (9, 6.9);
			\draw[->, thick, dotted] (9, 7) -- (9, 7.9);
			\draw[->, thick, dotted] (9, 8) -- (9.9, 8);
			\draw[->, thick, dotted] (10, 8) -- (10, 8.9);
			\draw[->, very thick] (10, 9) -- (11, 9);

			\draw[->, very thick] (5, 0) -- (5, .9);
			\draw[->, thick, dotted] (5, 1) -- (5.9, 1);
			\draw[->, thick, dotted] (6, 1) -- (6.9, 1);
			\draw[->, thick, dotted] (7, 1) -- (7, 1.9);
			\draw[->, very thick] (7, 2) -- (7, 2.9);
			\draw[->] (7, 3) -- (7.9, 3);
			\draw[->] (8, 3) -- (8, 3.9);
			\draw[->, very thick] (8, 4) -- (8.9, 4);
			\draw[->, thick, dotted] (9, 4) -- (9, 4.9);
			\draw[->, thick, dotted] (9, 5) -- (9, 5.9);
			\draw[->, thick, dotted] (9, 6) -- (9.9, 6);
			\draw[->, thick, dotted] (10, 6) -- (10, 6.9);
			\draw[->, thick, dotted] (10, 7) -- (10, 7.9);
			\draw[->, very thick] (10, 8) -- (11, 8);
			
			\draw[->, very thick] (7, 0) -- (7, .9); 
			\draw[->, thick, dotted] (7, 1) -- (7.9, 1); 
			\draw[->, thick, dotted] (8, 1) -- (8, 1.9);
			\draw[->, thick, dotted] (8, 2) -- (8.9, 2);
			\draw[->, thick, dotted] (9, 2) -- (9, 2.9); 
			\draw[->, thick, dotted] (9, 3) -- (9, 3.9); 
			\draw[->, thick, dotted] (9, 4) -- (9.9, 4);
			\draw[->, thick, dotted] (10, 4) -- (10, 4.9);
			\draw[->, thick, dotted] (10, 5) -- (10, 5.9);
			\draw[->, very thick] (10, 6) -- (11, 6);

			\draw[->, very thick] (9, 0) -- (9, .9);
			\draw[->, thick, dotted] (9, 1) -- (9.9, 1);
			\draw[->, thick, dotted] (10, 1) -- (10, 1.9); 
			\draw[->, thick, dotted] (10, 2) -- (10, 2.9); 
			\draw[->, very thick] (10, 3) -- (11, 3);

			\filldraw[fill=gray!50!white, draw=black] (1, 1) circle [radius=.1];
			\filldraw[fill=gray!50!white, draw=black] (1, 2) circle [radius=.1];
			\filldraw[fill=gray!50!white, draw=black] (1, 3) circle [radius=.1];
			\filldraw[fill=gray!50!white, draw=black] (1, 4) circle [radius=.1];
			\filldraw[fill=gray!50!white, draw=black] (1, 5) circle [radius=.1];
			\filldraw[fill=gray!50!white, draw=black] (1, 6) circle [radius=.1];
			\filldraw[fill=gray!50!white, draw=black] (1, 7) circle [radius=.1];
			\filldraw[fill=gray!50!white, draw=black] (1, 8) circle [radius=.1];
			\filldraw[fill=gray!50!white, draw=black] (1, 9) circle [radius=.1];
			\filldraw[fill=gray!50!white, draw=black] (1, 10) circle [radius=.1];
			
			\filldraw[fill=gray!50!white, draw=black] (2, 1) circle [radius=.1];
			\filldraw[fill=gray!50!white, draw=black] (2, 2) circle [radius=.1];
			\filldraw[fill=gray!50!white, draw=black] (2, 3) circle [radius=.1];
			\filldraw[fill=gray!50!white, draw=black] (2, 4) circle [radius=.1];
			\filldraw[fill=gray!50!white, draw=black] (2, 5) circle [radius=.1];
			\filldraw[fill=gray!50!white, draw=black] (2, 6) circle [radius=.1];
			\filldraw[fill=gray!50!white, draw=black] (2, 7) circle [radius=.1];
			\filldraw[fill=gray!50!white, draw=black] (2, 8) circle [radius=.1];
			\filldraw[fill=gray!50!white, draw=black] (2, 9) circle [radius=.1];
			\filldraw[fill=gray!50!white, draw=black] (2, 10) circle [radius=.1];
			
			\filldraw[fill=gray!50!white, draw=black] (3, 1) circle [radius=.1];
			\filldraw[fill=gray!50!white, draw=black] (3, 2) circle [radius=.1];
			\filldraw[fill=gray!50!white, draw=black] (3, 3) circle [radius=.1];
			\filldraw[fill=gray!50!white, draw=black] (3, 4) circle [radius=.1];
			\filldraw[fill=gray!50!white, draw=black] (3, 5) circle [radius=.1];
			\filldraw[fill=gray!50!white, draw=black] (3, 6) circle [radius=.1];
			\filldraw[fill=gray!50!white, draw=black] (3, 7) circle [radius=.1];
			\filldraw[fill=gray!50!white, draw=black] (3, 8) circle [radius=.1];
			\filldraw[fill=gray!50!white, draw=black] (3, 9) circle [radius=.1];
			\filldraw[fill=gray!50!white, draw=black] (3, 10) circle [radius=.1];
			
			\filldraw[fill=gray!50!white, draw=black] (4, 1) circle [radius=.1];
			\filldraw[fill=gray!50!white, draw=black] (4, 2) circle [radius=.1];
			\filldraw[fill=gray!50!white, draw=black] (4, 3) circle [radius=.1];
			\filldraw[fill=gray!50!white, draw=black] (4, 4) circle [radius=.1];
			\filldraw[fill=gray!50!white, draw=black] (4, 5) circle [radius=.1];
			\filldraw[fill=gray!50!white, draw=black] (4, 6) circle [radius=.1];
			\filldraw[fill=gray!50!white, draw=black] (4, 7) circle [radius=.1];
			\filldraw[fill=gray!50!white, draw=black] (4, 8) circle [radius=.1];
			\filldraw[fill=gray!50!white, draw=black] (4, 9) circle [radius=.1];
			\filldraw[fill=gray!50!white, draw=black] (4, 10) circle [radius=.1];
			
			\filldraw[fill=gray!50!white, draw=black] (5, 1) circle [radius=.1];
			\filldraw[fill=gray!50!white, draw=black] (5, 2) circle [radius=.1];
			\filldraw[fill=gray!50!white, draw=black] (5, 3) circle [radius=.1];
			\filldraw[fill=gray!50!white, draw=black] (5, 4) circle [radius=.1];
			\filldraw[fill=gray!50!white, draw=black] (5, 5) circle [radius=.1];
			\filldraw[fill=gray!50!white, draw=black] (5, 6) circle [radius=.1];
			\filldraw[fill=gray!50!white, draw=black] (5, 7) circle [radius=.1];
			\filldraw[fill=gray!50!white, draw=black] (5, 8) circle [radius=.1];
			\filldraw[fill=gray!50!white, draw=black] (5, 9) circle [radius=.1];
			\filldraw[fill=gray!50!white, draw=black] (5, 10) circle [radius=.1];
			
			\filldraw[fill=gray!50!white, draw=black] (6, 1) circle [radius=.1];
			\filldraw[fill=gray!50!white, draw=black] (6, 2) circle [radius=.1];
			\filldraw[fill=gray!50!white, draw=black] (6, 3) circle [radius=.1];
			\filldraw[fill=gray!50!white, draw=black] (6, 4) circle [radius=.1];
			\filldraw[fill=gray!50!white, draw=black] (6, 5) circle [radius=.1];
			\filldraw[fill=gray!50!white, draw=black] (6, 6) circle [radius=.1];
			\filldraw[fill=gray!50!white, draw=black] (6, 7) circle [radius=.1];
			\filldraw[fill=gray!50!white, draw=black] (6, 8) circle [radius=.1];
			\filldraw[fill=gray!50!white, draw=black] (6, 9) circle [radius=.1];
			\filldraw[fill=gray!50!white, draw=black] (6, 10) circle [radius=.1];
			
			\filldraw[fill=gray!50!white, draw=black] (7, 1) circle [radius=.1];
			\filldraw[fill=gray!50!white, draw=black] (7, 2) circle [radius=.1];
			\filldraw[fill=gray!50!white, draw=black] (7, 3) circle [radius=.1];
			\filldraw[fill=gray!50!white, draw=black] (7, 4) circle [radius=.1];
			\filldraw[fill=gray!50!white, draw=black] (7, 5) circle [radius=.1];
			\filldraw[fill=gray!50!white, draw=black] (7, 6) circle [radius=.1];
			\filldraw[fill=gray!50!white, draw=black] (7, 7) circle [radius=.1];
			\filldraw[fill=gray!50!white, draw=black] (7, 8) circle [radius=.1];
			\filldraw[fill=gray!50!white, draw=black] (7, 9) circle [radius=.1];
			\filldraw[fill=gray!50!white, draw=black] (7, 10) circle [radius=.1];
			
			\filldraw[fill=gray!50!white, draw=black] (8, 1) circle [radius=.1];
			\filldraw[fill=gray!50!white, draw=black] (8, 2) circle [radius=.1];
			\filldraw[fill=gray!50!white, draw=black] (8, 3) circle [radius=.1];
			\filldraw[fill=gray!50!white, draw=black] (8, 4) circle [radius=.1];
			\filldraw[fill=gray!50!white, draw=black] (8, 5) circle [radius=.1];
			\filldraw[fill=gray!50!white, draw=black] (8, 6) circle [radius=.1];
			\filldraw[fill=gray!50!white, draw=black] (8, 7) circle [radius=.1];
			\filldraw[fill=gray!50!white, draw=black] (8, 8) circle [radius=.1];
			\filldraw[fill=gray!50!white, draw=black] (8, 9) circle [radius=.1];
			\filldraw[fill=gray!50!white, draw=black] (8, 10) circle [radius=.1];
			
			\filldraw[fill=gray!50!white, draw=black] (9, 1) circle [radius=.1];
			\filldraw[fill=gray!50!white, draw=black] (9, 2) circle [radius=.1];
			\filldraw[fill=gray!50!white, draw=black] (9, 3) circle [radius=.1];
			\filldraw[fill=gray!50!white, draw=black] (9, 4) circle [radius=.1];
			\filldraw[fill=gray!50!white, draw=black] (9, 5) circle [radius=.1];
			\filldraw[fill=gray!50!white, draw=black] (9, 6) circle [radius=.1];
			\filldraw[fill=gray!50!white, draw=black] (9, 7) circle [radius=.1];
			\filldraw[fill=gray!50!white, draw=black] (9, 8) circle [radius=.1];
			\filldraw[fill=gray!50!white, draw=black] (9, 9) circle [radius=.1];
			\filldraw[fill=gray!50!white, draw=black] (9, 10) circle [radius=.1];
			
			\filldraw[fill=gray!50!white, draw=black] (10, 1) circle [radius=.1];
			\filldraw[fill=gray!50!white, draw=black] (10, 2) circle [radius=.1];
			\filldraw[fill=gray!50!white, draw=black] (10, 3) circle [radius=.1];
			\filldraw[fill=gray!50!white, draw=black] (10, 4) circle [radius=.1];
			\filldraw[fill=gray!50!white, draw=black] (10, 5) circle [radius=.1];
			\filldraw[fill=gray!50!white, draw=black] (10, 6) circle [radius=.1];
			\filldraw[fill=gray!50!white, draw=black] (10, 7) circle [radius=.1];
			\filldraw[fill=gray!50!white, draw=black] (10, 8) circle [radius=.1];
			\filldraw[fill=gray!50!white, draw=black] (10, 9) circle [radius=.1];
			\filldraw[fill=gray!50!white, draw=black] (10, 10) circle [radius=.1];

			\filldraw[fill=gray!50!black, draw=black] (1, 11) circle [radius=0] node[scale = 1.4]{$\Lambda$};
			\filldraw[fill=gray!50!black, draw=black] (3, 3.5) circle [radius=0] node[scale = .9]{$\Lambda'$};
			
			\draw[ultra thick, dashed] (2.5, 2.5) -- (8.5, 2.5) -- (8.5, 8.5) -- (2.5, 8.5) -- (2.5, 2.5);

			\draw[] (-.25, 1) -- (-.6, 1) -- (-.6, 2) -- (-.25, 2); 
			\draw[] (-1.15, 1.5) circle [radius = 0] node[]{$W$};
			
			\draw[] (-.25, 9) -- (-.6, 9) -- (-.6, 10) -- (-.25, 10); 
			\draw[] (-1.1, 9.5) circle [radius = 0] node[]{$W$};
			
			\draw[] (-.25, 3) -- (-.6, 3) -- (-.6, 8) -- (-.25, 8); 
			\draw[] (-1.1, 5.5) circle [radius = 0] node[]{$N$};
			
			\draw[] (1, -.25) -- (1, -.6) -- (2, -.6) -- (2, -.25);
			\draw[] (1.5, -1.1) circle [radius = 0] node[]{$W$};
			
			\draw[] (9, -.25) -- (9, -.6) -- (10, -.6) -- (10, -.25);
			\draw[] (9.5, -1.1) circle [radius = 0] node[]{$W$};
			
			\draw[] (3, -.25) -- (3, -.6) -- (8, -.6) -- (8, -.25);
			\draw[] (5.5, -1.1) circle [radius = 0] node[]{$N$};

			\draw[] (17, .5) -- (27, .5) -- (27, 10.5) -- (17, 10.5) -- (17, .5);
			\draw[] (19, 2.5) -- (25, 2.5) -- (25, 8.5) -- (19, 8.5) -- (19, 2.5);
			\draw[dashed] (19, .5) -- (19, 2.5);
			\draw[dashed] (17, 8.5) -- (19, 8.5);
			\draw[dashed] (25, 2.5) -- (27, 2.5);
			\draw[dashed] (25, 8.5) -- (25, 10.5);
			
			\draw[very thick, dashed] (19, .5) -- (19, 1.25);
			\draw[very thick, dashed] (17, 8.5) -- (17.75, 8.5);
			\draw[very thick, dashed] (25, 2.5) -- (25.75, 2.5);
			\draw[very thick, dashed] (25, 8.5) -- (25, 9.25);

			\draw[] (22, 5.5) circle[radius = 0] node[scale = 1.6]{$\Lambda'$};
			\draw[] (18, 4.5) circle[radius = 0] node[scale = 1.3]{$\Gamma_1$};
			\draw[] (23, 1.5) circle[radius = 0] node[scale = 1.3]{$\Gamma_2$};
			\draw[] (26, 6.5) circle[radius = 0] node[scale = 1.3]{$\Gamma_3$};
			\draw[] (21, 9.5) circle[radius = 0] node[scale = 1.3]{$\Gamma_4$};
			\draw[] (27, -.25) circle[radius = 0] node[scale = 1.9]{$\Lambda$};

			\draw[] (17, 4.5) circle[radius = 0] node[scale = .8, left]{$A_1$};
			\draw[] (18, .5) circle[radius = 0] node[scale = .8, below]{$B_1$};
			\draw[] (17, 9.5) circle[radius = 0] node[scale = .8, left]{$A_4$};
			\draw[] (21, 10.5) circle[radius = 0] node[scale = .8, above]{$D_4$};
			\draw[] (26, 10.5) circle[radius = 0] node[scale = .8, above]{$D_3$};
			\draw[] (27, 6.5) circle[radius = 0] node[scale = .8, right]{$C_3$};
			\draw[] (27, 1.5) circle[radius = 0] node[scale = .8, right]{$C_2$};
			\draw[] (23, .5) circle[radius = 0] node[scale = .8, below]{$B_2$};
			
			\draw[] (19, 5.5) circle[radius = 0] node[scale = .8, right]{$A'$};
			\draw[] (22, 2.5) circle[radius = 0] node[scale = .8, above]{$B' + 1$};
			\draw[] (25, 5.5) circle[radius = 0] node[scale = .8, left]{$C'$};
			\draw[] (22, 8.5) circle[radius = 0] node[scale = .8, below]{$D' + 1$};
			
			\draw[] (18, 8.5) circle[radius = 0] node[scale = .8, above]{$K_1$};
			\draw[] (19, 1.5) circle[radius = 0] node[scale = .8, right]{$K_2$};
			\draw[] (26, 2.5) circle[radius = 0] node[scale = .8, above]{$K_3$};
			\draw[] (25, 9.5) circle[radius = 0] node[scale = .8, right]{$K_4$};

			\end{tikzpicture}
			
		\end{center}	
		
		\caption{\label{lambdalambdae} To the left is a depiction of \Cref{ensemblee0e}. To the right are the domains $\Gamma_1$, $\Gamma_2$, $\Gamma_3$, and $\Gamma_4$ used in its proof, where the numbers along the boundaries there indicate how many paths enter or exit through those boundaries.  }
	\end{figure}
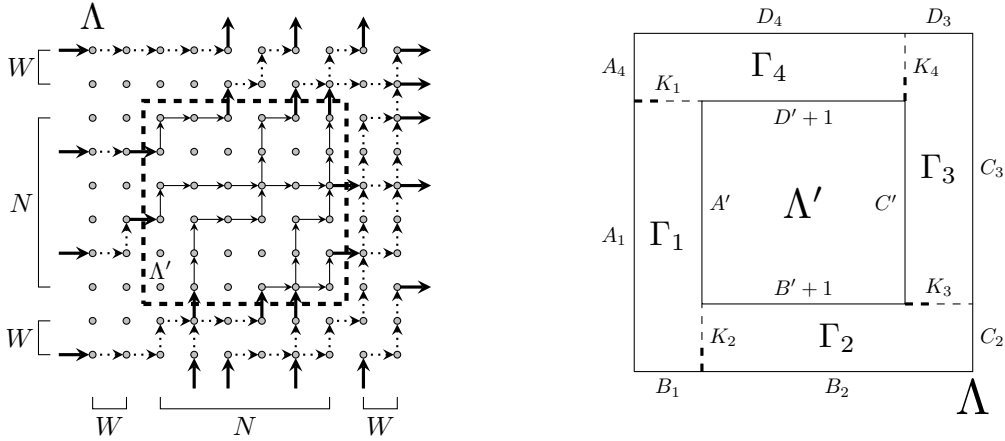

		\begin{proof}
		
		Let us partition $\Lambda \setminus \Lambda'$ into four subdomains $\Lambda \setminus \Lambda' = \Gamma_1 \cup \Gamma_2 \cup \Gamma_3 \cup \Gamma_4$ by setting 
		\begin{flalign*}
		& \Gamma_1 = [1, W] \times [1, N + W]; \qquad \qquad \qquad \qquad \qquad \qquad \Gamma_2 = [W + 1, N + 2W] \times [1, W]; \\
		& \Gamma_3 = [N + W + 1, N + 2W] \times [W + 1, N + 2W]; \qquad \Gamma_4 = [1, N + W] \times [N + W + 1, N + 2W].
		\end{flalign*} 
		
		We refer to the right side of \Cref{lambdalambdae} for a depiction. Next, we will define boundary data on the $\Gamma_i$ in such a way that they each admit a six-vertex ensemble $\mathcal{E}_i$; are consistent with each other; and are consistent with $\textbf{u} \cup \textbf{v}$ and $\textbf{u}' \cup \textbf{v}'$. Then, $\mathcal{E} \in \mathfrak{E}_{\textbf{u}; \textbf{v}} (\Lambda)$ will be formed by the union $\mathcal{F}' \cup \bigcup_{i = 1}^4 \mathcal{E}_i$. To implement this, we first require some notation. 
		
		Denote $\textbf{u} = (u_{-B}, u_{1 - B}, \ldots , u_A)$, $\textbf{v} = (v_{-B}, v_{1 - B}, \ldots , v_A)$, $\textbf{u}' = (u_{-B'}', u_{1 - B'}', \ldots , u_{A'}')$, and $\textbf{v}' = (v_{-B'}', v_{1 - B'}', \ldots , v_{A'}')$. Also let $C$ and $D + 1$ denote the numbers of vertices in $\textbf{v}$ on the east and north boundaries of $\Lambda$, respectively. In this way, $A + B + 1 = C + D + 1$ both denote the number of vertices in $\textbf{v}$, and $v_{C - B - 1}$ lies on the east boundary of $\Lambda$, but $v_{C - B}$ lies on its north boundary. Similarly, let $C'$ and $D' + 1$ denote the numbers of vertices in $\textbf{v}'$ on the east and north boundaries of $\Lambda'$, respectively. 
		
		Moreover, for each index $i \in \{ 1, 4 \}$, let $A_i$ denote the number of vertices in $\textbf{u}$ on the west boundary of $\Gamma_i$. Similarly, for each $i \in \{ 1, 2 \}$, let $B_i$ denote the number of vertices in $\textbf{u}$ on the south boundary of $\Gamma_i$; for each $i \in \{ 2, 3 \}$, let $C_i$ denote the number of vertices in $\textbf{v}$ on the east boundary of $\Gamma_i$; and for each $i \in \{ 3, 4 \}$, let $D_i$ denote the number of vertices in $\textbf{u}$ on the north boundary of $\Gamma_i$. For example, under this notation, we have that $\textbf{u} \cap \partial \Gamma_1 = (u_{1 - B_1}, u_{2 - B_1}, \ldots , u_{A_1})$, $A = A_1 + A_4$, and $B + 1 = B_1 + B_2$. We refer to the right side of \Cref{lambdalambdae} for a depiction. 
		
		Next, for each $i \in \{ 1, 2, 3, 4 \}$, we define an integer $K_i$ that will denote the number of paths passing between $\Gamma_{i - 1}$ and $\Gamma_i$, where we let $\Gamma_0 = \Gamma_4$. To that end, we first set $K_1 = \lfloor s W \rfloor$ and then define $K_2, K_3, K_4$ by the relations 
		\begin{flalign}
		\label{k123}
		\begin{aligned} 
		& 	A_1 + B_1 = K_1 + A' + K_2; \qquad K_2 + B_2 = C_2 + K_3 + B' + 1; \\
		& K_3 + C' + K_4 = C_3 + D_3; \qquad A_4 + K_1 + D' + 1 = K_4 + D_4.
		\end{aligned}
		\end{flalign}
		
		\noindent Observe that the fourth equality in \eqref{k123} is a consequence of the first three, together with the facts that $A' + B' + 1 = C' + D' + 1$ (both equal the number of paths passing through $\Lambda'$) and $A_1 + A_4 + B_1 + B_2 = A + B + 1 = C + D + 1 = C_2 + C_3 + D_3 + D_4$ (all equal the number of paths in $\Lambda$). Under the interpretation for $K_i$ as the number of paths passing between $\Gamma_{i - 1}$ and $\Gamma_i$, the four equations \eqref{k123} indicate that the same number of paths enter $\Gamma_i$ as exit $\Gamma_i$, for each $i \in \{ 1, 2, 3, 4 \}$; we again refer to the right side of \Cref{lambdalambdae} for a depiction. 
		
		Now, we claim that 
		\begin{flalign}
		\label{k1234estimate} 
		sW - 40 \eta N \le K_1, K_3 \le sW + 40 \eta N; \qquad tW - 40 \eta N \le K_2, K_4 \le tW + 40 \eta N,
		\end{flalign}
		
		\noindent which in particular implies that each of the $K_i$ are positive by \eqref{nwr}. To establish \eqref{k1234estimate}, first observe from the $(R; \eta)$-regularity of the boundary data $\textbf{u} \cup \textbf{v}$ and $\textbf{u}' \cup \textbf{v}'$; the bound $R \le W \le N$ from \eqref{nwr}; and \Cref{uetareta} that 
		\begin{flalign}
		\label{abcd1234estimate} 
		\begin{aligned} 
		& (t - 2 \eta) (N + W)  \le A_1, C_3 \le (t + 2 \eta) (N + W); \qquad (t - 2 \eta) W \le A_4, C_2 \le (t + 2 \eta) W;  \\
		& (s - 2 \eta) (N + W) \le B_2, D_4 \le (s + 2 \eta)(N + W); \qquad (s - 2 \eta) W \le B_1, D_3 \le (s + 2 \eta) W,
		\end{aligned}
		\end{flalign}
		
		\noindent and
		\begin{flalign}
		\label{abcdestimate} 
		(t - 2 \eta) N \le A', C' \le (t + 2 \eta) N; \qquad (s - 2 \eta) N \le B' + 1, D' + 1 \le (s + 2 \eta) N.
		\end{flalign}
		
		\noindent Since $W \le N$ and $K_1 = \lfloor sN \rfloor$, we deduce \eqref{k1234estimate} from inserting \eqref{abcd1234estimate} and \eqref{abcdestimate} into \eqref{k123}.
		
		Now, let us define boundary data $\textbf{u}^{(i)} \cup \textbf{v}^{(i)}$ on each $\Gamma_i$, which will be obtained as the union of the boundary data induced on $\Gamma_i$ by $\textbf{u} \cup \textbf{v}$ and $\textbf{u}' \cup \textbf{v}'$ with another set of $K_i + K_{i + 1}$ vertices $\mathcal{K}^{(i)} \subseteq (\Gamma_{i - 1} \cup \Gamma_{i + 1}) \cap \partial \Gamma_i$ (where $\Gamma_5 = \Gamma_1$). In particular, let $\mathcal{B} = \textbf{u} \cup \textbf{u}' \cup \textbf{v} \cup \textbf{v}'$, and set $\textbf{u}^{(i)} \cup \textbf{v}^{(i)} = (\mathcal{B} \cap \partial \Gamma_i) \cup \mathcal{K}^{(i)}$, where $\mathcal{K}^{(i)} = \mathcal{K}^{(i; 1)} \cup \mathcal{K}^{(i; 2)}$, and $\mathcal{K}^{(i; 1)}$ and $\mathcal{K}^{(i; 2)}$ denote the $K_i$ and $K_{i + 1}$ ``most south'' or ``most west'' vertices in $\Gamma_{i - 1} \cap \partial \Gamma_i$ and $\Gamma_{i + 1} \cap \partial \Gamma_i$, respectively (see the thick parts of the boundaries between the $\Gamma_i$ on the right side of \Cref{lambdalambdae}). More specifically, we set 
		\begin{flalign*}
		&\mathcal{K}^{(1; 1)} = \bigcup_{x = 1}^{K_1} \big\{ (x, N + W + 1) \big\}; \qquad \qquad \qquad \mathcal{K}^{(1; 2)} = \bigcup_{y = 1}^{K_2} \big\{ (W + 1, y) \big\}; \\
		& \mathcal{K}^{(2; 1)} = \bigcup_{y = 1}^{K_2} \big\{ (W, y) \big\}; \qquad	\qquad \qquad \qquad \qquad \mathcal{K}^{(2; 2)} = \bigcup_{x = 1}^{K_3} \big\{ (N + W + x, W + 1) \big\}; \\
		& \mathcal{K}^{(3; 1)} = \bigcup_{x = 1}^{K_3} \big\{ (N + W + x, W) \big\}; \qquad \qquad \qquad \mathcal{K}^{(3; 2)} = \bigcup_{y = 1}^{K_4} \big\{ (N + W, N + W + y) \big\}; \\
		& \mathcal{K}^{(4; 1)} = \bigcup_{y = 1}^{K_4} \big\{ (N + W + 1, N + W + y) \big\}; \qquad \mathcal{K}^{(4; 2)} = \bigcup_{x = 1}^{K_1} \big\{ (x, N + W) \big\}.
		\end{flalign*} 
		
	 	Setting $\textbf{u}^{(i)} \cup \textbf{v}^{(i)} = (\mathcal{B} \cap \partial \Gamma_i) \cup \mathcal{K}^{(i; 1)}  \cup \mathcal{K}^{(i; 2)}$, let us show $\mathfrak{E}_{\textbf{u}^{(i)}; \textbf{v}^{(i)}} (\Gamma_i)$ is nonempty for every $i \in \{ 1, 2, 3, 4 \}$. As the proof for each $i$ is entirely analogous, we only address the case $i = 1$. To that end, we use the following fact, which is directly verified by induction on $m + n + 1$. Let $\Gamma \subset \mathbb{Z}^2$ denote a rectangular domain, and let $\textbf{j} \cup \textbf{k}$ be some boundary data on $\Gamma$. Denoting $\textbf{j} = (j_{-n}, j_{1 - n}, \ldots , j_m)$ and $\textbf{k} = (k_{-n}, k_{1 - n}, \ldots , k_m)$, the set $\mathfrak{E}_{\textbf{j}; \textbf{k}} (\Gamma)$ is nonempty if and only if $k_i \ge j_i$ for each $i \in [-n, m]$ (where we recall from \Cref{Paths} that $(x_1, y_1) \ge (x_2, y_2)$ if $x_1 \ge y_1$ and $x_2 \ge y_2$). 
		
		So, setting $\textbf{u}^{(1)} = \big( u_{1 - B_1}^{(1)}, u_{2 - B_1}^{(1)}, \ldots , u_{A_1}^{(1)}\big)$ and $\textbf{v}^{(1)} = \big( v_{1 - B_1}^{(1)}, v_{2 - B_1}^{(1)}, \ldots , v_{A_1}^{(1)} \big)$, it suffices to show that $u_i^{(1)} \le v_i^{(1)}$ for each $i \in [1 - B_1, A_1]$. This holds if $i \in [1 - B_1, 0]$, since then $u_i^{(i)}$ lies on the south boundary of $\Gamma_1$, and if $i \in [A_1 - K_1 + 1, A_1]$, since then $v_i^{(1)}$ lies on the north boundary of $\Gamma_1$. So, we may suppose that $i \in [1, A_1 - K_1]$, in which case we denote $u_i^{(1)} = (0, w_i)$ and $v_i = (W + 1, y_i)$; it suffices to show that $w_i \le y_i$. 
		
		Letting $j = A_1 - i$, the $(R; \eta)$-regularity of the boundary data $\textbf{u} \cup \textbf{v}$ and $\textbf{u}' \cup \textbf{v}'$ (and \Cref{uetareta}) implies that 
		\begin{flalign*} 
		w_i = w_{A_1 - j} \le N - s^{-1} (j - 2 \eta j - R); \qquad y_i = y_{A_1 - j} \ge N - s^{-1} (j - K_1 + 2 \eta j + R).
		\end{flalign*}
		
		\noindent Thus, the bound $w_i \le y_i$ follows from \eqref{k1234estimate} and the fact that $sW - 40 \eta N \ge 8 \eta N + 2R \ge 4 \eta (N + W) + 2R \ge 4 \eta j + 2R$ (which holds by \eqref{nwr}). Hence, each $u_i^{(1)} \le v_i^{(1)}$, so $\mathfrak{E}_{\textbf{u}^{(1)}; \textbf{v}^{(1)}} (\Gamma_1)$ is nonempty and therefore contains a six-vertex ensemble $\mathcal{E}_1$. 
		
		Similarly, for each $i \in \{ 2, 3, 4 \}$, the set $\mathfrak{E}_{\textbf{u}^{(i)}; \textbf{v}^{(i)}} (\Gamma_i)$ is nonempty and contains some six-vertex ensemble $\mathcal{E}_i$. Letting $\mathcal{E} = \mathcal{F}' \cup \bigcup_{ = 1}^4 \mathcal{E}_i$, we have $\mathcal{E} \in \mathfrak{E}_{\textbf{u}; \textbf{v}} (\Lambda)$ and $\mathcal{E}_{\Lambda'} = \mathcal{F}'$, since the boundary data for the $\mathcal{E}_i$ are consistent with each other and with $\textbf{u} \cup \textbf{v} \cup \textbf{u}' \cup \textbf{v}'$; this yields the lemma.
	\end{proof}

	\subsection{Proof of \Cref{mulowerprobability}} 
	
	\label{Lower1} 
	
	In this section we establish \Cref{mulowerprobability}.

	\begin{proof}[Proof of \Cref{mulowerprobability}]
		
		Fix a real number $\delta \in \big( 0, 1 \big)$, and let $N > 0$ denote an integer, which we will chose to be sufficiently large below. Define the real number $\eta \in (0, 1)$ and integer $R > 0$ by
		\begin{flalign}
		\label{etar}
		\eta = st B_1 B_2 (1 - B_1) (1 - B_2) \frac{ \delta}{650}; \qquad R = \bigg\lfloor \displaystyle\frac{\eta N}{3} \bigg\rfloor.
		\end{flalign}
		
		\noindent Moreover, recall from \Cref{estimateprobabilitylower} the domain $\Lambda = \Lambda_N = [1, N] \times [1, N] \subseteq \mathbb{Z}^2$; the marginal distribution $\nu = \nu_N \in \mathscr{P} \big( \mathfrak{E} (\mathbb{Z}^2 \setminus \Lambda_N )\big)$ of $\mu$; the boundary data $\textbf{u} (\mathcal{H}) \cup \textbf{v} (\mathcal{H})$ on $\Lambda$ induced by any six-vertex ensemble $\mathcal{H} \in \mathfrak{E} (\mathbb{Z}^2 \setminus \Lambda)$; and the partition function $Z (\mathcal{H})$ for the $(B_1, B_2)$-stochastic six-vertex model with boundary data $\textbf{u} (\mathcal{H}) \cup \textbf{v} (\mathcal{H})$. 
		
		Define the event $\mathcal{A}$ on which $\textbf{u} (\mathcal{H}) \cup \textbf{v} (\mathcal{H})$ is $(R; \eta)$-regular with slope $(s, t)$. Since \Cref{rhosmu} implies for sufficiently large $N$ that $\mathbb{P}_{\nu} [\mathcal{A}] = \mathbb{P}_{\mu} [\mathcal{A}] \ge 1 - \frac{\eta}{3} \ge 1 - \delta$, it suffices to show that $Z (\mathcal{H}) \ge e^{- \delta N^2}$ holds on $\mathcal{A}$, if $N$ is sufficiently large. 
		
		Let us briefly outline how we will do this. We will first define a square subdomain of the form $\Lambda' = [W + 1, N - W] \times [W + 1, N - 1] \subset \Lambda$ occupying ``most'' of $\Lambda$. It will then suffice to lower bound the sum of $w(\mathcal{E}')$ over all $\mathcal{E}' \in \mathfrak{E}(\Lambda')$ whose boundary data has approximate slope $(s, t)$. Indeed, since \Cref{ensemblee0e} implies that each such $\mathcal{E}'$ admits an extension $\mathcal{E}$ to $\Lambda$ with boundary data $\textbf{u} (\mathcal{H}) \cup \textbf{v} (\mathcal{H})$, this would yield an estimate on the sum over all such $w (\mathcal{E})$ and therefore on $Z (\mathcal{H})$. To establish the former lower bound, we will define a pair $(s_0, t_0) \in (0, 1]^2$ such that $t_0 = \varphi (s_0)$ and consider the $(B_1, B_2)$-stochastic six-vertex model on $\Lambda'$ with double-sided $(t_0, s_0)$-Bernoulli entrance data. Conditioning on this entry data, the weight sum of all six-vertex ensembles with this entry data is equal to $1$. However, the dominant contribution to this sum arises from ensembles whose boundary data has approximate slope $(s_0, t_0)$ and not $(s, t)$. So, we will consider the weight sum of the $(L; K)$-restrictions these ensembles, whose boundary data will have approximate slope $(s, t)$ if $K$ and $L$ are appropriately chosen. Then the required lower bound will follow from \Cref{rholrho} estimating the weight sum of the latter, restricted ensembles in terms of that of the original, unrestricted ones.
		
		To implement this procedure, we begin by introducing the quantities $(s_0, t_0)$, $K$, $L$, and $W$ used there. So, recalling $\kappa$ from \eqref{kappafunction}, define the pair $(s_0, t_0)$ by setting 
		\begin{flalign} 
		\label{ws0t0}
		 s_0 = \displaystyle\frac{\kappa s - t}{(\kappa - 1)t}; \qquad t_0 = \displaystyle\frac{\kappa s - t}{(\kappa - 1) s}.
		\end{flalign}
		
		\noindent In this way, $(s_0, t_0)$ denotes the point where the line $\big\{ y = \frac{t x}{s} \big\}$ (passing through $(0, 0)$ and $(s, t)$) intersects the curve $\big\{ (x, y): y = \varphi (x) \big\} \cap \mathbb{R}_{> 0}^2$; we refer to \Cref{s0t0} for a depiction. Indeed, the facts that $\frac{t_0}{s_0} = \frac{t}{s}$ and $t_0 = \varphi (s_0)$ follow from the definitions \eqref{kappafunction} and \eqref{ws0t0} of $\varphi$ and $(s_0, t_0)$, respectively. Denoting $\vartheta = \frac{s_0}{s}$, we moreover have that 
		\begin{flalign}
		\label{s0t01ss0tt0}
		0 < s_0 \le t_0 \le 1; \qquad 0 < \frac{s}{s_0} = \vartheta = \frac{t}{t_0} = \frac{(\kappa - 1) st}{\kappa s - t} \le 1.
		\end{flalign}

		\begin{figure}
			
			\begin{center}
			
			\begin{tikzpicture}
			\begin{axis}[
			scale = .8,
			samples= 300, 
			xlabel = {$x$},
			ylabel = {$y$},
			ylabel style={rotate=270},
			xmin=0, 
			xmax=1, 
			ymin=0, 
			ymax=1, 
			xtick distance=1,
			ytick distance=1
			]
			\addplot[no marks, very thick, color=black] {4*x/(3*x+1)};
			\addplot[no marks, thick, color=black] {x};
			\addplot[no marks, thick, color=black] coordinates {(.5, .5)} node[right = 4, scale = .8]{$y = x$};
			\addplot[no marks, thick, color=black] coordinates {(.28, .65)} node[left = -1, scale = .8]{$y = \varphi (x)$};
			\addplot[mark=*,only marks] coordinates {(.5, .8)} node[below = 5, right, scale = .8]{$(s_0, t_0)$};
			\addplot[mark=*,only marks] coordinates {(.25, .4)} node[right, scale = .8]{$(s, t)$};
			\addplot[domain=0:.5, dashed] {1.6*x};
			\end{axis}
				
			\end{tikzpicture}

			\caption{\label{s0t0} Depicted above are $(s, t)$ and $(s_0, t_0)$.}
			
			\end{center}
		
		\end{figure}
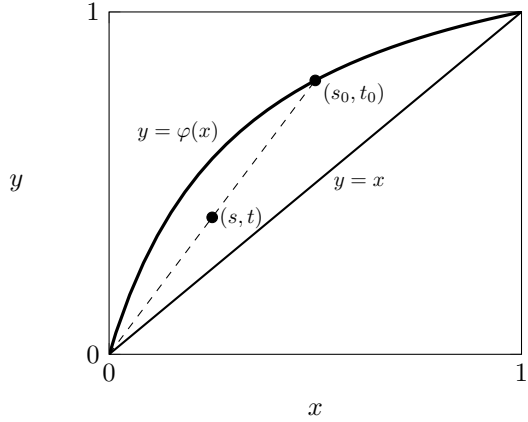

		\noindent To verify these, observe since $t \le \varphi (s)$ that 
		\begin{flalign}
		\label{kappa1st} 
		0 < (\kappa - 1 ) st \le \kappa s - t.
		\end{flalign}
		
		\noindent This implies the first bound $s_0 > 0$ in the first statement of \eqref{s0t01ss0tt0}. The bounds $s_0, t_0 \le 1$ in that statement follow from \eqref{ws0t0} and the fact that $t \ge s$, and then $s_0 \le t_0$  holds since $\varphi (z) \ge z$ for each $z \in (0, 1]$. The second statement of \eqref{s0t01ss0tt0} again follows from \eqref{kappa1st}. 
		
		Next, define the integers $K, L, W, M > 0$ by setting 
		\begin{flalign}
		\label{omegakl} 
		\qquad K = \bigg\lfloor \displaystyle\frac{s_0^2 t_0^2 \eta R}{16} \bigg\rfloor; \qquad L = \lceil \vartheta K \rceil; \qquad W = \Bigg\lceil 60 \bigg( \displaystyle\frac{1}{s} + \displaystyle\frac{1}{t} \bigg) \eta N \Bigg\rceil; \qquad M = \bigg\lceil \displaystyle\frac{N - 2W}{K} \bigg\rceil.
		\end{flalign}
		
		\noindent Further define the domain $\Lambda' = [W + 1, N - W] \times [W + 1, N - W] \subseteq \Lambda$. We will show that there exists boundary data $\textbf{x}' \cup \textbf{y}'$ on $\Lambda'$ that is $(R; \eta)$-regular with slope $(s, t)$ such that
		\begin{flalign}
		\label{weightesumxye1} 
		\displaystyle\sum_{\mathcal{E}' \in \mathfrak{E}_{\textbf{x}', \textbf{y}'} (\Lambda')} w (\mathcal{E}') \ge \big( B_1 B_2 (1 - B_1) (1 - B_2) \big)^{4MN + 4N}.
		\end{flalign} 
		
		Let us establish the proposition assuming \eqref{weightesumxye1}. To that end, first observe for sufficiently large $N$ that $\min \{ sW, tW \} \ge 50 \eta (N - 2W)$, $R \le \eta (N - 2W)$, and $50 (s^{-1} + t^{-1}) R \le W \le N - 2W$, due to the choices \eqref{omegakl} of $W$ and \eqref{etar} of $\eta < \frac{st}{650}$ and $R$. Thus, since $\textbf{x}' \cup \textbf{y}'$ and $\textbf{u} (\mathcal{H}) \cup \textbf{v} (\mathcal{H})$ are both $(R; \eta)$-regular with slope $(s, t)$ on the event $\mathcal{A}$, \Cref{ensemblee0e} implies for each $\mathcal{E}' \in \mathfrak{E}_{\textbf{x}'; \textbf{y}'} (\Lambda')$ the existence of some $\mathcal{E} \in \mathfrak{E}_{\textbf{u} (\mathcal{H}) \cup \textbf{v} (\mathcal{H})} (\Lambda)$ such that $\mathcal{E}_{\Lambda'} = \mathcal{E}'$. Then, since $|\Lambda \setminus \Lambda'| \le 4WN$ and the weight of any vertex under the $(B_1, B_2)$-stochastic six-vertex model is at most $B_1 B_2 (1 - B_1) (1 - B_2)$, we have
		\begin{flalign*}
		w (\mathcal{E}) \ge \big( B_1 B_2 (1 - B_1) (1 - B_2) \big)^{4WN} w (\mathcal{E}').
		\end{flalign*}		
		
		This, together with \eqref{weightesumxye1}, implies on $\mathcal{A}$ that
		\begin{flalign}
		\label{zf1estimate} 
		\begin{aligned}
		Z (\mathcal{H}) = \displaystyle\sum_{\mathcal{E} \in \mathfrak{E}_{\textbf{u} (\mathcal{H}); \textbf{v} (\mathcal{H})}} w (\mathcal{E}) & \ge \big( B_1 B_2 (1 - B_1) (1 - B_2) \big)^{4WN} 
		\displaystyle\sum_{\mathcal{E}' \in \mathfrak{E}_{\textbf{x}'; \textbf{y}'} (\Lambda')} w (\mathcal{E}') \\
		& \ge \big( B_1 B_2 (1 - B_1) (1 - B_2) \big)^{8 (M + W) N}.
		\end{aligned}
		\end{flalign}
		
		\noindent Then, since  \eqref{etar} and \eqref{omegakl} together yield
		\begin{flalign*}
		W \le \displaystyle\frac{65 \eta N}{st} < B_1 B_2 (1 - B_1) (1 - B_2) \displaystyle\frac{\delta N}{10}; \qquad M \le \displaystyle\frac{N}{K} \le \displaystyle\frac{96}{(s_0 t_0 \eta)^2},
		\end{flalign*}
		
		\noindent it follows from \eqref{zf1estimate} that
		\begin{flalign*}
		Z (\mathcal{H}) \ge \big( B_1 B_2 (1 - B_1) (1 - B_2) \big)^{8 (M + W) N} \ge e^{-\delta N^2},
		\end{flalign*}
		
		\noindent holds for sufficiently large $N$ on $\mathcal{A}$, which implies the proposition.
		
		Hence, it suffices to verify the existence of $(R; \eta)$-regular boundary data $\textbf{x}' \cup \textbf{y}'$ on $\Lambda'$ such that \eqref{weightesumxye1} holds. To that end, for any real numbers $R_0 > 1$ and $\eta_0 \in (0, 1)$; pair $(S, T) \in (0, 1]^2$; and rectangular domain $\Gamma \subseteq \mathbb{Z}^2$, let $\mathfrak{R} (R_0, \eta_0; S, T; \Gamma) \subseteq \mathfrak{E}(\Gamma)$ denote the set of six-vertex ensembles on $\Gamma$ whose boundary data is $(R_0; \eta_0)$-regular with slope $(S, T)$. Additionally, for any entrance data $\textbf{w}$ on $\Gamma$, let $\mathfrak{R}_{\textbf{w}} (R_0, \eta_0; S, T; \Gamma) = \mathfrak{R}(R_0, \eta_0; S, T; \Gamma) \cap \mathfrak{E}_{\textbf{w}} (\Gamma)$ denote the set of six-vertex ensembles in $\mathfrak{R} (R_0, \eta_0, S, T; \Gamma)$ with entrance data given by $\textbf{w}$. 
		
		Now, sample a random six-vertex ensemble $\mathcal{G} \in \mathfrak{E} (\Lambda')$ on $\Lambda'$ under a $(B_1, B_2)$-stochastic six-vertex model $\mathfrak{S}$ with $(s_0, t_0)$-Bernoulli entrance data; denote the (random) entrance data for $\mathcal{G}$ by $\textbf{z}$. Then the fact that $t_0 = \varphi (s_0)$, \Cref{murho1}, and \Cref{rhosmu} together imply for sufficiently large $N$ that 
		\begin{flalign*} 
		\mathbb{P}_{\mathfrak{S}} \Bigg[ \mathcal{G}\in \mathfrak{P} \bigg( L, \displaystyle\frac{s_0 t_0 \eta}{8}; s_0, t_0; \Lambda' \bigg) \Bigg] \ge \displaystyle\frac{1}{2}. 
		\end{flalign*} 
		
		\noindent Hence, 
		\begin{flalign*}
		\mathbb{E} \Bigg[ \displaystyle\sum_{\mathcal{E} \in \mathfrak{P}_{\textbf{z}} (L, s_0 t_0 \eta / 8; s_0, t_0; \Lambda')} w (\mathcal{E}) \Bigg] = \mathbb{P}_{\mathfrak{S}} \Bigg[ \mathcal{G}\in \mathfrak{P} \bigg( L, \displaystyle\frac{s_0 t_0 \eta}{8}; s_0, t_0; \Lambda' \bigg) \Bigg] \ge \displaystyle\frac{1}{2}.
		\end{flalign*} 
		
		\noindent where the expectation on the left side is with respect to the entrance data $\textbf{z}$ for $\mathcal{G}$. In particular, there exists some (deterministic) choice $\textbf{x}$ of entrance data on $\Lambda'$ such that 
		\begin{flalign}
		\label{weightesumx}
		\displaystyle\sum_{\mathcal{E} \in \mathfrak{P}_{\textbf{x}} (L, s_0 t_0 \eta / 8; s_0, t_0; \Lambda')} w (\mathcal{E}) \ge \displaystyle\frac{1}{2}.
		\end{flalign} 
		
		\noindent Moreover, since $|\partial \Lambda'| \le 4 (N - W) \ge 4N - 4$, there are at most $2^{4N - 4}$ possible choices of exit data for any $\mathcal{E} \in \mathfrak{P}_{\textbf{x}} \big( L, \frac{s_0 t_0 \eta}{8}; s_0, t_0; \Lambda' \big)$. Hence, it follows from \eqref{weightesumx} that there exists some (deterministic) exit data $\textbf{y}$ on $\Lambda'$ such that the boundary data $\textbf{x} \cup \textbf{y}$ is $\big( L; \frac{s_0 t_0 \eta}{8} \big)$-regular with slope $(s_0, t_0)$ and 
		\begin{flalign}
		\label{weightesumxy} 
		\displaystyle\sum_{\mathcal{E} \in \mathfrak{E}_{\textbf{x}, \textbf{y}} (\Lambda')} w (\mathcal{E}) \ge 2^{-4N}.
		\end{flalign} 
		
		Now let $\textbf{x}' \cup \textbf{y}'$ denote the $(L; K)$-restriction of $\textbf{x} \cup \textbf{y}$; we claim that $\textbf{x}' \cup \textbf{y}'$ is $(R; \eta)$-regular with slope $(s, t)$. To see this, observe from \eqref{omegakl} and the identity $\frac{s}{s_0} = \vartheta = \frac{t}{t_0}$ that $\big| \frac{s_0 L}{K} - s \big| < \frac{s_0 t_0 \eta}{8}$ and $\big| \frac{t_0 L}{K} - t \big| < \frac{s_0 t_0 \eta}{8}$ hold for sufficiently large $N$. Thus, since $\textbf{x} \cup \textbf{y}$ is $\big( L; \frac{s_0 t_0 \eta}{8} \big)$-regular, \Cref{restrictionregular} applies (whose $\omega$ and $\eta$ equal to $\frac{s_0 t_0 \eta}{8}$ here) and implies $\textbf{x}' \cup \textbf{y}'$ is $(R; \eta)$-regular with slope $(s, t)$. 
		
		Thus, it suffices to verify \eqref{weightesumxye1}. To that end, observe by summing \Cref{rholrho} that 
		\begin{flalign*}
		\displaystyle\sum_{\mathcal{E}' \in \mathfrak{E}_{\textbf{x}', \textbf{y}'} (\Lambda')} w (\mathcal{E}') \ge \big( (1 - B_1) (1 - B_2) \big)^{4MN} \displaystyle\sum_{\mathcal{E} \in \mathfrak{E}_{\textbf{x}, \textbf{y}} (\Lambda')} w (\mathcal{E}),
		\end{flalign*}
		
		\noindent and so \eqref{weightesumxye1} follows from \eqref{weightesumxy} and the fact that $B_1 B_2 (1 - B_1) (1 - B_2) \le \frac{1}{2}$. 
	\end{proof}

\end{document}